\documentclass[a4paper,12pt,oneside,reqno]{amsart}
\usepackage{hyperref}
\usepackage[headinclude,DIV13]{typearea}
\areaset{15.1cm}{25.0cm}
\usepackage{amsfonts,amssymb,amsmath,amsthm,bbm}
\usepackage[utf8]{inputenc}
\usepackage{graphicx, psfrag}

\newtheorem{theorem}{Theorem}[section]
\newtheorem{lemma}[theorem]{Lemma}
\newtheorem{proposition}[theorem]{Proposition}
\newtheorem{corollary}[theorem]{Corollary}

\theoremstyle{definition}
\newtheorem{definition}[theorem]{Definition}

\theoremstyle{remark}
\newtheorem*{remark}{Remark}

\def\paragraph#1{\noindent \textbf{#1}}

\numberwithin{equation}{section}

\def\sign{\mathop{\rm sign}\nolimits}
\def\d{\mathrm{d}}

\def\<{\langle}
\def\>{\rangle}
\def\a{\alpha}
\def\b{\beta}
\def\e{\epsilon}

\def\g{\gamma}

\def\l{\lambda}

\def\s{\sigma}
\def\t{\tau}
\def\th{\theta}

\def\o{\omega}

\def\L{\Lambda}
\def\G{\Gamma}
\def\O{\Omega}

\def\del{\partial}
\def\R{{\Bbb R}}  
\def\N{{\Bbb N}}  
\def\P{{\Bbb P}}  
\def\C{{\Bbb C}}  
\def\E{{\Bbb E}}  
\def\U{{\Bbb U}}
\def\Va{{\Bbb V}}  

\let\cal=\mathcal
\def\AA{{\cal A}}
\def\BB{{\cal B}}

\def\DD{{\cal D}}
\def\EE{{\cal E}}
\def\FF{{\cal F}}

\def\QQ{{\cal Q}}

\def\TT{{\cal T}}

 \def \X {{\Xi}}
 \def \G {{\Gamma}}
 \def \L {{\Lambda}}
 
 \def \b {{\beta}}
 \def \e {{\varepsilon}}
 \def \s {{\sigma}}

 \def \m {{\mu}}

 \def \t {{\tau}}
 \def \th {{\theta}}
 
 \def \g {{\gamma}}
 \def \l {{\lambda}}
 \def \d {{\delta}}
 \def \a {{\alpha}}
 \def \o {{\omega}}
 \def \O {{\Omega}}
 \def \th {{\theta}}

 \def \del {{\partial}}

 \def \c {{\chi}}

 \def \ba {\begin{array}}
 \def \ea {\end{array}}

\def \A {{\mathbb A}}
\def \B {{\mathbb B}}

\def \X {{\mathbb X}}

\def \frf {{\mathfrak f}}

 \def \cC {{\cal C}}
 
 \def \cE {{\cal E}}
 \def \cF {{\cal F}}
 \def \cG {{\cal G}}
 \def \cH {{\cal H}}

 \def \cQ {{\cal Q}}

 \def \cS {{\cal S}}
 \def \cU {{\cal U}}

 \def \cX {{\cal X}}
 \newcommand{\be}{\begin{equation}}
 \newcommand{\ee}{\end{equation}}

\newcommand{\bea}{\begin{eqnarray}}
 \newcommand{\eea}{\end{eqnarray}}
\def\TH(#1){\label{#1}}\def\thv(#1){\ref{#1}}
\def\Eq(#1){\label{#1}}\def\eqv(#1){(\ref{#1})}

\def\sfrac#1#2{{\textstyle{#1\over #2}}}

\def\capa{\hbox{\rm cap}}
 \def \1{\mathbbm{1}}
\def\wt {\widetilde}

 \def \bz {{\boldsymbol z}}
 \def \bx{{\boldsymbol x}}
   \def \by{{\boldsymbol y}}
\def \bu {{\boldsymbol u}}
\def \bm{{\boldsymbol m}}
 \def \bfe{{\boldsymbol e}}
 \def \bv {{\boldsymbol v}}

\def \bA {{\boldsymbol A}}
\def \bB {{\boldsymbol B}}
\def \bG {{\boldsymbol G}}
\def \bU {{\boldsymbol U}}

\def\checkv{{\boldsymbol {\check v}}}
\def \frf {{\mathfrak f}}

\newcommand \Capm[1]{\C ap_N^\epsilon   ( #1)}
\newcommand \wCapm[1]{\wt{\C ap}_N^\epsilon   ( #1)}

\newcommand{\step}[1]{S{\small TEP}~#1}

\def \lb {\left(}
\def \rb {\right)}
\def \lbr {\left\{}
\def \rbr {\right\}}

\def \sign{\hbox{\rm sign}}

\newcommand{\muc}[1]{\mu_{\beta , N}^{#1}}  
\def \so {\text{\small{o}}(1)}
\def \po {\text{\small{o}}}

\def\vm{\bm}

\renewcommand \Capm[1]{\C {\rm ap}_N^n  ( #1)}
\renewcommand \wCapm[1]{\wt{\C{\rm ap}}_N^n  ( #1)}

\def\ubx{\underline{\bx}}
\def\us{\underline{\s}}

\begin{document}

 \title[Metastability in the RFCW model]
{Sharp asymptotics for metastability in the\\
 random field Curie-Weiss model}
 \author[A. Bianchi]{
Alessandra Bianchi}
\address{A. Bianchi \\Weierstrass-Institut
f\"ur Angewandte Analysis und
 Stochastik\\ Mohrenstrasse 39\\ 10117 Berlin, Germany}
\email{bianchi@wias-berlin.de}
\author[A. Bovier]{Anton Bovier}
 \address{A. Bovier\\Weierstrass-Institut f\"ur Angewandte Analysis und
 Stochastik\\ Mohrenstrasse 39\\ 10117 Berlin, Germany\\ and
 Institut f\"{u}r Mathematik \\Technische Universit\"{a}t Berlin\\
 Strasse des 17.\ Juni 136\\ 10623 Berlin, Germany}
\email{bovier@wias-berlin.de}
\author[D. Ioffe]{Dmitry Ioffe}
 \address{D. Ioffe\\William Davidson Faculty of Industrial Engineering and
 Management\\ Technion\\ Haifa 32000, Israel}
\email{ieioffe@technion.ac.il}

\date{\today}

 \begin{abstract}

 In this paper we study the metastable behavior of one of the simplest
 disordered spin system, the random field Curie-Weiss model. We will
 show how the potential theoretic approach can be used to prove sharp
 estimates on capacities and metastable exit times also in the case
 when the distribution of the random field is continuous. Previous work
 was restricted to the case when the random field takes only finitely
 many values, which  allowed the reduction to a finite dimensional
 problem using lumping techniques. Here we produce the first genuine
sharp estimates in a context where entropy is important.
 \end{abstract}

\thanks{This research was supported through a grant by
 the German-Israeli Foundation (GIF). The kind hospitality of the Technion,
 Haifa, and the Weierstrass-Institute for Applied Analysis and
 Stochastics is gratefully acknowledged.}

 \maketitle


 \section{Introduction and main results}
 \label{S1}

The simplest example of disordered mean field  models is the random
field Curie-Weiss model. Here the state space is $\cS_N=\{-1,1\}^N$,
where $N$ is the number of particles of the system. Its Hamiltonian
is
\be
H_N[\o](\s) \equiv -\frac N2 \left(\frac 1N\sum_{i\in\L}
\s_i\right)^2 -\sum_{i\in\L} h_i[\o]\s_i, \Eq(M.5)
\ee
 where $\L\equiv\{1,\ldots, N\}$ and $h_i$, $i\in \L$, are
i.i.d. random variables on some probability space $(\O, \cF, \P_h )$.
For sake of convenience, we will assume throughout
this paper that the common distribution of $h$ has bounded support.

The dynamics of this model has been studied before: dai
Pra and den Hollander studied the short-time dynamics using large
deviation results and obtained the analog of the McKeane-Vlasov
equations \cite{dPdH}.  Mathieu and Picco \cite{MP1}
and Fontes, Mathieu, and Picco \cite{FMP}, considered
convergence to equilibrium in a particularly simple case where the
random field takes only the two  values $\pm \e$. 
Finally, Bovier et al. \cite{BEGK01}  analyzed this model
in the case when $h$ takes finitely many values,
as an example of the use of the potential theoretic approach to
metastability.
In this article we extend this analysis to the case of random fields
with continuous distributions, while at the same time improving the
results by giving sharp estimates of transition times between
metastable states.

The present paper should be seen, beyond the interest presented by
the model as such, as a first case study in the attempt to derive
precise asymptotics of metastable characteristics in kinetic Ising
models in situations where neither the temperature tends to zero nor
an exact reduction to low-dimensional models is possible. While the
RFCW model is certainly one of the simplest examples of this class,
we feel that the general methodology developed here will be
useful in a much wider class of systems.

\subsection{Gibbs measure and order parameter. The static picture}
\label{SS11} The equilibrium statistical mechanics of the RFCW model
was analyzed in detail in \cite{APZ} and \cite{Ku}. We give a very
brief review of some key features that will be useful later. As
usual, we define the Gibbs measure of the model as the random
probability measure
\be\Eq(static.1)
\mu_{\b,N}[\o](\s)\equiv
\frac{2^{-N}e^{-\b
    H_{N}[\o](\s)}}{Z_{\b,N}[\o]},
\ee
where the partition function is defined as
\be\Eq(static.2)
Z_{\b,N}[\o] \equiv \E_\s e^{-\b
    H_{N}[\o](\s)}\equiv2^{-N}\sum_{\s\in S_N}e^{-\b
    H_{N}[\o](\s)}  .
\ee
We define the total magnetization as
\be\Eq(static.3)
 m_N(\s)\equiv \frac 1N\sum_{i\in\L}\s_i.
\ee
The magnetization will be the \emph{order parameter} of the
model, and we define its distribution under the Gibbs measures as
the \emph{induced measure},
\be\Eq(static.4)
\cQ_{\b,N}\equiv
\mu_{\b,N}\circ m_N^{-1}, \ee  on the set of possible
values $\G_N\equiv \{-1,-1+2/N,\dots,1\}$.

Let us begin by writing
\be\Eq(static.11)
Z_{\b,N}[\o]\cQ_{\b,N}[\o](m)=\exp\left(\frac{N\b}2 m^2  \right)
Z^1_{\b,N} [\o](m)
\ee
where
\be\Eq(static.12)
 Z^1_{\b,N}[\o](m)\equiv\E_{\s} \exp\left(\b\sum_{i\in\L} h_i
  \s_i\right)\1_{\left\{ N^{-1}\sum_{i\in\L}\s_i=m\right\}}
  \equiv \E^{h}_{\s}\1_{\left\{N^{-1}\sum_{i\in\L}
\s_i=m\right\}}.
\ee
For simplicity we will in the sequel identify functions defined on the
discrete set $\G_N$ with functions defined on $[-1,1]$ by setting
$f(m)\equiv f([2 N  m]/2N)$.
Then, for  $m\in (-1,1)$, $Z^1_N(m)$ can be expressed,
using sharp large deviation estimates \cite{DIMA}, as
\be\Eq(static.13)
 Z^1_{\b,N}[\o](m)=\frac{\exp\left(-N I_{N}[\o](m)\right)
}{\sqrt{\sfrac {N\pi}2 /I''_{N} [\o](m)}} \left(1+\po(1)\right),
\ee
where
$\po(1)$ goes to zero as $N\uparrow\infty$. This means that we can
express the right-hand side in \eqv(static.11) as
\be\Eq(static.11.1)
Z_{\b,N}[\o]\QQ_{\b,N}[\o](m)= \sqrt{\sfrac{
   2 I''_{N}[\o](m)}{N\pi} }
\exp\left(-N\b F_{\b,N}[\o](m)\right) \left(1+\po(1)\right),
 \ee
where
\bea\Eq(static.116)
F_{\b,N}[\o](m)\equiv -\frac 12 m^2 +\frac
1\b I_{N}[\o](m).
\eea
Here $I_{N}[\o](y)$ is the Legendre-Fenchel transform of the
  log-moment generating function
\bea\Eq(static.14)
U_{N}[\o](t)&\equiv& \frac
1{N}\ln\E^{h}_{\s}\exp\left(t\sum_{i\in\L}  \s_i\right)
\\\nonumber
&=&\frac 1{N}\sum_{i\in \L}\ln\cosh\left(t+\b h_i\right).
\eea
Above we have indicated the random nature of all functions that appear by
making their dependence on the random parameter $\o$ explicit. To
simplify notation, in the sequel this dependence will mostly be dropped.

We are  interested in the behavior of this function near critical
points of $F_{\b,N}$. An  important
 consequence of
Equations \eqv(static.11) through
\eqv(static.14)
is that if  $m^*$ is a
critical point of $F_{\b,N}$, then for  $|v|\leq N^{-1/2+\d}$,
\be \Eq(static.16)
\frac {\QQ_{\b,N}(m^*+v)} {\QQ_{\b,N}(m^*)}
=\exp\left(-\frac{\b N}2 a(m^*)v^2\right)\left(1+\po(1)\right),
 \ee
with
\be\Eq(static.111)
a(m^*) \equiv F''_{\b,N}(m^*) = -1 + \b^{-1} I''_{N}(m^*).
\ee
 Now,   if $m^*$ is a critical
point of $F_{\b,N}$ , then
\be\Eq(static.17)
 m^* =\b^{-1} I'_{N}(m^*)\equiv \b^{-1} t^*,
\ee
or
\be\Eq(static.18) \b m^* = I_{N}'(m^*)=t^*.
\ee
Since $I_{N}$ is the Legendre-Fenchel transform of $U_{N}$,
$I_{N}'(x) = U_{N}'^{-1}
(x)$, so that
\be \Eq(static.19)
m^* =U'_{N}(\b m^*) \equiv \frac
1{N} \sum_{i\in\L} \tanh(\b(m^*+h_i))).
\ee
Finally, using that at a critical point,
$I''_{N,\ell}(m^*)=\frac
  1{U''_{N,\ell}(t^*)}$,
we get the alternative expression
\be \Eq(static.112)
a(m^*)  =-1+ \frac 1{\b  U''_{N} (\b m^*)}
=-1+\frac 1{\frac \b
N\sum_{i\in\L}\left(1-\tanh^2(\b(m^*+h_i))\right)}.
 \ee
We see that, by the law of large numbers, the set of critical points
converges, $\P_h$-almost surely,  to the set of solutions of the equation
\be\Eq(static.200)
m^*=\E_h \tanh\left(\b\left(m^*+h\right)\right),
\ee
and the second derivative of $F_{\b,N}(m^*)$ converges to
\be
\Eq(static.201) \lim_{N\rightarrow\infty} F_{\b,N}''(m^*)
=-1+\frac 1{ \b \E_h \left(1-\tanh^2(\b(m^*+h))\right)}.
\ee
Thus, $m^*$ is a local  minimum if
\be\Eq(static.202)
 \b \E_h \left(1-\tanh^2(\b(m^*+h))\right)<1,
\ee
and a local maximum if
\be\Eq(static.203)
 \b \E_h \left(1-\tanh^2(\b(m^*+h))\right)>1.
\ee
(The cases where $ \b \E_h \left(1-\tanh^2(\b(m^*+h))\right)=1$
correspond to second order phase transitions and will not be
considered here).

\begin{proposition}\TH(static.215)
Let $m^*$ be a critical point of $\QQ_{\b,N}$. Then, $\P_h$-almost surely,
for all but finitely many values of $N$,
\be\Eq(static.161)
Z_{\b,N}\QQ_{\b,N}(m^*) = \frac {\exp\left(-\b NF_{\b,N}(m^*)\right)
\left(1+\po(1)\right)}
{\sqrt{\sfrac {N\pi}{2}\left|\E
\left(1-\tanh^2(\b(m^*+h))\right)\right| }}
\ee
with
\be\Eq(static.162)
F_{\b,N}(m^*)=
\frac {\left(m^*\right)^2}2 - \frac
1{\b N}\sum_{i\in
  \L}\ln\cosh\left(\b(m^*+h_i)\right).
\ee
\end{proposition}

From this discussion we get a very precise picture of the
distribution of the order parameter.

\subsection{Glauber dynamics} \label{SS12}

We will consider for definiteness discrete time
Glauber dynamics with Metropolis
transition probabilities
\be\Eq(metro.1)
p_N[\o](\s,\s')\equiv\frac
1N\exp\left(-\b [H_N[\o](\s')-H_N[\o](\s)]_+\right),
\ee
if $\s$ and $\s'$ differ on a single coordinate,
\be\Eq(metro.2)
p_N[\o](\s,\s)\equiv 1-\sum_{\s'\sim\s}\frac 1N\exp\left(-\b
 [H_N[\o](\s')-H_N[\o](\s)]_+\right),
\ee
and $p_N(\s,\s')=0$ in all other cases. We will denote the
Markov chain corresponding to these transition
 probabilities $\s(t)$ and write $\P_\nu[\o]\equiv \P_\nu$,
for the law of this chain
 with initial distribution $\nu$, and we will set
$\P_\s\equiv \P_{\d_\s}$.
As is well known, this chain is ergodic and reversible with respect
to the Gibbs measure $\mu_{\b,N}[\o]$, for each $\o$.
Note that we might also study chains with different
 transition probabilities that are reversible with respect to the
 same measures. Details of our results will depend on this choice.
The transition matrix associated with these transition probabilities
 will be called $P_N$, and we will denote by $L_N\equiv P_N-\1$ the
 (discrete) generator of the chain.

Our main result will be sharp estimates for mean hitting times
between minima of the function $F_{\b,N}(m)$ defined in \eqv(static.116).

More precisely, for any subset $A\subset S_N$, we define the
stopping time
\be \Eq(strange.20)
\t_A\equiv \inf\{t>0|\s(t)\in A\}.
 \ee
We also need to define, for any two subsets $A,B\subset S_N$,
the probability measure  on $A$ given by
\be\Eq(theorem1.1)
\nu_{A,B}(\s)=\frac{\mu_{\b,N}(\s)\P_\s[\t_{B}<\t_A]}
{\sum_{\s\in A}\mu_{\b,N}(\s)\P_\s[\t_{B}<\t_A]}.
\ee
We will be mainly concerned with sets of
configurations with given magnetization. For any  $I\in \G_N$,
we thus introduce the notation
$S[I]\equiv\{\s\in S_N: m_N(\s)\in I\}$ and state the following:
\begin{theorem}\TH(theorem1)
Assume that $\b$ and the distribution of the magnetic field are such
that there exist more than one local minimum of $F_{\b,N}$. Let
$m^*$ be a local minimum of $F_{\b,N}$, $M\equiv M(m^*)$ be the set
of minima of $F_{\b,N}$ such that $F_{\b,N}(m)<F_{\b,N}(m^*)$,
and  $z^*$ be the minimax between $m$ and $M$, i.e.  the lower of the
highest maxima separating $m$ from $M$ to the left respectively right.
Then, $\P_h$-almost surely, for all but finitely many values of $N$,
\bea\Eq(theorem1.2)
 \E_{\nu_{S[m^*],S[M]}} \t_{S[M]}&=&
\exp\left(\b
N\left[F_{\b,N}(z^*)-F_{\b,N}(m^*)\right]\right)\\\nonumber
&&\times\frac {2\pi N}  {\b |\bar\g_1|}
\sqrt{\frac{\b\E_h\left(1-\tanh^2\left(\b(z^*+h)\right)\right)-1}
{1-\b\E_h\left(
    1-\tanh^2\left(\b(m^*+h)\right)\right)}}\left(1+\po(1)\right),
\eea
where $\bar\g_1$ is the unique negative solution of the
equation

\be\Eq(theorem1.3)
\E_h\left[\frac{ \left(1-\tanh(\b(z^*+h))\right)
\exp{(-2\b\left[z^* + h\right]_+)} }
{\frac{\exp{(-2\b\left[z^* +
 h\right]_+)}}{ \b\left(1+\tanh(\b(z^*+h)) \right)}-2\g}\right]=1.
 \ee
Note that we have the explicit
representation
 for the random quantity
\bea\Eq(theorem.1.2.1)
 F_{\b,N}(z^*)-F_{\b,N}(m^*)\, &=&
 \frac { \left(z^*\right)^2 - \left(m^*\right)^2}2\\\nonumber&-&
\frac{1}{\b N}\sum_{i\in
  \L}\left[\ln\cosh\left(\b(z^*+h_i)\right)
-\ln\cosh\left(\b(m^*+h_i)\right)\right].
\eea
\end{theorem}

The proof of this result on mean transition times relies on the
following result on \emph{capacities} (for a definition see
Eq. \eqv(pot.5) in Section 2 below).

\begin{theorem} \TH(CAP-thm) With the same notation as in Theorem
  \thv(theorem1) we have that
\be\Eq(cap1.2)
Z_{\b,N}\capa\left(S[m^*],S[M]\right)=
\frac   {\b |\bar\g_1|}{2\pi N} \frac{\exp\left(-\b
NF_{\b,N}(z^*)\right)\left(1+\po(1)\right)}{
\sqrt{{\b\E_h\left(1-\tanh^2\left(\b(z^*+h)\right)\right)-1}}}.
\ee
\end{theorem}

The proof of Theorem \thv(CAP-thm) is the core of the present paper.
As usual, the proof of an upper bound of the form \eqv(cap1.2) will
be relatively easy. The main difficulty is to prove a corresponding
lower bound. The main contribution of this paper is to provide a
method to prove such a lower bound in a situation where the entropy of
paths cannot be neglected.

Before discussing the methods of proof of these results, it will be
interesting to compare this theorem with the prediction of the
simplest uncontrolled approximation.

\paragraph{The naive approximation.} A widespread heuristic picture
 for metastable behavior of systems like the RFCW model is based on
 replacing the full Markov chain on $S_N$ by an effective Markov chain
 on the order parameter, i.e. by a nearest neighbor random walk on
 $\G_N$ with transition probabilities  that are reversible
with respect to the induced measure, $\cQ_{\b,N}$.  The ensuing
model
 can be solved exactly. In the absence of a random magnetic field,
 this replacement is justified since the image of $\s(t)$, $m(t)\equiv
 m_N(\s(t))$, is a Markov chain reversible w.r.t. $\cQ_{\b,N}$;
 unfortunately, this fact relies on the perfect permutation symmetry
 of the Hamiltonian of the Curie-Weiss model and fails to hold in the
 presence
of random field.

A natural choice for the transition rates of the heuristic dynamics
is
\be\Eq(naive.1)
 r_N[\o](m,m')\equiv \frac 1{\cQ_{\b,N}[\o](m)}
 \sum_{\s: m_N(\s)=m}
\mu_{\b,N}[\o](\s)
 \sum_{\s':m_N(\s')=m'} p_N[\o](\s,\s'),
\ee
which are different from zero only if $m'=m\pm 2/N$ or if
$m=m'$. The ensuing Markov process is a one-dimensional nearest
neighbor random walk for which most quantities of interest can be
computed quite explicitly by elementary means (see e.g.
\cite{vK,B04}). In particular, it is easy to show that for this
dynamics,
\bea\Eq(naive.2)\nonumber
 \E_{\nu_{S[m^*],S[M]}}
\t_{S[M]}&=& \exp\left(\b
N\left[F_{\b,N}(z^*)-F_{\b,N}(m^*)\right]\right)\\\nonumber
&&\times\frac {2\pi N}  {\b |a(z^*)|}
\sqrt{\frac{\b\E_h\left(1-\tanh^2\left(\b(z^*+h)\right)\right)-1}
{1-\b\E_h\left(
    1-\tanh^2\left(\b(m^*+h)\right)\right)}}\left(1+\po(1)\right),
\eea
where $a(z^*)$ is defined in \eqv(static.201).

The prediction of the naive approximation is slightly different from the
exact answer, albeit only by a wrong prefactor. One may of course
consider this as a striking confirmation of the quality of the naive
approximation; from a different angle, this shows that a true
understanding of the details of the dynamics is only reached when
the prefactors of the exponential rates are known (see \cite{MaSt}
for a discussion of this point).

The picture above is in some sense generic for a much wider class of
metastable systems: on a heuristic level, one wants to think of the
dynamics on metastable time scales to be well described by a diffusion
in a double (or multi) well potential. While this cannot be made
rigorous, it should be possible to find a family of mesoscopic
variables with corresponding (discrete) diffusion  dynamics that
asymptotically reproduce the metastable behavior of the true
dynamics. The main message of this paper is that such a picture can be
made rigorous within the potential theoretic approach.

\noindent {\bf Acknowledgments.} The authors thank   Alexandre
Gaudilli\`{e}re,  Frank den Hollander, and Cristian Spitoni for useful
discussions on metastability.

\section{Some basic concepts from potential theory}

Our approach to the analysis of the dynamics introduced above will
be based on the ideas developed in \cite{BEGK01,BEGK02,BEGK03} to
analyze metastability through a systematic use of classical
potential theory. Let us recall the basic notions we will need.

For two disjoint sets $A,B\subset S_N$, the equilibrium potential,
$h_{A,B}$, is the harmonic function, i.e. the solution of the
equation
\be\Eq(pot.1)
(Lh_{A,B})(\s)=0, \quad \s\not \in A\cup B,
\ee
with boundary conditions
\be \Eq(pot.2)
h_{A,B}(\s)=\begin{cases} 1,&\hbox{\rm  if }\, \s\in A\\
0, &\hbox{\rm  if }\, \s\in B\end{cases}.
\ee
The \emph{equilibrium measure} is the function
\be\Eq(pot.3)
e_{A,B}(\s)\equiv -(Lh_{A,B})(\s)=(Lh_{B,A})(\s),
 \ee
 which clearly is non-vanishing
only on $A$ and $B$. An important formula  is the
discrete analog of the first Green's identity: Let $D\subset S_N$
and $D^c\equiv S_N\setminus D$. Then, for any function $f$, we have
\bea\Eq(pot.4)
&&\frac 12\sum_{\s,\s'\in S_N}\mu(\s) p_N(\s,\s')
[f(\s)-f(\s')]^2\\\nonumber &&=-\sum_{\s\in D} \mu(\s) f(\s)
(Lf)(\s)-\sum_{\s\in D^c} \mu(\s) f(\s) (Lf)(\s).
\eea
In particular, for $f=h_{A,B}$, we get that
\bea\Eq(pot.5)
&&\frac 12\sum_{\s,\s'\in S_N}\mu(\s) p_N(\s,\s')
[h_{A,B}(\s)-h_{A,B}(\s')]^2
\\\nonumber
&&=\sum_{\s\in A} \mu(\s)
e_{A,B}(\s)\equiv \capa(A,B),
\eea
where the right-hand side is
called the \emph{capacity} of the capacitor $A,B$. The functional
appearing on the left-hand sides of these relations is called the
\emph{Dirichlet form} or \emph{energy}, and denoted
\be\Eq(pot.6)
\Phi_N(f)\equiv \frac 12\sum_{\s,\s'\in S_N}\mu(\s) p_N(\s,\s')
[f(\s)-f(\s')]^2.
\ee
As a consequence of the \emph{maximum principle}, the
function $h_{A,B}$ is the unique minimizer of $\Phi_N$ with boundary
conditions \eqv(pot.2), which implies the  \emph{Dirichlet
principle}:
\be\Eq(pot.7)
\capa(A,B) =\inf_{h\in \cH_{A,B}}\Phi_N(h),
 \ee
where $\cH_{A,B}$ denotes the space of functions
satisfying \eqv(pot.2).

Equilibrium potential and equilibrium measure have an immediate
probabilistic interpretation, namely \be\Eq(pot.8)
\P_\s[\t_A<\t_B]=\begin{cases} h_{A,B}(\s), &\,\hbox{\rm if}\,
\s\not
  \in A\cup B\\
e_{B,A}(\s),  &\,\hbox{\rm if}\, \s\in B.
\end{cases}
\ee
 An important observation is that equilibrium potentials and
equilibrium measures also determine the Green's function. In fact
(see e.g. \cite{BEGK02,Bo5}),
\be
h_{A,B}(\s)=\sum_{\s'\in A}
G_{S_N\setminus  B}(\s,\s') e_{A,B}(\s') \Eq(pot.9)
 \ee
In the case then $A$ is a single point, this relation can  be solved
for the Green's function to give
\be
 G_{S_N\setminus B}(\s,\s')=\frac{\mu(\s')h_{\s,B}(\s)}{\mu(\s)
   e_{\s,B}(\s)}.\Eq(pot.10)
\ee
 This equation is perfect if the cardinality of the state space does
 not grow too fast. In our case, however, it is  of  limited use, since
 both numerator and
denominator  tend to be very close to zero for the wrong reason.
However, \eqv(pot.9) remains useful. In particular, it gives the
following representation for mean hitting times
\be\Eq(pot.11)
\sum_{\s\in A} \mu(\s) e_{A,B}(\s) \E_\s \t_B =\sum_{\s'\in S_N}
\mu(\s') h_{A,B}(\s'),
\ee
 or, using definition \eqv(theorem1.1)
\be\Eq(pot.12)
\E_{\nu_{A,B}} \t_B =\frac{1}{\capa(A,B)}\sum_{\s'\in
S_N } \mu(\s') h_{A,B}(\s').
\ee

From these equations we see that our main task will be to obtain
precise estimates on capacities and some reasonably accurate
estimates on equilibrium potentials. In previous applications
\cite{BEGK01,BEGK02,BEGK03,BM02,BdHN}, three main ideas were used to
obtain such estimates:
\begin{itemize}
\item[(i)] Upper bounds on capacities can be obtained using the
  Dirichlet variational principle with judiciously chosen test
  functions.
\item[(ii)] Lower bounds were usually obtained using the monotonicity
  of capacities in the transition  probabilities (Raighley's
  principle). In most applications, reduction of the network to a set
  of parallel $1$-dimensional chains was sufficient to get good
  bounds.
\item[(iii)] The simple renewal estimate $h_{A,B}(x)\leq
  \frac{\capa(x,A)}{\capa(x,B)}$ was used to bound the equilibrium
  potential through capacities again.
\end{itemize}
These methods were sufficient in previous applications essentially
because entropy were not an issue there. In the models at hand,
entropy is important, and due to the absence of any symmetry, we
cannot use the trick to deal with entropy by a mapping of the model
to a low-dimensional one, as can be done in the standard Curie-Weiss
model and in the RFCW model when the magnetic field takes only
finitely many values \cite{MP1,BEGK01}.

Thus we will need to improve on these ideas. In particular, we will
need a new approach to lower
bounds for capacities. This will be done by exploiting a dual variational
representation of capacities in terms of flows, due to Berman and
Konsowa \cite{BerKon}. Indeed, one of the main messages of this paper
is to illustrate the power of this variational principle.

\paragraph{Random path representation and lower bounds
  on capacities.}
It will be convenient to think of the quantities
$\mu(\s)p_N(\s,\s')$ as \emph{conductances}, $c(\s,\s')$, associated
to the edges $e=(\s,\s')$ of the graph of allowed transitions of our
dynamics. This interpretation is justified since, due to
reversibility, $c(\s,\s')=c(\s',\s)$ is symmetric.

For purposes of the exposition, it will be useful to abstract from the
specific model and to
consider a general finite connected graph, $(S,\cE )$ such that
whenever $e=(a,b)\in \cE$, then also $-e\equiv (b,a)\in \EE$. Let
this graph be endowed with a symmetric function,  $c:\EE\rightarrow
\R_+$, called conductance.

Given two disjoint subsets $A,B\subset S$ define the capacity,
\be\label{Capacity}
\capa (A, B)= \frac12\min_{h|_A =0, \  h|_B
=1}\sum_{e=
  (a,b)\in\cE} c(a,b)\lb h (b) - h(a)\rb^2.
\ee

\begin{definition}\TH(unit-flow) Given two disjoint sets,
  $A,B\subset S$,
a non-negative, cycle free \emph{unit flow}, $f$, from $A$ to $B$ is a function
$f:\cE\rightarrow \R_+\cup \{0\}$, such that the following conditions
are verified:
\begin{itemize}
\item[(i)] if $f(e)>0$, then $f(-e)=0$;
\item[(ii)] $f$ satisfies \emph{Kirchoff's law}, i.e. for any vertex
  $a\in S\setminus (A\cup B)$,
\be\Eq(ab.1)
\sum_{b} f(b,a)=\sum_{d}f(a,d);
\ee
\item[(iii)]
\be \label{Normalization}
 \sum_{a\in A}\sum_{b} f (a,b )= 1=
\sum_{a}\sum_{b\in
  B}f(a,b );
\ee
\item[(iv)]  any path, $\g$, from $A$ to $B$ such that $f(e)>0$ for all $e\in
  \g$,  is self-avoiding.
\end{itemize}
We will denote the space of  non-negative, cycle free unit flows from $A$ to
$B$ by $\U_{A,B}$.
\end{definition}

An important example of a unit flow can be constructed from the
equilibrium potential, $h^*$, i.e. the unique minimizer of
\eqref{Capacity}. Since $h^*$  satisfies,
for any $a\in S\setminus (A\cup B)$,
\be\Eq(ab.2)
\sum_{b} c(a,b )(h^* (b) - h^* (a) )= 0,
 \ee
one verifies easily that the function, $f^*$, defined by
\be\label{fstarflow}
f^* (a,b ) \equiv \frac1{\capa  (A,B)} c(a,b )\lb
h^*(a)- h^*(b)\rb_+,
\ee
is a non-negative unit flow from $A$ to
$B$. We will call $f^*$ the \emph{harmonic flow}.

The key observation is that \emph{any} $f\in \U_{A,B}$
gives rise to a \emph{lower bound} on the capacity
$\capa(A,B)$, and that this bound becomes sharp for the harmonic
flow.
To see this we construct from $f$ a
stopped Markov chain $\X =\lb\X_0 ,\dots ,\X_\tau \rb$ as follows:
For each $a\in S\setminus B$ define $F (a) = \sum_b f (a,b )$.

We define the initial distribution of our chain as  $\P^f (a) =
F(a)$, for $a\in A$, and zero otherwise. The transition
probabilities are given by
\be\Eq(ab.3)
q^f (a,b )= \frac{f(a,b)}{F(a)},
\ee
for $a\not\in B$, and the chain is stopped on arrival
in $B$. Notice that by our choice of the initial distribution and in
view of \eqref{ab.3}
 $\X$ will never visit sites $a\in S\setminus B$ with $F (a) =0$.

Thus, given a trajectory  $\cX = (a_0 ,a_1,\dots ,a_r )$ with
$a_0\in A$, $a_r\in B$ and $a_\ell\in S\setminus (A\cup B)$ for
$\ell=0, \dots, r-1$,
\be \label{PfTrajectory}
\P^f\lb \X = \cX\rb= \frac{\prod_{\ell=0}^{r-1} f (e_\ell )}
{\prod_{\ell=0}^{r-1} F(a_\ell )} ,
 \ee
where $e_\ell = (a_\ell ,a_{\ell+1})$ and we use the  convention
$0/0 =0$. Note that, with the above definitions, the probability that
$\X$ passes through an edge $e$ is
\be \label{ProbFlow}
\P^f\lb e\in \X\rb= \sum_{\cX} \P^f (\cX )\1_{\lbr e\in\cX\rbr}=f(e).
 \ee
 Consequently, we have  a
partition  of unity,
\be\Eq(partunit)
\1_{\lbr f (e) >0\rbr}=
\sum_{\cX}\frac{\P^f (\cX )\1_{\lbr
    e\in\cX\rbr}}{f (e)} .
\ee
We are ready now to derive our $f$-induced lower bound: For
every function $h$ with $h|_A =0$ and $h|_B =1$,
\bea
\nonumber
\frac12\sum_{e} c(e)\lb \nabla_e h\rb^2 &\geq& \sum_{e : f (e) >0}
c(e)\lb \nabla_e h\rb^2\\
\nonumber
&=& \sum_\cX\sum_{e\in\cX} \P^f (\cX )\frac{c (e)}{f(e)} \lb \nabla_e h\rb^2 .
\eea
As a result, interchanging the minimum and the sum,
\bea
\nonumber \capa  (A,B) &\geq &  \sum_r\sum_{\cX = (a_0 ,\dots,
a_r)}\P^f (\cX ) \min_{h (a_0 )=0,\ h (a_r )=1} \sum_0^{r-1} \frac{
c (a_\ell ,a_{\ell+1}
  )}{f (a_\ell ,a_{\ell+1})}
\lb h(a_{\ell+1} ) - h (a_\ell )\rb^2\\
\label{LowerBound} &=& \sum_\cX \P^f (\cX )\left[ \sum_{e\in\cX}
\frac{ f(e)}{c (e)}\right]^{-1} .
\eea
 Since for the equilibrium
flow, $f^* $,
\be\Eq(eqflow.1)
\sum_{e\in\cX} \frac{ f^*(e)}{c (e)}
=  \frac1{\capa (A,B )} ,
\ee
with $\P^{f^*}$-probability one, the bound \eqref{LowerBound} is sharp.

Thus we have proven the following result  from \cite{BerKon}:

\begin{proposition}\TH(ab.4)
Let $A,B\subset S$. Then, with the notation introduced above,
\be\Eq(ab.5)
\capa(A,B)=\sup_{f\in \U_{A,B}}\E^f\left[\sum_{e\in
\cX}\frac{f(e)}{c(e)}\right]^{-1}
\ee
\end{proposition}

\section{Coarse graining and the mesoscopic approximation}
\label{S2}

The problem of entropy forces us to investigate the model on a
coarse grained scale. When the  random fields take only finitely
many values, this can be done by an exact mapping to a low-dimensional
chain. Here this is not the case, but we can will construct  a sequence
of approximate mappings that in the limit allow to extract the exact result.

\subsection{Coarse graining}
\label{SS21}

Let $I$ denote the support of the distribution of the random fields.
Let $I_\ell$, with $\ell\in\{1,\dots,n\}$, be a partition of $I$
such that, for some $C<\infty$ and
 for all $\ell$, $|I_\ell|\leq C/n\equiv \e$.

\def\mm{m}

Each realization of the random field $\{h_i[\o]\}_{i\in\N}$ induces
a random partition of the  set $\L\equiv \{1,\dots,N\}$ into subsets
\be
\L_k[\o]\equiv \{i\in\L:h_i[\o]\in I_k\}. \Eq(9.2)
\ee
 We may
introduce $n$ order parameters
 \be
  \bm_k[\o](\s)\equiv \frac
1N\sum_{i\in\L_k[\o]}\s_i. \Eq(9.3)
\ee
 We denote by $ \vm\, [\o]$ the
$n$-dimensional vector $(\bm_1[\o],\dots,\bm_n[\o])$. In the sequel
we will use the convention that bold symbols denote $n$-dimensional
vectors and their components, while the sum of the components is
denoted by the corresponding plain symbol,
e.g. $m\equiv \sum_{\ell=1}^n\bm_\ell$.
 $\vm$ takes values
in the set
 \be
 \G_N^n[\o] \equiv\times_{k=1}^n
\left\{-\rho_{N,k}[\o],-\rho_{N,k}[\o]+\sfrac 2N,\dots,
\rho_{N,k}[\o]-\sfrac 2N,\rho_{N,k}[\o]\right\}, \Eq(9.4)
 \ee
where
\be
\rho_k\equiv \rho_{N,k}[\o]\equiv \frac {|\L_k[\o]|}N. \Eq(9.5)
 \ee
We will denote by $\bfe_\ell$, $\ell=1,\dots,n$,
 the lattice vectors of the set $\G_N^n$, i.e. the
vectors of length $2/N$ parallel to  unit vectors.

Note that the random variables $\rho_{N,k}$ concentrate
exponentially (in $N$) around their mean values
$\E_h\rho_{N,k}=\P_h[h_i\in I_k]\equiv p_k$.

\vspace {3mm}
\noindent \textbf{Notational warning:} To simplify statements in the
remainder of the paper, we will henceforth assume that all statements
involving random variables on $(\O,\FF,\P_h)$  hold true with
$\P_h$-probability one, for all but finitely many values of $N$.

\vspace{3mm}

We may write the Hamiltonian in the form
\be
 H_N[\o](\s)= -N E(\vm[\o](\s)) +\sum_{\ell=1}^n \sum_{i\in
  \L_\ell}\s_i\tilde h_i[\o],\Eq(9.6)
\ee
where $E:\R^n\rightarrow \R$ is the function
\be
E(\bx)\equiv \frac 12\left(\sum_{k=1}^n \bx_k\right)^2+\sum_{k=1}^n
\bar h_k \bx_k \Eq(9.7),
\ee
with
\be\Eq(9.7.1)
\bar h_\ell\equiv
\frac 1{|\L_\ell|}\sum_{i\in \L_\ell} h_i,\quad\text{ and}\quad
\tilde h_i\equiv h_i-\bar h_\ell.
\ee
Note that if $h_i=\bar h_\ell$
for all $i\in \L_\ell$, which is the case when $h$ takes only
finitely many values and the partition $I_\ell$ is chosen suitably,
then the Glauber dynamics  under the family of functions $\bm_\ell$ is
again Markovian. This fact was exploited in \cite{MP1,BEGK01}. Here
we will consider the case where this is not the case. However, the
idea behind our approach is to exploit that by choosing $n$ large
we can get to a situation that is rather close to that one.

Let us define the equilibrium distribution of the variables
$\vm[\s]$
\bea
 \cQ_{\b,N}[\o](\bx)
 &\equiv&
\mu_{\b,N}[\o](\vm[\o](\s)=\bx)
\\\nonumber
&=& \frac 1{Z_N[\o]} e^{\b NE(\bx)} \E_\s
\1_{\{\vm[\o](\s)=\bx\}}e^{\sum_{\ell=1}^n \sum_{i\in
\L_\ell}\s_i(h_i-\bar h_\ell)}
%
\Eq(9.8)
 \eea
where $Z_N[\o]$ is the normalizing partition function. Note that with
some abuse of notation, we will use the same symbols $\cQ_{\b,N}$,
$F_{\b,N}$ as in Section 1 for functions defined on the
$n$-dimensional variables $\bx$. Since we distinguish the vectors
from the scalars by use of bold type, there should be no confusion
possible. Similarly, for a mesoscopic subset $\bA\subseteq \G_N^n [\o ]$,
we define its microscopic counterpart,
\be\Eq(put-numbers.10)
 A = \cS_N [\bA ] = \lbr\sigma\in\cS_N~:~ \bm (\sigma )\in \bA\rbr .
\ee

\subsection{The landscape near critical points.}\label{sect:nearcritical}

We now  turn to the precise computation of the behavior of the
measures $\QQ_{\b,N}[\o](\bx)$ in the neighborhood of the critical
points of $F_{\b,N}[\o](\bx)$. We will see that this goes very much
along the lines of the analysis in the one-dimensional case in Section 1.

Let us begin by writing
\be\Eq(fine.1)
Z_{\b,N}[\o]\QQ_{\b,N}[\o](\bx)=\exp\left(N\b\left(\frac
12\left(\sum_{\ell=1}^n\bx_\ell\right)^2 +\sum_{\ell=1}^n
\bx_\ell\bar h_\ell\right)\right) \prod_{\ell=1}^n
Z^\ell_{\b,N}[\o](\bx_\ell/\rho_\ell),
\ee
where
\be\Eq(fine.2)
 Z^\ell_{\b,N}[\o](y)
\equiv\E_{\s_{\L_\ell}} \exp\left(\b\sum_{i\in\L_\ell}\tilde h_i
  \s_i\right)\1_{\left\{|\L_\ell|^{-1}\sum_{i\in\L_\ell}
\s_i=y\right\}} \equiv \E^{\tilde
    h}_{\s_{\L_\ell}}\1_{\left\{|\L_\ell|^{-1}\sum_{i\in\L_\ell}
\s_i=y\right\}}.
\ee
For $y\in(-1,1)$, these $Z^\ell_N$ can be expressed,  using sharp
large deviation estimates \cite{DIMA}, as
\be\Eq(fine.3)
 Z^\ell_{\b,N}[\o](y)=\frac{\exp\left(-|\L_\ell| I_{N,\ell}[\o](y)\right)}
{\sqrt{\sfrac\pi 2 |\L_\ell|/I''_{N,\ell}[\o](y)}}
\left(1+\po(1)\right),
\ee
where $\po(1)$ goes to zero as $|\L_\ell|\uparrow\infty$. Note that as
in the one-dimensional case, we identify functions on $\G_N^n$ with
their natural extensions  to $\R^n$. This
means that we can express the right-hand side in \eqv(fine.1) as
\be\Eq(fine.1.1)
Z_{\b,N}[\o]\QQ_{\b,N}[\o](\bx)=\prod_{\ell=1}^n
{\sqrt\sfrac{
\left(I''_{N,\ell}[\o](\bx_\ell/\rho_\ell)/\rho_\ell\right)}
{ {N\pi}/ 2}}
\exp\left(-N\b F_{\b,N}[\o](\bx)\right) \left(1+\po(1)\right),
\ee
where
\be\Eq(fine.16)
F_{\b,N}[\o](\bx)\equiv - \frac 12
\left(\sum_{\ell=1}^n\bx_\ell\right)^2 -\sum_{\ell=1}^n \bx_\ell
\bar h_\ell +\frac 1\b \sum_{\ell=1}^n \rho_\ell
I_{N,\ell}[\o](\bx_\ell/\rho_\ell).
\ee

Here $I_{N,\ell}[\o](y)$ is the Legendre-Fenchel transform of the
  log-moment generating function,
\bea\Eq(fine.4)
U_{N,\ell}[\o](t)&\equiv& \frac
1{|\L_\ell|}\ln\E^{\tilde
  h}_{\s_{\L_\ell}}\exp\left(t\sum_{i\in\L_\ell}  \s_i\right)
\\\nonumber
&=&\frac 1{|\L_\ell|}\sum_{i\in \L_\ell}\ln\cosh\left(t+\b \tilde
h_i\right).
\eea
We again analyze our functions
near critical points, $\bz^*$, of $F_{\b,N}$.
Equations \eqv(fine.1)-\eqv(fine.4) imply: if $\bz^*$ is  a critical point, then,
for $\|\bv\|\leq
N^{-1/2+\d}$,
\be \Eq(fine.6)
\frac {\QQ_{\b,N}(\bz^*+\bv)}
{\QQ_{\b,N}(\bz^*)} =\exp\left(-\frac{\b
N}2(\bv,\A(\bz^*)\bv)\right)\left(1+\po(1)\right),
\ee
with
\be\Eq(fine.11)
(\A(\bz^*))_{k\ell} =\frac
{\del^2F_{\b,N}(\bz^*)}{\del \bz_k\del \bz_\ell} = -1 +\d_{k,\ell}
\b^{-1}\rho_\ell^{-1} I''_{N,\ell}(\bz^*_\ell/\rho_\ell)\equiv
-1+\d_{\ell,k} \hat\l_\ell.
\ee
Now,  if $\bz^*$ is a critical point of $F_{\b,N}$ ,
\be\Eq(fine.7)
\sum_{j=1}^n \bz^*_j + \bar h_\ell=\b^{-1}
I'_{N,\ell}(\bz^*_\ell/\rho_\ell)\equiv \b^{-1} t^*_\ell,
\ee
or, with $z^*= \sum_{j=1}^n \bz^*_\ell$,
\be\Eq(fine.8)
\b\left(z^* + \bar h_\ell\right)=
I_{N,\ell}'(\bz^*_\ell/\rho_\ell)=t^*_\ell.
\ee
By standard properties of  Legendre-Fenchel transforms,
we have that  $I_{N,\ell}'(x)
=U_{N,\ell}'^{-1} (x)$, so that
\be \Eq(fine.9)
\bz^*_\ell/\rho_\ell
=U'_{N,\ell}(\b(z^*+h_\ell)) \equiv \frac 1{|\L_\ell|}
\sum_{i\in\L_\ell} \tanh(\b(z^*+h_i))).
 \ee
Summing over $\ell$, we see that $z^*$ must satisfy the
equation
\be\Eq(fine.10)
z^*=\frac 1N\sum_{i\in \L}
\tanh(\b(z^*+h_i)),
\ee
which nicely does not depend on our choice
of the coarse graining (and hence on $n$).


 Finally, using that at a critical point
$I''_{N,\ell}(\bz^*_\ell/\rho_\ell)=\frac
  1{U''_{N,\ell}(t^*_\ell)}$,
we get the explicit expression for the random  numbers $\hat\l_\ell$  on
the right hand side of   \eqref{fine.11}
\be \Eq(fine.12)
 \hat\l_\ell =\frac 1{\b \rho_\ell U''_{N,\ell}
(\b(z^*+\bar h_\ell))}=\frac 1{\frac \b
N\sum_{i\in\L_\ell}\left(1-\tanh^2(\b(z^*+h_i))\right)}.
\ee
 The determinant of the matrix $\A(\bz^*)$ has
a simple expression of the form
\bea\Eq(fine.13)
\det\left(\A(\bz^*)\right) &=& \left(1-\sum_{\ell=1}^n \frac
1{\hat\l_\ell}\right)\prod_{\ell=1}^n \hat\l_\ell
\\\nonumber
&=&\left(1- \frac \b N\sum_{i\in \L}
\left(1-\tanh^2(\b(z^*+h_i))\right)\right) \prod_{\ell=1}^n
\hat\l_\ell
\\\nonumber
&=&\left(1-  \b \E_h
\left(1-\tanh^2(\b(z^*+h))\right)\right) \prod_{\ell=1}^n
\hat\l_\ell \left(1 +o(1)\right),
\eea
where $\po(1)\downarrow 0$, a.s., as $N\uparrow\infty$.
Combing these observations, we arrive at
the following proposition.

\begin{proposition}\TH(fine.15)
Let $\bz^*$ be a critical point of
  $\QQ_{\b,N}$. Then $\bz^*$ is given by \eqv(fine.9) where  $z^*$
 is a  solution of \eqv(fine.10).  Moreover,
\bea\Eq(fine.16.1)
Z_{\b,N}\QQ_{\b,N}(\bz^*) &=& \frac { \sqrt{
|\det(\A(\bz^*))| }}{\sqrt{\left( \sfrac {N\pi}
  {2\b}\right)^n\left|\b\E_h
\left(1-\tanh^2(\b(z^*+h))\right)-1\right| }}
\\\nonumber
&\times&\exp\left(\b N\left(-\frac { \left(z^*\right)^2}2 +
\frac 1{\b N}\sum_{i\in\L}\ln\cosh\left(\b(z^*+h_i)
\right)\right)\right)
\left(1+o(1)\right).
\eea
\end{proposition}

\begin{proof}
We only need to examine \eqv(fine.1.1) at a critical
  point $\bz^*$.
The equation for the prefactor follows by combining
  \eqv(fine.3) with \eqv(fine.13).
As for the exponential term, $F_{\b,N}$, notice that by convex
duality
\be\Eq(fine.17)
I_{N,\ell}(\bz^*_\ell/\rho_\ell)=t^*_\ell
  \bz^*_\ell/\rho_\ell-U_{N,\ell}(t^*_\ell)
=\b(z^*+\bar h_\ell)
\bz^*_\ell/\rho_\ell-U_{N,\ell}\left(\b(z^*+\bar
  h_\ell)\right).
\ee
Hence \eqv(fine.16) equals
\bea \Eq(fine.18)
\nonumber &&-\frac
12 \left(z^*\right)^2 -\sum_{\ell=1}^n \bz^*_\ell \bar h_\ell +\frac
1\b \sum_{\ell=1}^n \left[\rho_\ell\b(z^*+\bar h_\ell)
  \bz^*_\ell/\rho_\ell-
\rho_\ell U_{N,\ell}\left(\b(z^*+\bar
  h_\ell)\right)\right]
\\\nonumber
&&=-\frac 12 \left(z^*\right)^2 -\sum_{\ell=1}^n \left[\bz^*_\ell
\bar h_\ell - z^*\bz^*_\ell -\bar h\bz^*_\ell + \frac 1{\b
N}\sum_{i\in \L_\ell}\ln\cosh\left(\b(z^*+
  h_i)\right)\right]\\
&&=\frac 12 \left(z^*\right)^2 - \frac 1{\b N}\sum_{i\in
\L}\ln\cosh\left(\b(z^*+
  h_i)\right).
\eea
\end{proof}

\begin{remark}
The form given in Proposition \thv(fine.15) is highly suitable for
our purposes as the dependence on $n$ appears only in the
denominator of the prefactor. We will see that this is just
what we need to get a formula for capacities that is independent of
the choice of the  partition of $I$ and has a limit as $n\uparrow \infty$.
\end{remark}

\paragraph{Eigenvalues of the Hessian}.
We now describe the eigenvalues of the Hessian matrix  $\A(\bz^*)$.
\begin{lemma}\TH(fine.500)
Let $z^*$ be a solution of the equation \eqv(fine.10). Assume in
addition that all numbers $\hat\l_k$ are distinct. Then $\g$ is an
eigenvalue of $\A(\bz^*)$ if and only if it is a solution of the
equation
\be\Eq(fine.501)
 \sum_{\ell=1}^n \frac
1{\frac{1}{\frac{\b}N\sum_{i\in\L_{\ell}}
\left(1-\tanh^2\left(\b\left(z^*+h_i\right)\right)\right)}-\g}=1.
\ee
 Moreover, \eqv(fine.501) has at most one negative solution, and
it has such a negative solution if and only if
\be\Eq(fine.502)
\frac{\b}N\sum_{i=1}^N
\left(1-\tanh^2\left(\b\left(z^*+h_i\right)\right)\right)
>1.
\ee
\end{lemma}

\begin{remark}
To analyze the case when some $\hat\l_k$ coincide is
  also not difficult. See \cite{BEGK01}.
\end{remark}

\begin{proof}
To find the eigenvalues of $\A$, just replace  $\hat\l_k$
  by $\hat\l_k-\g$ in the first line of \eqv(fine.13). This gives
\be\Eq(fine.503)
\det\left(\A(\bz^*)-\g\right))=
 \left(1-\sum_{\ell=1}^n
\frac 1{\hat\l_\ell-\g}\right)\prod_{\ell=1}^n (\hat\l_\ell-\g),
\ee
provided none of the $\hat\l_\ell-\g=0$. \eqv(fine.501) is then just
the demand that the first factor on the right of \eqv(fine.503)
vanishes. It is easy to see that, under the hypothesis of the lemma,
this equation has $n$ solutions, and that exactly one of them is
negative under the hypothesis \eqv(fine.502).
\end{proof}

\paragraph{Topology of the landscape.} From the analysis of the
critical points of $F_{\b,N}$ it follows that the landscape of this
function is closely slaved to the one-dimensional landscape
described in Section 1. We collect the following features:
\begin{itemize}
\item[(i)] Let $m_1^*<z_1^*<m_2^*<z_2^*<\dots <z_k^*<m_{k+1}^*$
be the sequence of minima resp. maxima of the one-dimensional
function $F_{\b,N}$ defined in \eqv(static.116). Then to each
minimum, $m^*_i$, corresponds a minimum, $\bm_i^*$ of $F_{\b,N}$, such
that $\sum_{\ell=1}^n \bm_{i,\ell}^* =m_i^*$, and two each maximum,
$z^*_i$, corresponds a saddle point, $\bz_i^*$  of $F_{\b,N}$, such
that $\sum_{\ell=1}^n \bz_{i,\ell}^* =z_i^*$.
\item[(ii)] For any value $m$ of the total magnetization,
the function $F_{\b,N}(\bx)$ takes its relative minimum on the set
$\{\by :\sum \by_\ell =m\}$ at the point $\hat\bx\in\R^n$  determined
 (coordinate-wise)
 by the equation
\be\Eq(fine.505)
\hat\bx_\ell ( m) =\frac
1N\sum_{i\in\L_\ell}\tanh\left(\b\left(m+a+h_i\right)\right),
\ee
where $a = a (m)$ is recovered from
\be\Eq(fine.506)
 m=\frac
1N\sum_{i\in\L}\tanh\left(\b\left(m+a+h_i\right)\right). \ee
Moreover,
 \be\Eq(fine.507)
F_{\b,N}(m)
\leq\ F_{\b,N}(\hat\bx) \leq F_{\b,N}(m) + O(n\ln N/N).
 \ee
\end{itemize}
\begin{remark}
Note that the minimal energy curves $\hat\bx(\cdot )$ defined by
\eqref{fine.505}
pass through
the minima and saddle points, but  are in general not the integral
curves of the gradient flow connecting them.
 Note also that since we assume that
random fields $\lbr h_i (\o )\rbr$ have bounded support, for every
$\delta >0$  there exist two universal
constants $0<c_1\leq c_2<\infty$, such that
\be
\label{eq:mincurve}
c_1\rho_\ell\leq \frac{{\rm d}\hat\bx_\ell (m)}{{\rm d}m}\leq c_2\rho_\ell ,
\ee
uniformly in $N$, $m\in [-1+\delta ,1-\delta]$ and in $\ell =1, \dots ,n$.
\end{remark}

\begin{figure}\label{fig.0}
\begin{center}
\psfrag {m1}{$m_1^*$} \psfrag {m2}{$m_2^*$} \psfrag {z1}{$z_1^*$}
\psfrag {x1}{$m_1^*$} \psfrag {x2}{$m_2^*$} \psfrag {x3}{$z_1^*$}
\includegraphics[width=5cm]{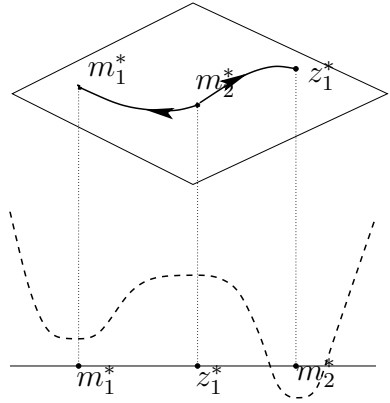}
\end{center}
\caption{Correspondence of one and $n$-dimensional landscape}
\end{figure}

\section{Upper bounds on capacities}

This and the next section are devoted to proving Theorem \thv(CAP-thm).
In this section we derive upper bounds on
capacities between two local minima. The
procedure to obtain these bounds has two steps. First, we show that
using test functions that only depend on the block variables
$\bm(\s)$, we can always get upper bounds in terms of a finite
dimensional Dirichlet form. Second, we produce a good test function
for this Dirichlet form.

\subsection{First blocking.}
Let us consider two sets, $A,B\subset \cS_N$, that are defined in
terms of
block variables $\vm$.
This means that for some
$\bA ,\bB\subseteq \G_N^n$,  $A = \cS_N [\bA ]$ and $B = \cS_N [\bB ]$.
Later we will be  interested in pre-images
of two minima of the function $F_{\b,N}$. We get the obvious
upper bound
\begin{eqnarray}\Eq(rfcw.10)
\capa(A,B)&=&\inf_{h\in \cH_{A,B}}\frac 12 \sum_{\s,\s'\in \cS_N}
\mu_{\b,N}[\o](\s) p(\s,\s')\left[h(\s)-h(\s')\right]^2
\nonumber\\
& \leq& \inf_{u\in \cG_{\bA, \bB}} \frac 12 \sum_{\s,\s'\in \cS_N}
\mu_{\b,N}[\o](\s)
p(\s,\s')\left[u(\vm(\s))-u(\vm(\s'))\right]^2
\nonumber\\
&=&\inf_{u\in \cG_{\bA ,\bB}} \sum_{\bx,\bx'\in \G_N^n}
\left[u(\bx)-u(\bx')\right]^2 \sum_{\s\in
\cS_N [\bx ]}\mu_{\b,N}[\o](\s)
 \sum_{\s'\in\cS_N [\bx' ]} p(\s,\s')
\nonumber\\
&\equiv& \inf_{u\in \cG_{\bA ,\bB}} \sum_{\bx,\bx'\in
\G_N^n}\cQ_{\b,N}[\o](\bx) r_N(\bx,\bx')\left[u(\bx)-u(\bx')\right]^2
\nonumber\\
&\equiv&
\Capm{\bA ,\bB} .
\end{eqnarray}
with
\be\Eq(rfcw.11)
 r_N(\bx,\bx')\equiv \frac
1{\cQ_{\b,N}[\o](\bx)}
 \sum_{\s\in\cS_N [\bx ]}
\mu_{\b,N}[\o](\s)
 \sum_{\s'\in\cS_N [ \bx' ]} p(\s,\s').
\ee
Here
\be\Eq(up.1)
\cH_{A,B}\equiv \{h:S_N\rightarrow [0,1]:
\forall \s\in A, h(\s)=1,\forall \s\in B, h(\s)=0\}
\ee
 and
\be\Eq(up.1.1)
\cG_{\bA ,\bB}\equiv \{u:\G_N^n\rightarrow [0,1]:
\forall \bx\in
\bA,~
 u(\bx)=1,\forall \bx\in
\bB,~ u(\bx)=0\}.
\ee

\subsection {Sharp upper bounds for saddle point crossings}

Let now $\bz^*$ be a saddle point, i.e. a critical point of
$\cQ_{\b,N}$ such that the matrix $\A(\bz^*)$ has  exactly one
negative eigenvalue and that all its other eigenvalues are strictly
positive. Let $\bA, \bB$  be two disjoint
neighborhoods of
 minima of $F_{\b,N}$ that are connected through $\bz^*$,
i.e.
$\bA$ and $\bB$  are strictly contained in two different
connected components of the level set $\{\bx:
F_{\b,N}(\bx)<F_{\b,N}(\bz^*)\}$, and there exists a path $\g$
from $\bA$ to $\bB$ such that  $\max_{\bx\in
  \g}F_{\b,N}(\bx)=F_{\b,N}(\bz^*)$.

To estimate such capacities it suffices to compute the
capacity of some small set near the saddle point (see e.g.
\cite{B04} or \cite{BEGK03} for an explanation). For a given
(small) constant $\rho = \rho (N) \ll 1
$,  we define
 \be\Eq(up.71)
D_N(\rho)\equiv\{\bx\in \G_N^n: |\bz^*_\ell-\bx_\ell|\leq \rho,\forall
1\leq \ell\leq n\},
\ee
In this section we will later  choose $\rho= C\sqrt{\ln N/N}$, with
$C<\infty$.
 $D_N(\rho)$
is the hypercube in $\G_N^n$ centered in $\bz^*$ with sidelenght $2\rho$.
For a fixed
vector, $\bv\in \G_N^n$,
consider  three disjoint subsets,
 \bea
\Eq(up.71.1)
  W_0&=&\{\bx\in \G_N^n: |(\bv,(\bx-\bz^*))|<\rho\}
\nonumber\\
   W_1&=&\{\bx\in \G_N^n: (\bv,(\bx-\bz^*))\leq -\rho\}
\nonumber \\
   W_2&=&\{\bx\in \G_N^n: (\bv,(\bx-\bz^*))\geq
\rho\}.
\eea
 We will compute the capacity of the Dirichlet
form restricted to the set $D_N(\rho)$ with boundary conditions zero
and one, respectively, on the sets $W_1\cap D_N(\rho)$ and $W_2\cap
D_N(\rho)$. This will be done by exhibiting an approximately
harmonic function with these boundary conditions. Before doing this,
it will however be useful to slightly simplify the Dirichlet form
we have to work with.

\paragraph{Cleaning of the Dirichlet form.}
One problem we are faced with in our setting is that the transition
rates $r_N(\bx,\bx')$ are given in a somewhat unpleasant form. At
the same time it would be nicer to be able to replace the measure
$\cQ_{\b,N}$ by the approximation given in \eqv(fine.7). That we are
allowed to do this follows from the simple assertion below, that is
an immediate consequence of the positivity of the terms in the
Dirichlet form, and of the Dirichlet principle.

\begin{lemma} \TH(trivial.1)
Let $\Phi_N, \wt\Phi_N$ be two Dirichlet forms defined on the same
space, $\G$, corresponding to the measure $\cQ$ and transition rates
$r$, respectively $\wt \cQ$ and $\tilde r$. Assume that, for all
$\bx,\bx'\in \G$,
\be\Eq(trivial.2)
\left|\frac{\cQ(\bx)}{\wt\cQ(\bx)}-1\right|
\leq
\d,\quad, \left|\frac{r(\bx,\bx')}{\wt r(\bx,\bx')}-1\right|\leq \d.
\ee
Then for any sets $\bA, \bB$
\be\Eq(trivial.3)
(1-\d)^2\leq \frac{\Capm{\bA ,\bB}}{\wCapm{\bA ,\bB }}\leq (1-\d)^{-2}.
 \ee
\end{lemma}


\begin{proof}
Note that $\Capm{\bA ,\bB } \equiv \inf_{u\in
\cG_{\bA ,\bB }}\Phi_N(u)=\Phi_N(u^*)$, and\newline $\wCapm{\bA, \bB} \equiv
\inf_{u\in
  \cG_{\bA,\bB}}\wt\Phi_N(u)=\wt\Phi_N(\tilde u^*)$.
But clearly
\bea\Eq(trivial.4)
 \Phi_N(u^*) &=& \frac 12
\sum_{\bx,\bx'\in \G} \wt \cQ(\bx) \frac {\cQ(\bx)}{\wt\cQ(\bx)}
\tilde r(\bx,\bx')\frac {r(\bx,\bx'))}{\wt r(\bx,\bx')}
\left(u^*(\bx)-u^*(\bx')\right)
\\\nonumber
&\geq&  \frac 12
\sum_{\bx,\bx'\in \G} \wt \cQ(\bx) (1-\d) \tilde r(\bx,\bx')(1-\d)
\left(u^*(\bx)-u^*(\bx')\right)
\\\nonumber
&\geq&  (1-\d)^2 \inf_{u\in\cG_{\bA ,\bB } }\frac 12
\sum_{\bx,\bx'\in \G} \wt \cQ(\bx)
\tilde r(\bx,\bx') \left(u(\bx)-u(\bx')\right)
\\\nonumber
&=&(1-\d)^2 \wCapm{\bA ,\bB }.
\eea
By the same token,
\bea\Eq(trivial.5)
 \wt\Phi_N( u^*) &\geq& (1-\d)^2 \Capm{\bA ,\bB }.
\eea
 The claimed relation follows.
\end{proof}

To make use of this observation, we need to control the rates
$r_N(\bx,\bx')$ and the measure
$\cQ_{\b,N}(\bx)$ in terms of suitable modified rates and measures.
In fact, we see easily that
\be\Eq(trivial.7)
\wt\cQ_{\b,N}(\bx)\equiv \cQ_{\b,N}(\bz^*)
\exp\left(-\frac{\b N}2((\bx-\bz^*), \A(\bz^*)(\bx-\bz^*))\right),
\ee
so that, for all $\bx\in D_N(\rho)$ and  for some $K<\infty$, it holds
\be\Eq(trivial.7.1)
\left|\frac{\cQ_{\b,N}(\bx)}{\wt\cQ_{\b,N}(\bx)}-1\right|\leq K N\rho^3.
\ee
For that concerns the rates, let us first define, for $\s\in\cS_N$,
\be\Eq(put.20)
 \Lambda^{\pm}_k (\s )\equiv  \lbr i\in \Lambda_k~:~\s (i) =\pm1\rbr .
\ee
For all $\bx\in \G_N^n$,
we then have
\bea\Eq(bounded.1)
r_N(\bx,\bx+\bfe_\ell)&=& \cQ_{\b,N}(\bx)^{-1}
\sum_{\s\in\cS_N[\bx]}
\mu_{\b ,N}[\o ] (\s )
\sum_{i\in \Lambda^{-}_\ell (\s )} p(\s,\s^i)\\
&=&\cQ_{\b,N}(\bx)^{-1}
\sum_{\s\in\cS_N[\bx]}
\mu_{\b ,N}[\o ] (\s )
\sum_{i\in \Lambda^{-}_\ell (\s )}
\sfrac{1}{N} e^{-2\b\left[m(\s)-\sfrac{1}{N} +h_i\right]_+}.\nonumber
\eea
Notice that for all $\s\in\cS_N(\bx)$,
$|\Lambda^{-}_\ell (\s)|$
is a constant just depending on $\bx$.
Using that $h_i= \bar h_\ell + \wt h_i$, with $\wt h_i\in[-\e,\e]$,
we get the bounds
\be\Eq(bounded.2)
 r_N(\bx,\bx+\bfe_\ell)=
\sfrac{\left|\Lambda^{-}_\ell (\bx )\right|}{N}
e^{-2\b\left[m(\s) +\bar h_\ell\right]_+} (1+O(\e)) .
\ee
It follows  easily  that, for all $\bx\in D_N(\rho)$,
\be \label{trivial.6}
\left|\frac{r_N(\bx,\bx+\bfe_\ell)}{r_N(\bz^*,\bz^*+\bfe_\ell)}
-1\right| \leq \b (\e+n\rho)
\ee

With this in mind, we let  $\wt L_{N}$ be the generator
of the dynamics on $D_N(\rho)$ with rates
$\wt r(\bx, \bx +\bfe_\ell)\equiv r_N(\bz^*,\bz^*+\bfe_\ell)\equiv r_\ell$
 and $\wt r (\bx +\bfe_\ell, \bx)\equiv r_\ell\frac{\wt\cQ_{\b,N}(\bx)}{\wt\cQ_{\b,N}(\bx+\bfe_\ell)}$,
 and thus with reversible measure
 $\wt\cQ_{\b,N}(\bx)$. 
 For $u\in \cG_{\bA ,\bB}$, we write the corresponding Dirichlet form as
\be\Eq(trivial.10)
\wt\Phi_{D_N}(u)\equiv \cQ_{\b,N}(\bz^*)\sum_{\bx\in
  D_N(\rho)}\sum_{\ell=1}^nr_\ell
e^{-\b N((\bx-\bz^*),\A(\bz^*)(\bx-\bz^*))}
\left(u(\bx)-u(\bx+\bfe_\ell)\right)^2.
\ee

\subsection{Approximately harmonic functions for $\wt\Phi_{D_N}$}

 We will now describe a function that we will show to be almost
 harmonic with respect to the Dirichlet form $\wt \Phi_{D_N}$.
Define the matrix $\B(\bz^*)\equiv \B$ with elements
\be\Eq(up.20)
\B_{\ell,k}\equiv \sqrt {r_\ell}\A(\bz^*)_{\ell,k}\sqrt{r_k}.
\ee
Let $\hat \bv^{(i)}$, $i=1,\dots,n$ be the normalized eigenvectors
of $\B$, and $\hat\g_i$ be the corresponding  eigenvalues.
We denote by $\hat \g_1$ the unique negative eigenvalue of $\B$,
and characterize it in the following lemma.

\begin{lemma}\TH(up.83)
Let $z^*$ be a solution of the equation \eqv(fine.10) and assume in
addition that
\be\Eq(up.84)
\frac{\b}N\sum_{i=1}^N
\left(1-\tanh^2\left(\b\left(z^*+h_i\right)\right)\right)
>1.
\ee
 Then,
$\bz^*$ defined through \eqv(fine.9)
 is a saddle point and the unique negative eigenvalue of
$\B(\bz^*)$ is the unique negative solution, $\hat\g_1\equiv
\hat\g_1(N,n)$, of the equation
\be\Eq(up.85)
\sum_{\ell=1}^n\rho_\ell \frac{
\frac{1}{|\L_\ell|}
\sum_{i\in\L_\ell}\left(1-\tanh(\b(z^*+h_i))\right)
\exp{(-2\b\left[z^* +\bar h_\ell\right]_+)}}
{\frac{\frac1{|\L_\ell|}
\sum_{i\in\L_\ell}\left(1-\tanh(\b(z^*+h_i))\right)
\exp{(-2\b\left[z^* +
\bar h_\ell\right]_+)}}
{\frac{\b}{|\L_\ell|}\sum_{i\in\L_\ell}\left(1-\tanh^2(\b(z^*+h_i))
\right) }-2\g}=1.
\ee
Moreover, we have that
\be\Eq(up.86)
\lim_{n\uparrow\infty}\lim_{N\uparrow\infty} \hat\g_1(N,n) \equiv
\bar\g_1,
\ee
where $\bar\g_1$ is the unique negative solution of the equation
\be\Eq(up.87)
\E_h\left[\frac{ \left(1-\tanh(\b(z^*+h))\right)\exp{(-2\b\left[z^* +
 h\right]_+)} }
{\frac{\exp{(-2\b\left[z^* +
 h\right]_+)}}{ \b\left(1+\tanh(\b(z^*+h)) \right)}-2\g}\right]=1.
\ee
\end{lemma}

\begin{proof}
The particular form of the matrix $\B$ allows to obtain a
simple characterization of all eigenvalues and
eigenvectors. The eigenvalue equations can be written as
\be\Eq(up.88)
 -\sum_{\ell=1}^n\sqrt{r_\ell r_k}
u_\ell+(r_k\hat\l_k-\g)u_k=0,\forall 1\leq k\leq n.
\ee
Assume for simplicity that all $r_k\hat\l_k$ take distinct
values.
Then  there is no non-trivial solution of these equation
with $\g =r_k\hat\l_k$, and thus $\sum_{\ell=1}^n\sqrt{r_\ell}
u_\ell\neq 0$. Thus,
\be\Eq(Up.89)
u_k=\frac{\sqrt{r_k}\sum_{\ell=1}^n\sqrt{r_\ell}u_\ell}{r_k\hat\l_k-\g}.
\ee
 Multiplying by $\sqrt{r_k}$ and summing over $k$, $u_k$ is a
solution if and only if $\g$ satisfies the equation
\be\Eq(up.90)
\sum_{k=1}^n \frac{r_k}{r_k\hat\l_k-\g}=1.
\ee
Using \eqv(bounded.2) and noticing that
$\frac{|\L_k^-|}N=\sfrac{1}{2}(\rho_k-\bz^*_k)$, we get
\be\Eq(bounded.3)
r_k=\sfrac{1}{2}(\rho_k-\bz^*_k) \exp\left({-2\b\left[m(\s) +
\bar h_k\right]_+}\right)(1+O(\e)).
\ee
Inserting the expressions  for $\bz^*_k/\rho_k$ and  $\hat\l_k$ given
by \eqv(fine.9) and \eqv(fine.12) into \eqv(bounded.3)
and substituting the result into \eqv(up.90),
we recover \eqv(up.85).

Since the left-hand side of \eqv(up.90) is monotone decreasing in $\g$
as long as $\g\geq 0$, it follows that there can be at most one
negative solution of this equation, and such a solution exists if
 and only if left-hand
side is larger than $1$ for $\g=0$.
The claimed convergence property \eqv(up.86) follows easily.
\end{proof}

We continue our construction  defining the vectors $\bv^{(i)}$ by
\be\Eq(up.31)
\bv_\ell^{(i)} \equiv\hat \bv^{(i)}_\ell /\sqrt {r_\ell},
 \ee
and the vectors ${\boldsymbol {\check v}}^{(i)}$ by
\be\Eq(up.31.1)
\checkv_\ell^{(i)} \equiv\hat \bv^{(i)}_\ell \sqrt
{r_\ell} =r_\ell \bv^{(i)}_\ell.
 \ee
We will single out the vectors $\bv\equiv \bv^{(1)}$ and
$\checkv\equiv\checkv^{(1)}$. The important facts about these vectors
is that
\be\Eq(up.31.2)
\A\check \bv^{(i)}=\hat\g_i\bv^{(i)},
\ee
and that
\be\Eq(up.31.3)
(\checkv^{(i)},\bv^{(j)})=\d_{ij}.
\ee
This implies the following non-orthogonal
decomposition of the quadratic form $\A$,
\be\Eq(up.31.5)
(\by,\A\bx)=\sum_{i=1}^n\hat\g_i (\by,\bv^{(i)})(\bx,\bv^{(i)}).
\ee

A consequence of the computation in the proof of Lemma \thv(up.83),
on which we shall rely in the sequel, is the following:
\begin{lemma}
 \label{lem:ventries}
There exists a positive constant $\delta >0$ such that independently of $n$,
\be
\label{eq:ventries}
\delta \leq  \min_k \bv_k \leq\max_k \bv_k \leq \frac{1}{\delta} .
\ee
\end{lemma}
\begin{proof}
 Due to our explicit computations,
\be\label{dimaputnumbers.1}
 r_k\hat\l_k = \frac{1}{2} \lb 1-\frac{\bz_k^*}{\rho_k}\rb
 \left[ \b \frac{1}{\left|\Lambda_k\right|}
\sum_{i\in\Lambda_k} \lb 1- \tanh^2\lb \b(z^* +h_i )\rb\rb\right]^{-1}
 e^{-2\b\left[z^* +\bar h_k\right]_+} .
\ee
Consequently, the quantities $\phi_k\equiv r_k\hat\l_k -\hat{\gamma}_1
(N ,n)$ are bounded away from zero and infinity,  uniformly
in $N$, $n$ and  $k=1, \dots ,n$. Since
by \eqref{up.31} and \eqref{Up.89} the entries
 of $\bv$ are given by
\be\label{dimaputnumbers.2}
 \bv_k = \frac{1}{\phi_k}\lbr \sum_\ell \frac{r_\ell}{\phi_\ell^2}\rbr^{-1/2} ,
\ee
the assertion of the lemma follows.
\end{proof}

Finally, define the function $f:\R\rightarrow \R_+$ by
\bea\Eq(up.32)
f(a)&=&\frac{\int_{-\infty}^a e^{-\b N|\hat\g_1|u^2/2}du}
{\int_{-\infty}^{\infty} e^{-\b N|\hat\g_1|u^2/2}du}
\\\nonumber
&=&  \sqrt{\frac{\b N|\hat\g_1|}{2\pi}}  \int_{-\infty}^a
  e^{-\b N|\hat\g_1|u^2/2}du.
\eea
We claim that the function
\be\Eq(up.33)
g(\bx)\equiv f((\bv,\bx))
 \ee
 is the desired approximately harmonic function.

Notice first, that $g(\bx)=\po(1)$ for all $\bx\in W_1\cap D_N(\rho)$,
while $g(\bx)=1-\po(1)$ for all $\bx\in W_2\cap D_N(\rho)$. Moreover,
 the following holds:
\begin{lemma}\TH(up.34)
Let $g$ be defined in \eqv(up.33). Then, for all $\bx\in D_N(\rho)$,
there exists a constant $c< \infty$ such that
\be\Eq(up.35)
\left|\left(\wt L_{N}g\right)(\bx) \right|\leq
\left(\sqrt{\frac{\b|\hat\g_1|}{2\pi N}} e^
{-\b N|\hat\g_1|(\bx,\bv)^2/2}
\sum_{\ell=1}^n r_\ell {\bv_\ell}\right) c\rho^2.
 \ee
\end{lemma}

\begin{remark}
The point of the estimate \eqv(up.35) is that it is by
  a factor $\rho^2$ smaller than what we would get for an arbitrary
  choice of the parameters $\bv$ and $\g_1$. We will actually use this
  estimate in the proof of the \emph{lower bound}.
\end{remark}

\begin{proof}
 To simplify the notation we will assume throughout the
 proof that coordinates are chosen such that $\bz^*=0$. We also set
 $\A\equiv \A(\bz^*)$. Using the detailed balance condition, we get
 \be\Eq(dbc.1)
\wt r(\bx, \bx-\bfe_\ell)= \frac{\wt\cQ_{\b,N}(\bx-\bfe_\ell)}
{\wt\cQ_{\b,N}(\bx)}\wt r(\bx-\bfe_\ell, \bx)= \frac{\wt\cQ_{\b,N}
(\bx-\bfe_\ell)}{\wt\cQ_{\b,N}(\bx)}r_\ell.
 \ee
Moreover, from the definition of $\wt\cQ_{\b,N}$ and using that we are
near a critical point, we have that
\bea\Eq(up.41)
\frac{\wt\cQ_{\b,N}(\bx-\bfe_\ell)}{\wt\cQ_{\b,N}(\bx)}&=&
\exp\left(-\sfrac{\b N}2\bigl[\bigl(\bx,\A \bx \bigr)-
\bigl((\bx-\bfe_\ell), \A(\bx-\bfe_\ell)\bigr)\bigr]\right)
\\\nonumber
&=&\exp\left(-\b \bigl(\bfe_\ell, \A \bx\bigr)
\right)\left(1+O\left(N^{-1}\right)\right).
\eea
From \eqv(dbc.1) and \eqv(up.41), the generator can be written as
\bea\Eq(trivial.20)
\left(\wt L_{N}g\right)(\bx) &=&\sum_{\ell=1}^n r_\ell
\left(g(\bx+\bfe_\ell)-g(\bx)\right)
\\
&& \quad\quad\times \left( 1-\exp\left(-\b \bigl(\bfe_\ell, \A \bx\bigr) \right)
\frac{g(\bx)-g(\bx-\bfe_\ell)}{g(\bx+\bfe_\ell)-g(\bx)}
\left(1+O(N^{-1})\right)\right).\nonumber
\eea
Now we use the explicit form of $g$ to obtain
\bea\Eq(up.50)
g(\bx+\bfe_\ell)-g(\bx) &=&
f((\bx,\bv)+\bv_\ell/N)-f((\bx,\bv)
\\\nonumber
&=&f'((\bx,\bv)) \bv_\ell/N +\bv_\ell^2N^{-2} f''(\bx,\bv)/2
+\bv_\ell^3N^{-3}f'''((\tilde \bx,\bv))/6
\\\nonumber
&=&{\bv_\ell}\sqrt{\frac{\b |\hat\g_1|}{2\pi N}} e^{-\b
  N|\hat\g_1|(\bx,\bv)^2/2}
\left(1 -\bv_\ell {\b |\hat\g_1|(\bx,\bv)}/2 +O\left(\rho^2\right)
\right).
\eea
 In particular, we get from here that
 \bea\Eq(up.51)
\frac{g(\bx)-g(\bx-\bfe_\ell)}{g(\bx+\bfe_\ell)-g(\bx)}
&=&\exp\left({-\b N|\hat\g_1|\left[(\bx-\bfe_\ell,\bv)^2-
(\bx,\bv)^2\right]/2}\right)
\\\nonumber
&&\times \frac {1-\bv_\ell \b |\hat\g_1|[(\bx,\bv)
-\bv_\ell/N]/2+O\left(
\rho^2
\right)} {1-\bv_\ell \b
|\hat\g_1|(\bx,\bv)/2+O\left(
\rho^2
\right)}
\\\nonumber
&=& \exp\left({-\b |\hat\g_1| \bv_\ell(\bx,\bv)}\right) \left(1+
\frac{\bv_\ell^2 \b |\hat\g_1|/2N+O\left(
\rho^2
\right)}
{1-\bv_\ell \b |\hat\g_1|(\bx,\bv) +O\left(
\rho^2
\right)}\right)
\\\nonumber
&=& \exp\left({-\b |\hat\g_1| \bv_\ell(\bx,\bv)}\right) \left(1+
O(\rho^2)\right)
 \eea
 Let us now insert these equations into \eqv(trivial.20):
\bea\Eq(up.55)
\left(\wt L_{N}g\right)(\bx) &=&\sqrt{\frac{\b |\hat\g_1|}{2\pi
N}} e^{-\b  N|\hat\g_1|(\bx,\bv)^2/2}\sum_{\ell=1}^n
r_\ell {\bv_\ell} \left(1 -\bv_\ell \b |\hat\g_1|(\bx,\bv)/2
+O\left(\rho^2\right) \right).
\nonumber\\
&&\times
\left( 1-\exp\left\{-\b \bigl(\bfe_\ell, \A \bx\bigr)-\b |\hat\g_1|
\bv_\ell(\bx,\bv)\right\} \left(1+ O(\rho^2)\right)
   \right).
\eea
Now
\bea\Eq(up.57)
&&1-\exp\left(-\b \bigl(\bfe_\ell, \A
\bx\bigr)-\b |\hat\g_1| \bv_\ell(\bx,\bv)\right) \left(1+
O(\rho^2)\right)
\nonumber\\
&&\quad\quad\quad\quad=\b \bigl(\bfe_\ell, \A \bx\bigr) {+\b |\hat\g_1|
\bv_\ell(\bx,\bv)} +O(\rho^2).
\eea
Using this fact, and collecting
the leading order terms, we get
 \bea\Eq(up.58)
\left(\wt L_{N}g\right)(\bx)&=&\sqrt{\frac{\b |\hat\g_1|}{2\pi N}} e^{-\b
  N|\hat\g_1|(\bx,\bv)^2/2}
\nonumber\\
&&\times\sum_{\ell=1}^n
r_\ell {\bv_\ell} \left[ \left(\b \bigl(\bfe_\ell, \A \bx\bigr)
{+\b|\hat\g_1| \bv_\ell(\bx,\bv)}\right)+O(\rho^2)\right].
\eea
Thus we will have proved the lemma provided that
\be\Eq(up.60)
 \sum_{\ell=1}^n r_\ell {\bv_\ell} \left(\bigl(\bfe_\ell,\A\bx\bigr)
{- \hat\g_1\bv_\ell(\bx,\bv)}\right)=0.
\ee
But note that from \eqv(up.31.5) we
get that
\be\Eq(up.61)
\bigl(\bfe_\ell, \A \bx\bigr) -\hat\g_1
\bv_\ell(\bx,\bv) = \sum_{j=2}^n \hat\g_j \bv_\ell^{(j)}
(\bx,\bv^{(j)}).
 \ee
Hence using that by \eqv(up.31.1) $r_\ell\bv_\ell=\check\bv_\ell$
and that by \eqv(up.31.3) $\check\bv$ is orthogonal to $\bv^{(j)}$
with $j\geq 2$, \eqv(up.60) follows and the lemma is proven.
\end{proof}

Having established that  $g$ is a good approximation of the
equilibrium potential in a neighborhood of $\bz^*$, we can now use it
to compute a good upper bound for the capacity.
Fix now $\rho= C\sqrt{\ln N/N}$.

\begin{proposition}\TH(up.main)
With the notation introduced above and for every $n\in\N$, we get
\be\Eq(up.70)
\capa( A , B )\leq \cQ_{\b,N}(\bz^*) \frac{\b |\hat\g_1|}{2\pi N}
 \left(\frac{\pi N }{2\b}\right)^{n/2}
 \prod_{\ell=1}^n \sqrt{\frac{r_\ell}{|\hat\g_j|}}
\left(1+O(\e+\sqrt{ (\ln N)^3/N})\right).
\ee
\end{proposition}

\begin{proof}
The upper bound on $\capa( A , B )$ is inherited from the upper bound
on the mesoscopic capacity $\Capm{\bA ,\bB}$.  As for the latter,
we first estimate the  energy of the mesoscopic neighborhood
$D_N\equiv D_N(\rho)$ of the  saddle point $\bz^*$.
By Lemma \thv(trivial.1), this can be controlled
in terms of the  modified Dirichlet form
$\wt\Phi_{D_N}$ in \eqv(trivial.10). Thus, let $g$ the function defined
in \eqv(up.33) and choose coordinates such that $\bz^*=0$.
Then
\bea\Eq(up.72)
\wt\Phi_{D_N}(g)&\equiv&
\wt\cQ_{\b,N}(0) \sum_{\bx\in D_N}\sum_{\ell=1}^n
e^{-\b N((\bx, \A\bx))/2}  r_\ell
\left(g(\bx+\bfe_\ell)-g(\bx)\right)^2
\\\nonumber
&=& \wt\cQ_{\b,N}(0) {\frac{\b |\hat\g_1|}{2\pi N}} \sum_{\bx\in
D_N} e^{-\b N|\hat\g_1|(\bx,\bv)^2}    e^{-\b N((\bx, \A \bx))/2}
\sum_{\ell=1}^n {r_\ell \bv_\ell^2}
\\\nonumber
&&\quad\quad\times
\left(1 -\bv_\ell {\b |\hat\g_1|(\bx,\bv)} +O\left(N^{-1}\ln
N\right) \right)^2
\\\nonumber
&=& \wt\cQ_{\b,N}(0) {\frac{\b
|\hat\g_1|}{2\pi N}} \sum_{\bx\in D_N} e^{-\b
  N|\hat\g_1|(\bx,\bv)^2}    e^{-\b N((\bx, \A \bx))/2}
\left(1 +O\left(\sqrt{\ln N /N}\right) \right).
\eea
 Here we used
that $\sum_\ell r_\ell \bv_\ell^2 =\sum_\ell \hat \bv_\ell^2 =1$. It
remains to compute the sum over $\bx$. By a standard approximation of
  the sum by an integral we get
\bea\Eq(up.73)
&&\sum_{\bx\in D_N} e^{-\b
  N|\hat\g_1|(\bx,\bv)^2}    e^{-\b N((\bx, \A \bx))/2}
\\\nonumber
&&=\left(\frac N2\right)^n\int d^n \bx e^{-\b
  N|\hat\g_1|(\bx,\bv)^2}    e^{-\b N((\bx, \A \bx))/2}
\left(1+O(\sqrt {\ln
  N/N} )\right)
\\\nonumber
&&=\left(\frac N2\right)^n \left( \prod_{\ell=1}^n \sqrt
{r_\ell}\right) \int d^n y e^{-\b
  N|\hat\g_1|(y,\hat v)^2}    e^{-\b N((y, \B y))/2}
\left(1+O(\sqrt {\ln N/N})\right)
\\\nonumber
&&=\left(\frac
N2\right)^n \left(\prod_{\ell=1}^n \sqrt{r_\ell}\right)
 \int d^n y e^{-\b
  N\left(|\hat\g_1|(y,\hat v)^2  +\sum_{j=1}^n\hat\g_j
(\hat \bv^{(j)}, y)^2/2\right)}
\left(1+O(\sqrt {\ln N/N})\right)
\\\nonumber
&&=\left(\frac N2\right)^n\left(\prod_{\ell=1}^n \sqrt{r_\ell}\right)
 \int d^n y e^{-\b N\sum_{j=1}^n|\hat\g_j| (\hat \bv^{(j)}, y)^2/2}
\left(1+O(\sqrt {\ln N/N})\right)
\\\nonumber
&& =\left(\frac N2\right)^n \left(\prod_{\ell=1}^n \sqrt{r_\ell}
\right) \left(\frac{2\pi }{\b  N}\right)^{n/2}
\frac 1{\sqrt{\prod_{j=1}^n |\hat\g_j|}}
\left(1+O(\sqrt{ \ln N/N})\right)
\\\nonumber
&& =\left(\frac{\pi
N}{2\b} \right)^{n/2} \prod_{\ell=1}^n \sqrt{\frac{r_\ell}{
|\hat\g_\ell|}}\left(1+O(\sqrt{ \ln N/N})\right).
\eea
 Inserting \eqv(up.73) into \eqv(up.72) we see that
the left-hand side of \eqv(up.72) is equal
to the right-hand side of \eqv(up.70) up to error terms.

It remains to show  that the contributions from the sum outside
$D_N$ in the Dirichlet form do not contribute significantly to the
capacity. To do this, we define a global test function $\wt g$ given
by
\be \Eq(up.80)
\wt g(\bx)\equiv\left\{
\begin{array}{ll}
0,& \bx\in W_1 \\
1,& \bx\in  W_2 \\
g(\bx),& \bx\in W_0
\end{array}
\right.
\ee
Clearly, the only non-zero contributions to the Dirichlet form
$ \Phi_N(\wt g)$ come from $\overline{W}_0\equiv W_0 \cup
\partial W_0$, where $\partial W_0$ denotes the boundary of $W_0$.
\begin{figure} \label{fig.1}
\begin{center}
\psfrag{a}{$W_1$} \psfrag{p}{$W_2$} \psfrag{m}{$m_1^*$}
\psfrag{n}{$m_2^*$} \psfrag{d}{$D_N$} \psfrag{wo}{$W_0^{in}$}
\psfrag{wd}{$W_0^{out}$} \psfrag{k}{$f$} \psfrag{z}{$z^*$}
\psfrag{g}{$\wt g=0$} \psfrag{b}{$\wt g=1$}
\includegraphics[width=12cm]{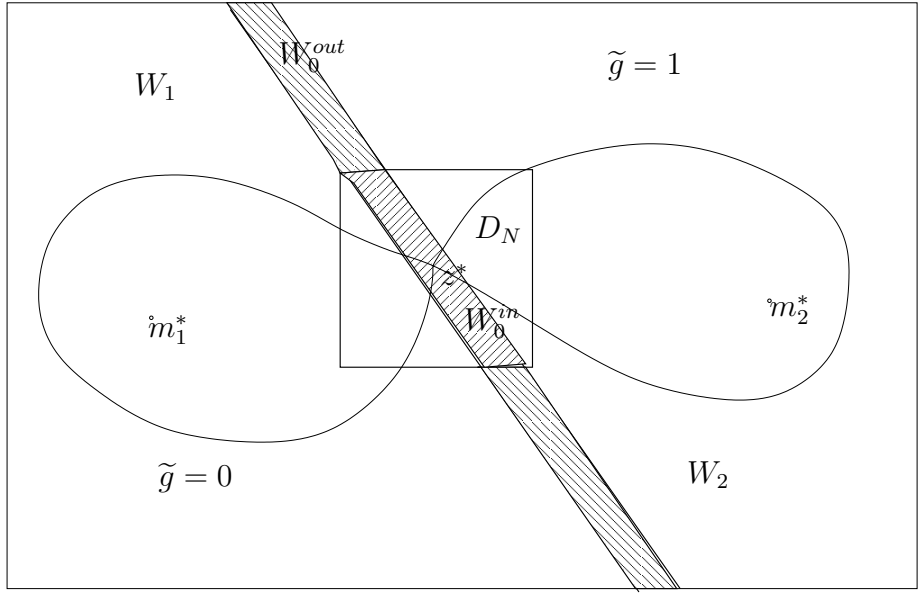}
\caption{Domains for the construction of the test function in the
  upper bound}
\end{center}
\end{figure}
 Let us
thus consider the sets  $W_0^{in}= W_0 \cap D_N$ and $W_0^{out}= W_0
\cap D_N^c$ (see Figure \ref{fig.1}). We denote by
$\Phi_{W_0^{in}}^{||}(\wt g)$ the Dirichlet form of $\wt g$
restricted to $W_0^{in}$ and to the part of its boundary contained in
$D_N$, i.e. to $\overline{W}_0^{in}\cap D_N$, and by
$\Phi_{W_0^{out}}^{\Box}(\wt g)$ the Dirichlet form of $\wt g$
restricted to $\overline{W}_0^{out}$. With this notation, we have
\bea\Eq(up.75)
 \Phi_N(\wt g)& = & \Phi_{W_0^{in}}^{||}(\wt g) +
 \Phi_{W_0^{out}}^{\Box}(\wt g)\\
 &=&\wt\Phi_{W_0^{in}}^{||}(\wt g)
 \left(1+O\left(\sqrt{\ln N/N}\right) \right) +
 \Phi_{W_0^{out}}^{\Box}(\wt g)
\nonumber\\
&=& \left(\wt\Phi_{W_0^{in}}^{||}( g)-
  \left(\wt\Phi_{W_0^{in}}^{||}(g)-\wt\Phi_{W_0^{in}}^{||}(\wt g)\right)
  \right) \left(1+O\left(\sqrt{\ln N/N}\right) \right)
 +  \Phi_{W_0^{out}}^{\Box}(\wt g).\nonumber
  \eea
The first term in \eqv(up.75) satisfies trivially the bound
\be\Eq(up.75.1)
\wt\Phi_{D_N'}(g)\leq\wt\Phi_{W_0^{in}}^{||}( g)\leq \wt\Phi_{D_N}(g),
\ee
where $D_N'\equiv D_N(\rho')$ is defined as in
\eqv(up.75.1) but with constant $\rho'=C'\sqrt{\ln N/N}$ such that
$D_N'\subset W_0^{in}$.  Performing the same computations as in
\eqv(up.72) and \eqv(up.73) it is easy to show that
$\wt\Phi_{D_N'}(g)=\wt\Phi_{D_N}( g)(1+\po(1))$,
and then from \eqv(up.75)  it follows that
\be\Eq(up.75.2)
\wt\Phi_{W_0^{in}}^{||}( g)= \wt\Phi_{D_N}(g)(1-\po(1)).
\ee
Consider now
the second term in \eqv(up.75).
Since $\wt g\equiv g$ on $W_0$, we get
\bea\Eq(up.76)
&&\wt\Phi_{W_0^{in}}^{||}(g)-\wt\Phi_{W_0^{in}}^{||}(\wt g)=
\sum_{\bx\in\partial W_0^{in}\cap W_1}
 \sum_{\ell=1}^n \wt\cQ(\bx)r_\ell
 \left[\left(g(\bx+\bfe_\ell)-g(\bx)\right)^2 -g(\bx)^2\right]
\nonumber\\
&&\quad\quad+\sum_{\bx\in\partial W_0^{in}\cap W_2}
 \sum_{\ell=1}^n \wt\cQ(\bx)r_\ell
 \left[\left(g(\bx+\bfe_\ell)-g(\bx)\right)^2
   -\left(1-g(\bx)\right)^2\right],
\eea
 where we also used that the function $\wt g$ has boundary
conditions zero and one respectively on $W_1$ and $W_2$. By
symmetry, let us just consider the first sum in the r.h.s. of
\eqv(up.76).
For   $\bx\in\partial W_0^{in}\cap W_1$ it holds that
$(\bx,\bv)\leq -\rho=-C\sqrt{\ln N/N}$, and hence
 \be\Eq(up.76.1)
g(\bx)^2 \leq \frac{1}{\sqrt{2\pi\b|\hat\g_1|}
C\sqrt{\ln N}} e^{-\b N|\hat\g_1|\rho^2}.
\ee
Using this bound together with  inequality \eqv(up.50) to
control $\left(g(\bx+\bfe_\ell)-g(\bx)\right)^2$, we get
\bea\Eq(up.77)
 && \sum_{\bx\in\partial W_0^{in}\cap W_1}
 \sum_{\ell=1}^n \wt\cQ(\bx)r_\ell \left[\left(g(\bx+\bfe_\ell)-
g(\bx)\right)^2 -g(\bx)^2\right]\hspace{5cm}
\nonumber\\
&&\quad\quad\leq\frac{\b|\hat\g_1|}{2\pi N}  e^{-\b
N|\hat\g_1|\rho^2}\sum_{\bx\in\partial W_0^{in}\cap W_1}\wt\cQ(\bx)
\left( 1+\frac{c N}{\sqrt{\ln N}}\right)
\nonumber\\
 &&\quad\quad\leq \wt\cQ_{\b,N}(0)\frac{\b|\hat\g_1|}{2\pi N}
e^{-\b N|\hat\g_1|\rho^2} \sum_{\bx\in\partial W_0^{in}\cap W_1}
 e^{-\b N((\bx, \A\bx))/2} \lb 1+c\frac{N}{\sqrt{\ln N}}\rb
\eea
for some constant $c$ independent on $N$.
The sum over $\bx\in\partial W_0^{in}\cap W_1$ in
the last term can then be computed  as in \eqv(up.73).
However, in this case the integration runs over the
$(n-1)$-dimensional hyperplane orthogonal to $v$ and thus we have
\bea\Eq(up.78)
&&\sum_{\bx\in\partial W_0^{in}\cap W_1} e^{-\b N((\bx, \A\bx))/2}
\nonumber\\
  &&=\left(\frac N2\right)^{n-1}\int d^{n-1} \bx
e^{-\b N((\bx, \A \bx))/2}
\nonumber\\
&&=\left(\frac N2\right)^{n-1} \left( \prod_{\ell=2}^n
\sqrt{r_\ell}\right)\int d^{n-1} y e^{-\b N((y, \B y))/2}
\nonumber\\
 &&\leq \left(\frac N2\right)^{n-1}
\left(\prod_{\ell=2}^n \sqrt{r_\ell}\right)
e^{-\b  N\hat\g_1\rho^2/2} \int d^{n-1} y e^{-\b N
\left(\sum_{j=2}^n\hat\g_j (\hat \bv^{(j)}, y)^2/2\right)}
\nonumber\\
&& =\left(\frac{\pi N}{2\b} \right)^{\frac{n-1}{2}}
\prod_{\ell=2}^n\sqrt{\frac{r_\ell}{ |\hat\g_\ell|}}
e^{-\b  N\hat\g_1\rho^2/2}.
\eea
Inserting \eqv(up.78) in \eqv(up.77), and comparing the result with
$\wt \Phi_{D_N}(g)$, we get that the l.h.s of \eqv(up.77) is bounded
as
\be\Eq(up.79)
\left (1+c\frac{N}{\ln N}\right)\sqrt{N} e^{-\b N|\hat\g_1|\rho^2/2}\wt
\Phi_{D_N}(g)= \po(N^{-K})
 \wt \Phi_{D_N}(g),
\ee
 with $K= \frac{\b|\hat\g_1|C-1}{2}$, which is positive if $C$ is
 large enough. A similar bound can be obtained for the second sum in
\eqv(up.76), so that we finally get
\be\Eq(up.79.1)
\left|\wt\Phi_{W_0^{in}}^{||}(g)-\wt\Phi_{W_0^{in}}^{||}(\wt g)\right|
\leq \po(N^{-K})\wt\Phi_{D_N}(g).
 \ee
The last term to analyze is the Dirichlet form
$\Phi_{W_0^{out}}^{\Box}(\wt g)$. But it is easy to realize that
this is negligible with respect to the leading
 term. Indeed, since  for all $\bx\in D_N^c$ it holds that
$F_{\b,N}(\bx)\geq F_{\b,N}(\bz^*) +K'\ln N/N$, for some positive
$K'<\infty$ depending on $C$, we get
 \be\Eq(up.79.2)%
 \Phi_{W_0^{out}}^{\Box}(\wt g)\leq Z_{\b,N}^{-1} e^{-\b N
F_{\b,N}(\bz^*)} N^{-(K'-n)}= \po(N^{-K''})\wt\Phi_{D_N}(g).
 \ee
  From \eqv(up.75) and the estimates given in \eqv(up.75.2),
 \eqv(up.79) and \eqv(up.79.2),  we get that
 $\Phi_{N}(\wt g)= \wt \Phi_{D_N}(g)(1+\po(1))$
 provides the claimed upper bound.
\end{proof}

Combining this proposition with Proposition \thv(fine.15), yields,
after some computations, the following more explicit representation
of the upper bound.

\begin{corollary}\TH(up.better) With the same notation of Proposition \thv(up.main),
\be\Eq(up.70.1)
Z_{\b,N}\capa( A , B )
\leq \frac{\b |\bar\g_1|}{2\pi N}
\frac {\exp\left(-\b N F_{\b,N}(z^*)\right)
 \left(1+\po(1)\right) }{\sqrt{
\b N\E_h
\left(1-\tanh^2\left(\b\left(z^*+h\right)\right)\right)
-1}},
\ee
where $\bar\g_1$ is defined through Eq. \eqv(up.87).

\end{corollary}

\begin{proof}
First, we want  to show that
\be \Eq(up.81)
|\det (\A(\bz^*))|=\left(\prod_{\ell=1}^n r_\ell \right)^{-1}
\prod_{\ell=1}^n \hat\g_\ell.
 \ee
To see this, note that
$$
\B=R\A(\bz^*)R,
$$
where $R$ is the diagonal matrix with elements
$R_{\ell,k}=\d_{k,\ell} \sqrt {r_\ell}$. Thus
\be\Eq(up.82)
\prod_{\ell=1}^n |\hat \g_\ell| =\left|\det (\B)\right|=\left|
\det(R\A(\bz^*)R)\right| =|\det (\A(\bz^*))| \det( R^2) =\left|
\det(\A(\bz^*))\right|\prod_{\ell=1}^nr_\ell.
\ee
as desired.
Substituting in \eqv(up.70) the expression of $\cQ_{\b,N}(\bz^*)$ given
in Proposition \eqv(fine.15), and after the cancellation due to
\eqv(up.81),
we obtain an upper bound which is almost in the form we want.
The only $n$-dependent quantity is the
eigenvalue $\hat\g_1$ of the matrix $\B$.
Taking the limit of $n\rightarrow\infty$ and using
the second part of Lemma \thv(up.83), we recover
the assertion \eqv(up.70.1) of the corollary.
\end{proof}

This corollary concludes the first part of the proof
of Theorem \thv(CAP-thm). The second part, namely the construction
of a matching lower bound, will be discussed in the next section.

\def \by {{\boldsymbol y}}
\def \bbe {{\boldsymbol e}}


\section{Lower bounds on capacities}

In this section we will exploit the variational principle form
Proposition \thv(ab.5) to derive lower bounds on capacities. Our
task is to construct a suitable non-negative unit flow. This will be
done in  two steps. First we construct a good flow for the coarse
grained Dirichlet form in the mesoscopic variables
 and then we use this to construct a flow on the
microscopic variables.

\subsection{Mesoscopic lower bound: The strategy}
Let $\bA$ and $\bB$ be  mesoscopic neighborhoods of two
minima $\bm_{\bA}$ and $\bm_\bB$ of $F_{\b ,N}$, exactly as in the preceding section,
 and let $\bz^*$ be the highest critical point of
$F_{\b ,N}$ which lies between $\bm_\bA$ and $\bm_\bB$.
It would be convenient to pretend that $\bm_\bA, \bz^* ,\bm_\bB\in \Gamma_N^n$:
In general we should  substitute critical  points by their closest
approximations on the latter grid, but the proofs will not be sensitive to the corresponding
corrections.
Recall that the energy landscape around $\bz^*$ has been described in
Subsection~3.2.

Recall that the {\em mesoscopic capacity},
$\Capm{\bA ,\bB}$, is defined in \eqv(rfcw.10).
We will  construct a unit flow, $\frf_{\bA,\bB}$, from $\bA$ to $\bB$ of the
form
\be \label{FlowForm}
\frf_{\bA,\bB}(\bx ,\bx' )= \frac{\cQ_{\b ,N}
(\bx) r_N (\bx ,\bx' )}{\Phi_N (\wt g)}
\phi_{\bA,\bB} (\bx ,\bx' ),
\ee
such that the associated  Markov chain,
$\lb\P^{\frf_{\bA,\bB}}_N ,\cX_{\bA,\bB}\rb$, satisfies
\be \label{onepath}
\P^{\frf_{\bA,\bB}}_N\lb \sum_{\bbe\in\cX_{\bA,\bB}} \phi_{\bA,\bB} (\bbe)
= 1+\so \rb= 1 - \so .
\ee
In view of the general lower bound \eqref{LowerBound},
Eq. \eqref{onepath} implies that the mesoscopic
capacities satisfy
\be \label{CapNbound1}
 \Capm{\bA ,\bB} \geq \E_N^{\frf_{\bA,\bB}} \lbr
\sum_{\bbe= (\bx ,\bx')\in\cX}\frac{\frf_{\bA,\bB} (\bbe)}{\cQ_{\b ,N
}(\bx) r_N (e)} \rbr^{-1} \geq \Phi_N (\wt g)\lb 1 - \so \rb ,
\ee
which is the  lower bound we want to achieve on the mesoscopic
level.

We shall channel all of the flow $\frf_{\bA,\bB}$ through a certain
(mesoscopic) neighborhood $G_N$ of $\bz^*$ .
Namely, our global flow, $\frf_{\bA,\bB}$, in \eqref{FlowForm}
will consist of three (matching)  parts, $   \frf_\bA ,\frf $
 and $\frf_\bB$, where $\frf_\bA$ will be a flow from $\bA$ to
$\partial G_N$, $\frf$ will be a flow through $G_N$,
and $\frf_\bB$ will be a flow from $\partial G_N$ to $\bB$.
 We will recover \eqref{onepath} as a consequence of the
three estimates
\be \label{onepathGN}
\P^{\frf}_N \lb \sum_{\bbe\in\cX} \phi (\bbe)=
1 +\so \rb= 1 - \so ,
\ee
whereas,
\be \label{onepathAB}
\P^{\frf_\bA}_N \lb \sum_{\bbe\in\cX_\bA} \phi_\bA (\bbe)=  \so \rb=
1 - \so\quad \text{and}\quad \P^{\frf_\bB}_N\lb
\sum_{\bbe\in\cX_\bB} \phi_\bB (\bbe)= \so \rb= 1 - \so .
\ee
The construction of $\frf$ through
$G_N$ will be by far the most difficult part. It will
 rely crucially on  Lemma~\ref{up.34}.

\subsection{Neighborhood $G_N$}  We chose
again  mesoscopic coordinates in such a way that $\bz^* = 0$.
 Set $\rho = N^{-1/2 +\delta}$ and fix a (small)
positive number, $\nu >0$. Define
\be \label{GNset} G_N\equiv G_N
(\rho ,\nu )\equiv D_N (\rho )\cap\lbr\bx~:~ (\bx ,\checkv )\in
(-\nu\rho, \nu \rho)\rbr ,
\ee
 where $\checkv\equiv \checkv^{(1)}$
is defined in \eqref{up.31.1}, and $D_N$ is the same as in
\eqref{up.71}.  Note that in view of the discussion in Section~4,
within the region  $G_N$ we may  work with the modified quantities,
$\widetilde{\cQ}_{\b ,N}$ and $r_\ell$; $\ell=1, \dots ,n$, defined
in \eqref{trivial.7} and \eqref{trivial.10}.

The boundary $\partial G_N$ of $G_N$ consists of three disjoint
pieces, $\partial G_N =\partial_\bA G_N\cup\partial_\bB G_N\cup
\partial_r G_N$, where
\be\Eq(strange.30)
 \partial_\bA G_N= \lbr \bx\in\partial G_N: (\bx ,\checkv)
\leq -\nu\rho\rbr\quad {\rm and}\quad
 \partial_\bB G_N= \lbr \bx\in\partial G_N: (\bx ,\checkv)
\geq  \nu\rho\rbr .
\ee
We choose $\nu$ in \eqref{GNset} to be so small that there
exists $K>0$, such that
\be \label{partialrGN}
 F_{\b ,N} (\bx )> F_{\b ,N} (0) + K\rho^2  ,
\ee
uniformly over the remaining part of the boundary
$\bx\in\partial_r G_N$.

Let $\wt g$ be the approximately harmonic function defined in
\eqref{up.33} and \eqv(up.80).
Proceeding along the lines of \eqref{up.72} and
\eqref{up.73} we infer that,
\be \label{Phinu}
\Phi_N (\wt g)\lb
1+\so\rb= \sum_{\bx\in G_N\cup \partial_\bA G_N}
\widetilde{\cQ}_{\b,N} (\bx ) \sum_{\ell\in I_{G_N} (\bx ) } r_{\ell }
\lb \wt g(\bx +\bbe_{\ell} ) - \wt g (\bx)\rb^2 ,
\ee
where $I_{G_N}(\bx )\equiv \lbr\ell~:~
\bx +\bbe_{\ell}\in G_N\rbr$.
For  functions, $\phi$, on
 oriented edges, $(\bx, \bx + \bbe_\ell)$, of $D_N$, we use the
notation
 $\phi_\ell
(\bx) = \phi (\bx, \bx + \bbe_\ell )$,
 and  set
\bea \label{cFforms}
&&\cF_\ell [\phi ] (\bx)\equiv{\widetilde{\cQ}_{\b ,N}(\bx)}
r_\ell \phi_\ell (\bx) ,\nonumber \\
&& {\rm d}\cF[\phi ](\bx) \equiv
\sum_{\ell=1}^n \lb \cF_\ell [\phi ] (\bx) - \cF_\ell [\phi ] (\bx -
\bbe_\ell )\rb .\nonumber
\eea
In particular, the left hand side of
\eqref{up.35} can be written as
 $|{\rm d}\cF [\nabla \wt g ]|/ \widetilde{\cQ}_{\b ,N}(\bx) $.

Let us sum by parts in \eqref{Phinu}. By \eqref{partialrGN} the
contribution coming from $\partial_r G_N$ is negligible and,
consequently, we have, up to a factor of order $(1+\so )$,
\be\Eq(strange.31)
 \sum_{\bx\in G_N}\wt g(\bx ){\rm d} \cF[\nabla \wt g] (\bx )+
\sum_{\bx\in \partial_\bA G_N}\sum_{\ell\in  I_{G_N}(\bx )}
\cF_{\ell}[\nabla g] (\bx ).
\ee
 Furthermore,  comparison between the claim of Lemma~\ref{up.34} and
 \eqref{up.72}  (recall that $\rho^2 =N^{2\delta -1}\ll N^{-1/2}$)
 shows that the first term above is also negligible with
respect to $\Phi_N (\wt g )$. Hence,
\be \label{flowenters}
\Phi_N (\wt g )\lb 1+\so \rb= \sum_{\bx\in \partial_\bA G_N}\sum_{\ell\in
I_{G_N}(\bx )} \cF_{\ell}[\nabla \wt g] (\bx ) .
\ee

\subsection{Flow through $G_N$}
\label{sub:FlowDN} The relation \eqref{flowenters} is the starting
point for our construction of a unit flow of the form
\be\label{fshape}
\frf_\ell (\bx )= \frac{c}{\Phi_N (\wt g )} \cF_\ell[\phi ](\bx )
\ee
through $G_N$. Above $c= 1+\so$ is a normalization
constant. Let us fix $0<\nu_0\ll \nu$ small enough and define,
\be\label{GNnot}
 G_N^0= G_N\cap\lbr \bx~:~ \left| \bx - \frac{(\bx ,\checkv)\checkv}
{\| \checkv\|^2} \right| <\nu_0\rho\rbr .
\ee
Thus, $G_N^0$ is a narrow tube along the principal
$\checkv$-direction (Figure~\ref{fig:DNsets}). We want to construct
$\phi$ in \eqref{fshape} such that the following properties holds:

\noindent {\bf P1:} $\frf$ is confined to $G_N$, it runs from
$\partial_\bA G_N$ to $\partial_\bB G_N$ and it is a  unit flow.
That is,
\be \label{unitflow}
\forall\bx\in G_N, d\cF [\phi ](\bx )=
0\quad\hbox{\rm and}\quad \sum_{\bx\in\partial_\bA G_N}
\sum_{\ell\in I_{G_N}(\bx )} \frf_{\ell}[\phi ] (\bx )= 1 .
\ee

\noindent {\bf P2:}
$\phi$ is a small distortion of $\nabla \wt g$ inside $G_N^0$,
\be\label{insideDN0}
\phi_\ell (\bx)= \nabla_\ell \wt g (\bx) \lb  1 +\so\rb ,
\ee
uniformly in $\bx\in G_N^0$ and $\ell=1, \dots ,n$.

\noindent {\bf P3:}
The flow $\frf$ is negligible outside $G_N^0$ in
the following sense: For some $\kappa>0$,
\be\label{foutside}
\max_{\bx\in G_N\setminus G_N^0}\max_\ell \frf_\ell
(\bx) \leq \frac1{N^{\kappa}} .
\ee
 Once we are able to construct
$\frf$ which satisfies {\bf P1}-{\bf P3} above, the associated
 Markov chain $\lb \P_N^\frf ,\cX\rb$ obviously satisfies
\eqref{onepathGN}.

The most natural candidate  for $\phi$ would seem to be $\nabla \wt g$.
However, since $\wt g$ is not strictly harmonic, this choice does not
satisfies Kirchoff's law, and we would need to correct this by
adding a (hopefully) small perturbation, which in principle can
be constructed recursively.
It turns out, however, to be more
convenient to use as a starting choice
\be\Eq(ab.999)
\phi_\ell^{(0)}(\bx)\equiv {\bv_\ell}\sqrt{\frac{\b |\hat\g_1|}
{2\pi N}} \exp\left({-\b N|\hat\g_1|(\bx,\bv)^2/2}\right),
\ee
which, by \eqv(up.50), satisfies
\be\Eq(ab.999.1)
\phi_\ell^{(0)}(\bx)
=\left(\wt g(\bx+\bbe_\ell)-\wt g(\bx)\right)\left(1+O(\rho)\right),
\ee
uniformly in $G_N$.
Notice that, by \eqref{fshape}, this choice corresponds to the Markov chain with
transition probabilities
\be
\label{qell}
q(\bx,\bx+\bbe_\ell)=\frac{\check\bv_\ell}
{\sum_k\checkv_k}(1 +\so)\equiv q_\ell (1+\so ) .
\ee
From
 \eqref{fine.6} and the decomposition
\eqref{up.31.5} we see that
 \bea \Eq(strange.40)
\nonumber \frac{1+ O(\rho)}{\wt\cQ_{N ,\b} (0)}\cF_\ell[\phi^{(0)}]
&= & r_\ell\bv_\ell \sqrt{\frac{\b
|\hat\gamma_1|}{2\pi N}} \exp\lb -\sfrac{\b N}{2}\lb
|\hat\gamma_1|(\bx ,\bv )^2 + (\bx ,\A \bx )\rb\rb
\\\nonumber
&=& \checkv_\ell \sqrt{\frac{\b |\hat\gamma_1|}{2\pi N}}
\exp\lb - \sfrac{\b N}{2} \lb \sum_{j=2}^n \hat\gamma_j
 (\bx,\bv^{(j)})^2 \rb\rb .
\eea
 In particular, there exists a constant
$\c_1 >0$ such that
\be \label{FLapriori}
\frac{\cF_\ell [\phi^{(0)}] (x) }{\wt\cQ_{N ,\b} (0)}
 \leq \exp\left({-\c_1 N^{2\delta}}\right) ,
 \ee
uniformly in $\bx \in G_N\setminus G_N^0$ and $l =1,\dots ,n$.

Next, by inspection of the proof of  Lemma~\ref{up.34}, we see that
there exists $\c_2$, such that,
\be \label{dFapriori}
\left| {\rm d}\cF [\phi^{(0)}] (\bx )\right|
\leq \c_2 \rho^2 \cF_\ell [\phi^{(0)}] (\bx ) ,
\ee
uniformly in $x\in G_N$ and $\ell =1,\dots
,n$. Notice that we are relying on the strict uniform (in $n$) positivity of the
entries $\bv_\ell$, as stated in Lemma~\ref{lem:ventries}

\smallskip
\paragraph{Truncation of $\nabla g$, confinement of $\frf$
and property P1. }
Let $\cC_+$ be the positive cone spanned by the axis directions
$\bbe_1 ,\dots ,\bbe_n$. Note that the vector $\checkv$ lies in the
interior of $\cC_+$.  Define (see Figure~\ref{fig:DNsets})
\be\Eq(put-numbers.11)
G_N^1= \text{int}\lb \partial_B G_N^0 - \cC_+\rb\cap
G_N\quad\text{and}\quad G_N^2= \lb \partial_A G_N^1 + \cC_+\rb\cap
G_N .
\ee
We assume that the constants $\nu$ and $\nu_0$ in the definition of
$G_N$
 and, respectively, in the definition of $G_N^0$
 are tuned in such a way that $G_N^2\cap\partial_r G_N = \emptyset$.
\begin{figure}[tbh]
\psfrag{DN}{$G_N$} \psfrag{DN0}{$G_N^0$} \psfrag{DN1}{$G_N^1$}
\psfrag{DN2}{$G_N^2$} \psfrag{z}{$\bz^*$} \psfrag{v}{$\checkv$}
\psfrag{DA}{$\partial_A G_N$} \psfrag{DNB}{$\partial_B G_N$}
\includegraphics[width=11cm]{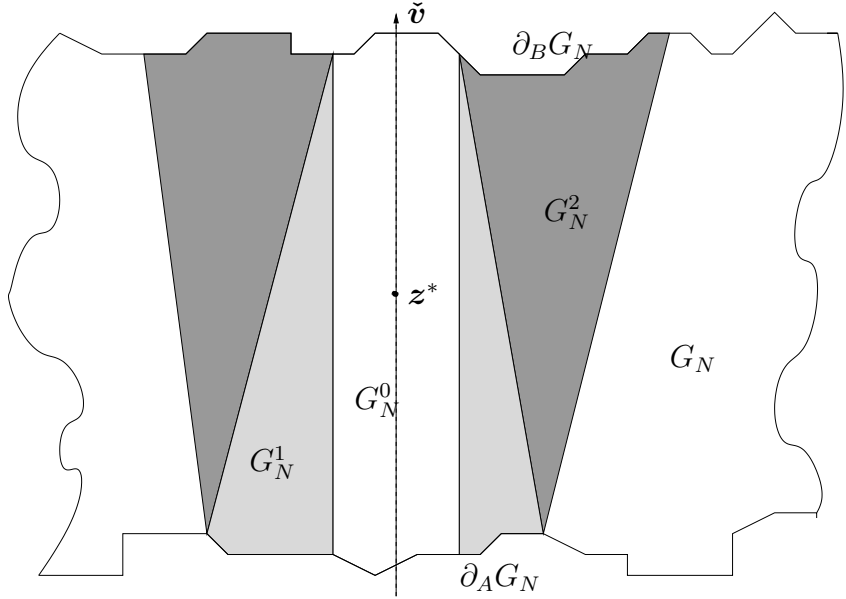}
\label{fig:DNsets}
\caption{ Narrow tube $G_N^0$ and  sets $G_N^1$ and $G_N^2$ }
\end{figure}
Let $\wt\phi^{(0)}$ be the restriction of $\phi^{(0)}$ to $G^1_N$,
\be \label{nablag}
\widetilde\phi^{(0)}_\ell(\bx) \equiv
\phi^{(0)}_\ell(\bx) \1_{\lbr \bx\in G_N^1\rbr}.
\ee

Now we turn to the construction of the full flow. To this end we
start by setting the values of $\phi_\ell$ on $\partial_\bA G_N$
equal to $\widetilde{\phi}_{(0)} $ if $\ell\in I_{G_N} (\bx )$
and zero otherwise. By \eqref{flowenters} and the bound
\eqref{FLapriori},
 the second  of the relations in \eqref{unitflow} is satisfied.

In order to satisfy  Kirchoff's law inside $G_N$, we write $\phi$ as
$\phi = \wt\phi^{(0)} + u$ with $u$ satisfying the recursion,
\be\label{uequation}
\sum_{\ell=1}^n  \cF_\ell [u] (\bx )=
\sum_{\ell=1}^n \cF_\ell[u] (\bx - \bbe_\ell )
- {\rm d}\cF [\widetilde\phi^{(0)}](\bx) .
\ee
Since
$\widetilde\phi^{(0)} \equiv 0$ on $G_N\setminus G_N^1$, we may
trivially take $u\equiv 0$ on $G_N\setminus G_N^2$ and then
 solve \eqref{uequation} on $G_N^2$ using the latter as an
insulated boundary condition on $\partial G_N^2\cap G_N $.
\smallskip

\paragraph{Interpolation of the flow inside $G_N^2$.}
  We  first solve
\eqref{uequation} inside $G_N^1$. By construction, if $\bx\in G_N^1$
then $\bx-\bbe_\ell \in G_N^1\cup \partial_\bA G_N^1$, for every
$\ell=1, \dots ,n$. Accordingly, let us slice $G_N^1$ into
layers ${\mathbb L}_k$ as follows:  Set
\be\Eq(strange.41)
{\mathbb L}_0= \partial_\bA G_N^1 ,
\ee
and, for $k=0,1,\dots $,
\be\Eq(strange.42)
 {\mathbb L}_{k+1}= \lbr \bx\in G_N ~:~ \bx - \bbe_\ell\in
\bigcup_{j=0}^k{\mathbb L}_j\ \text{for all}\ \ell =1, \dots, n\rbr .
\ee
Since all entries of $\bv$ are  positive,
there exists $\c_3 = c_3 (n)$ and $M\leq \c_3/\rho$, such that
\be\Eq(strange.44)
 G_N^1= \bigcup_{j=0}^{M} {\mathbb L}_j .
\ee
Now define recursively, for each $\bx\in {\mathbb L}_{k+1}$,
\be\label{interpolation}
\cF_\ell [u](\bx)= q_\ell\lb \sum_{j=1}^n\cF_j
[u](\bx- \bbe_j) -  {\rm d}\cF[\wt\phi^{(0)}] (\bx)\rb ,
\ee
where the probability distribution, $q_1 ,\dots ,q_n$, is defined as in
\eqref{qell}.
 Obviously,
this produces a solution of  \eqref{uequation}. The particular
choice of the constants $q_\ell$ in \eqref{qell} leads to a
rather miraculous looking cancellation we will encounter below.

\paragraph{Properties P2 and P3.} We now prove recursively a bound on
$u$ that will imply that  Properties {\bf P2} and {\bf P3} hold. Let
$c_k$ be constants such that, for all $\by \in {\mathbb L}_k$,
\be\label{kBound}
 \left| \cF_\ell [u] (\by)\right| \leq
c_k \rho^2 \cF_\ell [\nabla \wt g] (\by).
\ee
Then, for $\bx\in {\mathbb L}_{k+1}$, we get
 by construction \eqref{interpolation} and in view
of \eqref{dFapriori} that
\bea\Eq(strange.45)
\frac{| \cF_\ell [u](\bx)|}{\cF_{\ell } [\wt\phi^{(0)}](\bx)}
&\leq& q_\ell \sum_j \frac{|\cF_j [u](\bx - \bbe_j)|}
{\cF_{\ell }[\wt\phi^{(0)}](\bx)} + \chi_2\rho^2
\\\nonumber
&\leq& \rho^2\lb c_k q_\ell \sum_j \frac{\cF_j
[\wt\phi^{(0)}](\bx - \bbe_j)}
{\cF_{\ell}[\wt \phi^{(0)}] (\bx)} + \chi_2\rb.
\eea
 By our choice of $\phi^{(0)}$ in \eqref{nablag},
\bea\Eq(strange.46)
 \frac{\cF_j [\wt\phi^{(0)}](\bx - \bbe_j)}{\cF_{\ell }
[\wt\phi^{(0)}](\bx)}
&=& \frac{\checkv_j}{\check \bv_\ell} \exp\lbr \frac{\b N}{2}
\sum_{i=2}^n \hat\gamma_i\lb  (\bx ,\bv^{(i)})^2  - (\bx -\bbe_j
,\bv^{(i)})^2 \rb\rbr
\\\nonumber
&=& \frac{\checkv_j}{\check
\bv_\ell} \exp\lbr \b N \sum_{i=2}^n \hat\gamma_i (\bx ,
\bv^{(i)})(\bbe_j ,\bv^{(i)})\rbr\lb 1 + O\lb 1/N\rb\rb
\\\nonumber
&=& \frac { \checkv_j + 2\b  (\bbe_j ,\hat\bv ) \sum_{i=2}^n (\bbe_j
,\hat\bv^{(i)} ) (\bx , \bv^{(i)}) } {\checkv_\ell} \lb 1 + O(\rho^2
)\rb .
 \eea
However, for each $i = 2, \dots, n$,
\be\Eq(strange.47)
 \sum_{j=1}^n (\bbe_j ,\hat\bv )(\bbe_j ,\hat\bv^{(i)} )= 0 .
\ee
Therefore, with the choice $q_\ell =\frac{\check\bv_\ell}
{\sum_k\checkv_k}(1 +\so)$, we get
\be\Eq(strange.48) q_\ell \sum_j \frac{\cF_j [\wt\phi^{(0)} ]
(\bx - \bbe_j)}
{\cF_{\ell }[\wt\phi^{(0)}] (\bx)} =  1 + O(\rho^2 ) ,
\ee
uniformly in $\bx\in G_N^1$ and $l=1, \dots
,n$. Thus, the coefficients $c_k$ satisfy the  recursive bound
\be\Eq(strange.48.1)
 c_{k+1}\leq  c_k \lb 1 + O(\rho^2 )\rb + \chi_2\rho^2,
\ee
with  $c_0 =0$.
Consequently, there exists a constant, $c$, such
that \be\Eq(strange.49) c_{k}\leq k\rho^2 c e^{kc\rho^2}, \ee and
hence,  since  $M \leq \c_3/\rho$,  $c_M= O(\rho )$. As a result,
we have constructed $u$ on $G_N^1$ such that
\be \label{uniBound}
\left| \cF_\ell [u] (\bx)\right|= O\lb\rho \rb
\cF_\ell [\nabla g] (\bx),
\ee
uniformly in $\bx\in G_N^1$ and $\ell=1,\dots ,n$. In
particular,  \eqref{insideDN0}
 holds  uniformly in $\bx\in G_N^1$ and hence, by \eqref{FLapriori},
{\bf P3} is satisfied on $G_N^1\setminus G_N^0$.
Moreover, since by construction $\phi\equiv 0$ on
$G_N\setminus G_N^2$,  {\bf P3} is trivially satisfied
in the latter domain.
Hence
both {\bf P2}
 and {\bf P3} hold on $G_N^1\cup \lb G_N\setminus G_N^2\rb$.
\smallskip

It remains to reconstruct $u$ on $G_N^2\setminus G_N^1$.  Since we
have truncated $\nabla g$ outside $G_N^1$,   Kirchoff's equation
\eqref{uequation}, for $\bx\in G_N^2\setminus G_N^1$, takes the form
 $\cF [u] (\bx) = 0$.  Therefore, whatever we do in order to
reconstruct $\phi$, the total flow through $G_N^2\setminus G_N^1$
equals
\be\Eq(strange.51)
\frac{1+\so}{\Phi_N (\wt g)}\sum_{\bx\in
G_N^1}\sum_{\ell=1}^n \cF_\ell[\phi ]  (\bx)\1_{\lbr
  \bx+\bbe_\ell\not\in G_N^1\rbr} .
\ee
By \eqref{uniBound} and \eqref{FLapriori}, the latter is of the
order  $O\lb \rho^{1-n}e^{-\c_1 N^{2\delta}}\rb$.
Thus, {\bf P3} is established.

\subsection{Flows from $\bA$ to $\partial_\bA G_N$ and from $\partial_\bB
  G_N$ to $\bB$}
\label{sub:Outside} Let $\frf$ be the unit flow through $G_N$
constructed above.
 We need to construct a flow
\be \label{phiA}
 \frf_\bA (\bx ,\by)= (1+\so)\frac{\cQ_{\b ,N} (\bx )
r_N (\bx ,\by )}{\Phi_N (\wt g)}\phi_\bA (\bx ,\by ) \ee from $\bA$ to
$\partial_\bA G_N$ and, respectively, a flow \be\Eq(miracle.1)
 \frf_\bB (\bx ,\by)= (1+\so)\frac{\cQ_{\b ,N} (\bx )
 r_N (\bx ,\by )}{\Phi_N (\wt g)}\phi_\bB (\bx ,\by )
\ee
from $\partial_\bB G_N$ to $\bB$, such that \eqref{onepathAB} holds
and, of course, such that the concatenation $\frf_{\bA,\bB} = \lbr
\frf_\bA ,\frf ,\frf_\bB\rbr$ complies with Kirchoff's law.  We shall
work out only the $\frf_\bA$-case, the $\frf_\bB$-case is
completely analogous.

The expressions for $\Phi_N (\wt g )$ and $\cQ_{\b ,N} (\bx )$ appear
on the right-hand sides
 of \eqref{up.70} and \eqref{fine.1.1}. For the rest we need
only rough bounds: There exists
a constant $L =L(n)$, such that we are able to rewrite \eqref{phiA}
as,
\be \label{phiBound}
 \phi_\bA (\bx ,\by) =  \frac{(1+\so )\Phi_N (\wt g)
\frf_\bA (\bx ,\by)}{{\cQ}_{\b ,N}
  (\bx){r}_N (\bx ,\by)} \leq LN^{n/2 +1}
e^{-N (F_{\b,N} (\bz^* )- F_{\b,N}(\bx))} .
 \ee
This would imply a uniform stretched exponentially small upper bound
on $\phi_\bA$  at points $\bx$ which are mesoscopically away from
$\bz^*$ in the direction of $\nabla F_{\b ,N} $, for example  for $\bx$
satisfying
\be \label{bxbelow}
F_{\b ,N} (\bz^* ) - F_{\b ,N}  (\bx) >c N^{2\delta-1}.
\ee
With the above discussion in mind let us try to construct
$\frf_\bA$  in such a way that it charges only bonds $(\bx, \by )$ for
which \eqref{bxbelow} is satisfied.  Actually we shall do much better and give
a more or less explicit construction of the part of $\frf_\bA$ which flows through
$G_N^0$:  Namely, with each point $\bx\in\partial_\bA G_N^0$  we shall associate
a nearest neighbor path $\gamma^\bx = (\gamma^\bx (-k_A (\bx )), \dots ,
\gamma^\bx (0))$ on $\G_N^n$ such that
 \eqref{bxbelow} holds for all
$\by\in \gamma^\bx$ and,
\be\Eq(put.12)
 \gamma^\bx (-k_A (\bx ))\in\bA, \ \gamma^\bx (0) = \bx\quad\text{and}\quad
m (\gamma^\bx (\cdot +1)) = m (\gamma^\bx (\cdot )) + 2/N .
\ee
The flow from $\bA$ to $\partial_\bA G_N^0$ will be then defined as
\be
\label{fAflow}
\frf_\bA (\bfe ) = \sum_{\bx\in \partial_\bA G_N^0}\1_{\lbr \bfe\in \gamma^\bx\rbr}
\sum_{\ell\in I_{G_N} (\bx )}\frf_\ell (\bx ) .
\ee
By construction $\frf_\bA$ above satisfies the Kirchoff's law and matches with the flow
$\frf$ through $G_N$ on $\partial_\bA G_N^0$.
Strictly speaking, we should also specify how one extends
$\frf$ on the remaining part $\partial_\bA  G_N\setminus\partial_\bA G_N^0$. But this
is irrelevant: Whatever we do the $\P^{\frf_{\bA,\bB}}_N$-probability of
passing through $\partial_\bA G_N\setminus \partial_\bA G_N^0$ is equal to
\be\Eq(put.13)
\sum_{\bx\in \partial_\bA G_N\setminus \partial_\bA G_N^0} \sum_\ell
\frf_\ell (\bx)= \so .
\ee
It remains, therefore, to construct the family of paths $\lbr\gamma^\bx\rbr$
such that \eqref{bxbelow} holds.

\noindent
Each such path $\gamma^\bx$ will be constructed as a concatenation
 $\gamma^\bx= \hat{\gamma}\cup\eta^\bx$.
\vskip 0.2cm

\noindent
\step{1} Construction of $\hat{\gamma}$.
 Pick $\delta$
such that $\delta -1 < m_A = m (\bm_\bA )$ and  consider  the part $\hat\bx [\delta -1 ,z^*]$ of the
minimal energy curve as described in \eqref{fine.505}.  Let ${\gamma}$ be a
nearest neighbor
$\G_N^n$-approximation of $\hat\bx [\delta -1 ,z^*]$, which in addition
satisfies $m(\hat\gamma (\cdot +1)) = m(\hat\gamma (\cdot )) +2/N$. Since by
\eqref{eq:mincurve} the curve $\hat\bx[\delta -1 ,z^*]$ is coordinate-wise increasing, the
Hausdorff distance between $\hat\gamma$ and $\hat\bx[\delta -1 ,z^*]$ is at most
$2\sqrt{n}/N$.  Let $\bx^\bA$ be the first point where $\gamma$ hits
the set $D_N (\rho )$, and let $\bu^\bA$ be the last point where
$\gamma$ hits  $\bA$ (we assume now that the neighborhood $\bA$ is sufficiently large so
that $\bu^\bA$ is well defined). Then $\hat\gamma$ is just the portion of
$\gamma$ from $\bu^\bA$ to $\bx^\bA$.
\vskip 0.2cm

\noindent
\step{2} Construction of $\eta^\bx$. At this stage
 we assume that the parameter $\nu$ in \eqref{GNset} is so small
that $G_N$ lies deeply inside  $D_N (\rho )$. In particular, we may assume that
\[
 F_{\beta, N} (\bx^\bA )\, <\, \min\lbr F_{\beta ,N} (\bx )~:~\bx\in\partial_A G_N^0\rbr ,
\]
and, in view of \eqref{eq:mincurve}, we may also assume that
\be
\label{eq:xellup}
 \bx^\bA_\ell <\bx_\ell\quad \forall \bx\in\partial_A G_N^0\ \text{and}\  \ell=1, \dots, n.
\ee
Therefore, $\bx - \bx^\bA$ has strictly positive entries and, as it now follows from
\eqref{up.31.2},
\[
\lb\A\checkv ,\bx - \bx^\bA\rb \, =\, \lb \bv ,\bx - \bx^\bA\rb > 0 .
\]
By construction $G_N^0$ is a small tube in the direction of $\checkv$. Accordingly,
we may assume that $\lb\A\bx,\bx - \bx^\bA\rb > 0$ uniformly on $\partial_A G_N^0$.
But this means that the function
\[
t: [0,1]\mapsto  \lb \A (\bx^\bA +t (\bx - \bx^\bA ) , (\bx^\bA +t (\bx - \bx^\bA )\rb
\]
is strictly increasing. Therefore, $F_{\beta ,N}$ is, up to negligible corrections,
increasing on the straight line segment, $[\bx^\bA ,\bx ]\subset\R^n$
which connects $\bx^\bA$ and $\bx$.   
Then, our target path $\eta^\bx$ is
 a
nearest neighbor
$\G_N^n$-approximation of $[\bx^\bA ,\bx ]$ which runs from $\bx^\bA$
to $\bx$ . 
In view of the preceeding discussion it is possible to prepare $\eta^\bx$ in such a way
that   $
F_{\b ,N} (\bz^* ) - F_{\b ,N}  (\cdot ) >c N^{2\delta-1}$ along $\eta^\bx$.
Moreover, by
\eqref{eq:xellup} it is possible to ensure that the total
magnetization is increasing along $\eta^\bx$ .

This concludes the construction of a flow $\frf_{\bA ,\bB}$ satisfying \ref{CapNbound1}.  \qed

\vskip 0.2cm

In the sequel  we shall index vertices of $\gamma^\bx = \hat\gamma\cup\eta^\bx$ as,
\be\Eq(put.14)
 \gamma^\bx = \lb \hat\gamma^\bx (-k_\bA ), \dots \hat\gamma^\bx (0)\rb .
\ee
Since,
\be\Eq(put.15)
 F_{\b ,N} (\by  ) \leq F_{\b ,N} (\bz^* ) - c_1 \lb \by - \bz^* , \bv\rb^2 ,
\ee
for every $\by$ lying on the minimal energy curve $\hat\bx [\delta - 1 ,z^*]$ and since  the
Hessian of $F_{\b ,N}$ is uniformly bounded on $\hat\bx [\delta - 1 ,z^*]$, we conclude that
if $\nu_0$ is chosen small enough,
then there exists $c_2 >0$ such that
\be
\label{pathdecay}
F_{\b ,N} (\gamma^\bx ( \cdot )) \leq F_{\b ,N} (\bz^* ) -
c_2\lb \gamma^\bx (\cdot)  - \bz^* , \bv\rb^2 ,
\ee
uniformly in $\bx\in\partial_\bA G_N^0$. Finally, since the entries of $\bv$ are
uniformly strictly positive, it follows from \eqref{pathdecay} that,
\be
\label{pathdecayk}
F_{\b ,N} (\gamma^\bx ( - k )) \leq F_{\b ,N} (\bz^* ) - c_3 \frac{(N^{1/2 +\delta} +k)^2}{N^2},
\ee
uniformly in $\bx\in\partial_\bA$ and $k\in \lbr 0, \dots ,k_\bA (\bx )\rbr$.

\subsection{Lower bound on $\capa (A,B )$ via microscopic flows}
\label{sub:LBMicro}
Recall that $\bA$ and $\bB$ are mesoscopic neighborhoods of two
minima of $F_{\b ,N}$, $\bz^*$ is the corresponding saddle
point, and $A = \cS_N[\bA ]$, $B=\cS_N [\bB ]$ are the microscopic
counterparts of $\bA$ and $\bB$.  Let $\frf_{\bA ,\bB} = \lbr\frf_\bA ,\frf , \frf_\bB\rbr$
 be the mesoscopic flow from $\bA$ to $\bB$ constructed above. In this section we
are going to construct a subordinate microscopic flow, $f_{A ,B}$, from $A$ to $B$. In the
sequel, given a microscopic bond, $b = (\s , \s')$, we use $\bfe (b) = (\bm (\s ) ,\bm (\s' ))$
for its mesoscopic pre-image.  Our subordinate flow will satisfy
\be
\label{Fsubordinate}
\frf_{\bA ,\bB} ( \bfe ) = \sum_{b :\bfe (b )=\bfe} f_{A,B} (b) .
\ee
In fact, we are going to employ a much more stringent notion of
subordination on the level of induced Markov chains:
Let us label the realizations
of the mesoscopic chain $\cX_{\bA ,\bB}$
as $\underline{\bx} = \lb \bx_{-\ell_A}, \dots ,\bx_{\ell_B}\rb$,
in such a way that $\bx_{-\ell_A}\in\bA$, $\bx_{\ell_B}\in\bB$, and $m(\bx_0 )= m (\bz^* )$.
If $\bfe$ is a mesoscopic bond, we write $\bfe\in\ubx$ if $\bfe = (\bx_\ell ,\bx_{\ell+1})$ for some
$\ell = -\ell_A, \dots ,\ell_B -1$.
 To each path, $\underline{\bx}$, of positive probability, we  associate a
 subordinate microscopic \emph{unit flow}, $f^{\ubx}$, such that
\be\Eq(put.16)
 f^{\ubx} (b) >0\ \  \text{if and only if}\ \ \bfe (b)\in\ubx .
\ee
Then the total microscopic flow, $f_{A,B}$, can be decomposed as
\be
\label{FsubordinateC}
f_{A,B} = \sum_{\ubx}\P_N^{\frf_{\bA ,\bB}}\lb \cX_{\bA ,\bB} = \ubx\rb f^{\ubx} .
\ee
Evidently, \eqref{Fsubordinate} is satisfied: By construction,
\be\Eq(put.17)
 \sum_{b : \bfe (b) = \bfe} f^{\ubx}(b)  = 1\ \ \text{for every $\ubx$ and each $\bfe\in\ubx$} .
\ee
On the other hand, $\frf_{\bA ,\bB} (\bfe ) =
\sum_{\ubx}\P_N^{\frf_{\bA ,\bB}}\lb \cX_{\bA ,\bB}=\ubx\rb
\1_{\{\bfe\in\ubx\}}$.

Therefore, \eqref{FsubordinateC} gives rise to the following
decomposition of unity,
\be\Eq(put.19)
 \1_{\{ f_{A ,B} (b) >0\}} = \sum_{\ubx\ni \bfe (b)}\sum_{\us\ni
   b} \frac{\P_N^{\frf_{\bA ,\bB}}
\lb
\cX_{\bA ,\bB} = \ubx\rb \P^{\ubx}\lb \Sigma = \us\rb}{\frf_{\bA ,\bB }(\bfe (b))
f^{\ubx }(b )} ,
\ee
where $\lb \P^{\ubx } ,\Sigma \rb$ is the {\em microscopic} Markov chain from
$A$ to $B$ which is associated to the flow $f^{\ubx}$.

Consequently, our general lower bound \eqref{ab.5} implies that
\bea
\label{frf-fbound}
\nonumber
\capa(A,B)&\geq& \sum_{\ubx} \P_N^{\frf_{\bA ,\bB}} \lb
\cX_{\bA ,\bB} = \ubx\rb \E^{\ubx}
\lbr
\sum_{\ell = -\ell_A}^{\ell_B -1}\frac{\frf_{\bA ,\bB} (\bx_{\ell} ,\bx_{\ell +1})
f^{\ubx}(\s_{\ell} ,\s_{\ell+1})}{\m_{\b ,N} (\s_{\ell}) p_N (\s_{\ell},\s_{\ell +1})}\rbr^{-1}\hspace{5mm} \\
&\geq&
\sum_{\ubx} \P_N^{\frf_{\bA ,\bB}}
\lb
\cX_{\bA ,\bB} = \ubx\rb
\lbr \E^{\ubx}
\sum_{\ell = -\ell_A}^{\ell_B -1}\frac{\frf_{\bA ,\bB} (\bx_{\ell} ,\bx_{\ell +1})
f^{\ubx}(\s_{\ell} ,\s_{\ell+1})}{\m_{\b ,N} (\s_{\ell}) p_N (\s_{\ell},\s_{\ell +1})}\rbr^{-1}
\eea
We need to recover $\Phi_N (\wt g)$ from the latter expression. In view of
\eqref{FlowForm}, write,
\bea
\label{eq:product}
\frac{\frf_{\bA ,\bB} (\bx_{\ell} ,\bx_{\ell +1})
f^{\ubx}(\s_{\ell} ,\s_{\ell+1})}{\m_{\b ,N} (\s_{\ell}) p_N (\s_{\ell},\s_{\ell +1})}
&=&
\frac{\phi_{\bA ,\bB}  (\bx_{\ell} ,\bx_{\ell +1}) }{\Phi_N (\wt g)}\\
\nonumber
&\times&
\frac{\cQ_{\b ,N} (\bx_\ell )r_N (\bx_{\ell  }, \bx_{\ell +1}) f^{\ubx}(\s_{\ell},\s_{\ell +1})}
{\m_{\b ,N}(\s_\ell ) p_N (\s_{\ell},\s_{\ell +1})} .
\eea
Since we prove  lower bounds,
we may restrict attention to a subset of {\em good} realizations $\ubx$ of the mesoscopic chain
$\cX_{\bA ,\bB}$ whose  $\P_N^{\frf_{\bA ,\bB}}$ -probability is
  close to one.
In particular, \eqref{onepathGN} and
\eqref{onepathAB} insure that the first term in the above product is
precisely what we need.
The  remaining effort, therefore, is to find a judicious choice of
$f^{\ubx}$ such that
the second factor  in \eqref{eq:product} is close to one.  To
this end we need
some  additional notation: Given a mesoscopic trajectory $\ubx = (\bx_{-\ell_A} ,\dots ,
\bx_{\ell_B})$, define $k = k(\ell )$ as the direction  of the  increment of $\ell$-th jump. That is,
$\bx_{\ell +1} = \bx_\ell + \bfe_k$.  On the microscopic level such a transition corresponds
to a flip of a spin from the $\Lambda_k$ slot.
Thus, recalling the notation
$\Lambda^{\pm}_k (\s )\equiv  \lbr i\in \Lambda_k~:~\s (i) =\pm1\rbr$,
we have that, if $\s_\ell\in\cS_N [\bx_\ell ]$ and $\s_{\ell+1}\in
\cS_N [\bx_{\ell+1} ]$, then $\s_{\ell+1} = \theta_i^+\s_\ell$ for some
$i\in\Lambda^-_{k(\ell )} (\s_\ell )$.
By our choice of transition probabilities, $p_N$, and their mesoscopic
counterparts,
 $r_N$, in \eqref{rfcw.11},
\be\Eq(put.21)
 \frac{r_N (\bx_\ell , \bx_{\ell +1})}{p_N (\s_\ell , \s_{\ell +1})} =
\left| \Lambda^-_{k(\ell)} (\s_\ell )\right|\lb 1+ O(\epsilon )\rb ,
\ee
uniformly in $\ell$ and in all  pairs of neighbors $\s_\ell,
\s_{\ell+1}$.   Note that
the cardinality, $\left| \Lambda^-_{k(\ell)} (\s_\ell )\right|$, is
the same for all $\s_\ell\in \cS_N [\bx_\ell ]$.

For $\bx\in\Gamma_N^n$, define the canonical measure,
\be\Eq(put.22)
 \m^\bx_{\b ,N} (\s ) = \frac{\1_{\lbr \s\in\cS_N [\bx ]\rbr}\m_{\b ,N} (\s )}
{\cQ_{\b ,N} (\bx )} .
\ee
The second term in \eqref{eq:product} is equal to
\be
\label{eq:ratio}
 \frac{f^{\ubx }(\s_\ell ,\s_{\ell +1} )}
{\m^{\bx_\ell}_{\b ,N} (\s_\ell )\cdot 1/\left| \Lambda^-_{k(\ell)} (\s_\ell )\right|}
\lb 1+ O(\epsilon )\rb  .
\ee
If the magnetic fields, $h$, were  constant on each set $I_k$, then we
could chose the flow $f^{\ubx }(\s_\ell ,\s_{\ell +1} )=
\m^{\bx_\ell}_{\b ,N} (\s_\ell )\cdot 1/\left| \Lambda^-_{k(\ell)}
(\s_\ell )\right|$,
and  consequently we would be done. In the
general case of  continuous distribution of $h$, this is not the
case. However, since the fluctuations of $h$ are bounded by $1/n$, we
can hope  to construct $f^{\ubx}$ in such a way that the ratio in
\eqref{eq:ratio} is kept very  close to one.
\smallskip

\paragraph{Construction of $f^{\ubx}$.}
We  construct now a Markov chain, $\P^{\ubx}$, on microscopic
trajectories, $\Sigma = \lbr \s_0 ,\dots ,\s_{\ell_B}\rbr$, from $\cS
[\bx_0 ]$ to $B$,
 such that $\s_\ell\in \cS [\bx_\ell ]$, for all $\ell=0,\dots ,\ell_B$.
The microscopic
flow, $f^{\ubx}$, is  then defined through the identity
$\P^{\ubx}\lb b\in \Sigma \rb= f^{\ubx} (b) $.

The construction of a microscopic flow from $A$ to $\cS [\bx_0 ]$ is
completely similar (it is just the reversal of the above) and we
will omit it.

 We now construct  $\P^{{\ubx}}$.

\noindent \step{1}. Marginal distributions: For each $\ell=0,\dots, \ell_B$
we use $\nu_\ell^{\ubx}$ to denote the marginal distribution of
$\s_\ell$ under
 $\P^{\ubx}$. The measures $ \nu_\ell^{\ubx} $ are concentrated on
$\cS [\bx_\ell ]$.
The initial measure, $\nu_0^{\ubx}$, is just the canonical measure
$\muc{\bx_0 }$.  The measures $\ \nu_{\ell+1}^{\ubx}$ are then defined
through the recursive equations
\be\label{Measurenul}
\nu_{\ell+1}^{\ubx}
(\s_{\ell +1} )= \sum_{\s_{\ell}\in\cS [\bx_\ell ]} \nu_{\ell}^{\ubx} (\s ) q_\ell
(\s_{\ell} ,\s_{\ell +1} ) .
\ee
\step{2}. Transition probabilities. The transition
probabilities, $q_\ell (\s_{\ell} ,\s_{\ell+1} )$, in \eqref{Measurenul} are
defined in the following way:
As we have already remarked, all the microscopic
jumps are of the form $\s_{\ell} \mapsto \theta_j^+\s_{\ell}$, for some
$j\in\Lambda_{k(\ell )}^- (\s )$, where $ \theta_j^+$ flips the $j$-th spin
from $-1$ to $1$. For such a flip define
\be\label{qRates}
q_\ell  (\s_{\ell} , \theta_j^+\s_{\ell} )= \frac{ e^{2\beta \tilde{h}_j}}{
\sum_{i\in \Lambda_k^- (\s_{\ell} )} e^{2\beta \tilde{h}_i}}  .
 \ee
Then the microscopic flow through an admissible  bound,
$b = (\s_{\ell} , \s_{\ell +1})$, is equal to
\be\Eq(put.23)
 f^{\ubx} (\s_{\ell} ,\s_{\ell +1}) = \P^{\ubx}\lb b\in\Sigma\rb =
\nu^{\ubx}_{\ell}(\s_{\ell}) q_{\ell }(\s_{\ell} ,\s_{\ell+1})=
\frac{\nu^{\ubx}_{\ell}(\s_{\ell})}{\left| \Lambda^-_{k(\ell)} (\s_\ell )\right|}
\lb 1+ O(\epsilon )\rb .
\ee
Consequently, the expression in \eqref{eq:ratio}, and hence the second term
in \eqref{eq:product}, is equal to
\be
\label{eq:Psil}
 \frac{\nu^{\ubx}_{\ell}(\s_{\ell}) }{\m^{\bx_\ell}_{\b ,N} (\s_{\ell} )}\lb 1+ O(\epsilon )\rb
\equiv \Psi_{\ell} (\s_{\ell})\lb 1+ O(\epsilon )\rb  .
\ee

\smallskip
\paragraph{Main result.}   We claim that there exists a set, $\TT_{\bA ,\bB}$,
of {\em good} mesoscopic trajectories from $\bA$ to $\bB$, such that
\be
\label{AllAreGood}
\P_N^{\frf_{\bA ,\bB}}\lb \cX_{\bA ,\bB}\in \TT_{\bA ,\bB}\rb = 1 -\so ,
\ee
and, uniformly in $\ubx\in \TT_{\bA ,\bB}$,
\be
\label{CorrectorBound}
\E^{\ubx}\lb
\sum_{\ell = -\ell_A}^{\ell_B -1} \Psi_{\ell} (\s_{\ell})\phi_{\bA ,\bB}(\bx_{\ell }, \bx_{\ell +1})
\rb
\leq 1 + O(\epsilon )  .
\ee
This will imply  that,
\be
\label{capaLB}
\capa (A, B)\geq \Phi_N (\wt g) \lb 1 - O(\epsilon )\rb ,
\ee
which is the lower bound necessary to prove Theorem \thv(CAP-thm).

The rest of the Section is devoted to
the proof of \eqref{CorrectorBound}. First of all we derive
recursive estimates
 on $\Psi_{\ell}$ for a given realization, $\ubx$, of the mesoscopic chain.
After that it will be obvious how to define $\TT_{\bA ,\bB}$.

\subsection{Propagation of errors along microscopic paths}
Let $\ubx$ be given.
 Notice that $\m^{\bx_\ell}_{\b ,N}$ is the product measure,
\be
\label{eq:muproduct}
 \m^{\bx_\ell}_{\b ,N} = \bigotimes_{j=1}^n \m^{\bx_\ell (j)}_{\b ,N} ,
\ee
where $\m^{\bx_\ell (j)}_{\b ,N}$ is the corresponding canonical measure on the
mesoscopic slot $\cS_N^{(j)} = \lbr -1 ,1\rbr^{\Lambda_j}$.  On the other hand,
 according to \eqref{qRates},
the {\em big}  microscopic chain $\Sigma$ splits into
a direct product of
$n$ {\em small} microscopic chains, $\Sigma^{(1)}, \dots ,\Sigma^{(n)}$, which
independently evolve on $\cS_N^{(1)} ,\dots ,\cS_N^{(n)}$.  Thus,
$k(\ell ) = k$
means that the $\ell$-th step of the mesoscopic chain induces a step
of the $k$-th small
microscopic chain $\Sigma^{(k)}$.  Let $\tau_1 [\ell ], \dots ,\tau_n
[\ell ]$ be the
numbers of steps performed by each of the small microscopic chains after $\ell$
steps of the mesoscopic chain or, equivalently, after $\ell$ steps of
the big microscopic
chain $\Sigma$.  Then the corrector, $\Psi_\ell$, in \eqref{eq:Psil} equals
\be
\label{PsilProduct}
\Psi_\ell\lb \s_{\ell}\rb= \prod_{j=1}^n \psi_{\tau_j [\ell ]}^{(j)} (\s_\ell^{(j)}) ,
\ee
where $\s_\ell^{(j)}$ is the projection of $\s_\ell$ on $\cS_N^{(j)}$.
Therefore we are left
with two separate tasks: On the microscopic level we need to control
the propagation of
errors along {\em small} chains and, on the mesoscopic level, we need
to control the
statistics of  $\tau_1 [\ell ], \dots ,\tau_n [\ell ]$.  The latter
task is related to
characterizing the set, $\TT_{\bA ,\bB}$, of {\em good} mesoscopic
trajectories and
it is relegated to Subsection~\ref{sub:Good}
\smallskip

\paragraph{Small microscopic chains.} It would be convenient to study
the propagation
of errors along small microscopic chains in the following slightly
more general context:   Fix $1\ll M\in \N$ and $0\leq \epsilon \ll 1$.
Let $g_1 ,\dots ,g_M\in [-1,1]$.
Consider spin
configurations, $\xi\in \cS_M = \lbr -1 ,1\rbr^M$, with product weights
\be\Eq(ut.24)
 w (\xi ) = {\rm e}^{\epsilon \sum_i g_i \xi (i)} .
\ee
As before, let $\Lambda^{\pm}(\xi ) = \lbr i~:~\xi (i) =\pm 1\rbr$.
Define layers of fixed  magnetization,  $\cS_M [K ]= \lbr \xi\in\cS_M ~:~
\left|\Lambda^+ (\xi )\right| = K \rbr$.
Finally, fix $\delta_0, \d_1 \in (0,1)$,
such that $\delta_0 <\d_1$.

Set $K_0 =\lfloor \delta_0 M\rfloor$ and $r = \lfloor (\d_1 -
 \delta_0)M\rfloor$.
 We  consider a
Markov chain, $\Xi = \lbr \Xi_0 ,\Xi_1, \dots, \Xi_r\rbr $ on $\cS_M$,
 such that
$\Xi_\tau\in \cS_M [K_0 +\tau ]\equiv \cS_M^\tau$ for $\tau  = 0,1,\dots , r$.
Let $\mu_\tau $ be the canonical measure,
\be\Eq(put.25)
 \mu_\tau  (\xi )= \frac{w (\xi ) \1_{\lbr \xi\in \cS_M^\tau \rbr}}{Z_{\tau }} .
\ee
We take $\nu_0= \m_0$ as the initial distribution of $\Xi_0$ and,
following \eqref{qRates},
we define transition rates,
\be
\label{qRatestau}
 q_\tau  (\xi_\tau ,\theta_j^+\xi_\tau  )= \frac{{\rm e}^{2\epsilon g_j}}
{\sum_{i\in\Lambda^- (\xi_\tau  )}{\rm e}^{2\epsilon g_i}}.
\ee
We denote by  $\P$ the  law of this Markov chain and
let $\nu_\tau$ be the distribution of $\Xi_\tau $ (which is concentrated
on $\cS_M^\tau $), that is, $\nu_{\tau }(\xi ) = \P \lb \Xi_\tau =
\xi\rb$.   The propagation of errors
along paths of our chain  is then quantified in terms of
$\psi_\tau (\cdot) \equiv\nu_\tau (\cdot )/ \mu_\tau (\cdot  )$.
\begin{proposition}
For every $\tau=1, \dots, r$ and each $\xi\in \cS_M^\tau$ define
\be
\label{ABtau}
\BB_\tau (\xi ) \equiv \sum_{i=1}^M {\rm e}^{2\epsilon g_i}\1_{\lbr i\in\Lambda^{-} (\xi )\rbr}\quad
\text{and}\quad \AA_\tau =\mu_\tau\lb \BB_\tau(\cdot )\rb
= \sum_{i=1}^M {\rm e}^{2\epsilon g_i}
\mu_\tau \lb i\in\Lambda^-  (\cdot )\rb .
\ee
Then there exists
$c =c (\delta_0 ,\d_1 )$ such that the following holds: For any
 trajectory, $\underline{\xi} = (\xi_0 , \dots ,\xi_r )$,  of positive
 probability under $\P$, it holds that
\be
\label{eq:psiBound}
\psi_\tau (\xi_\tau )\leq \left[\frac{\AA_0}{\BB_0 (\xi_0 )}\right]^\tau
{\rm e}^{c\epsilon\tau^2/M},
\ee
for all $\tau = 0, 1, \dots , r$.
\end{proposition}
\noindent
\begin{proof}
By construction, $\psi_0\equiv 1$.
Let $\xi_{\tau +1}\in \cS_M^{\tau +1}$.   Since $\nu_{\t}$ satisfies
the recursion
\be\Eq(put.26)
\nu_{\tau +1} (\xi_{\tau +1})= \sum_{j\in\Lambda^+ (\xi_{\tau +1})}
\nu_\tau (\theta_j^-\xi_{\tau +1} )q_\tau (\theta_j^-\xi_{\tau +1}, \xi_{\tau +1}),
\ee
it follows that $\psi_{\t}$ satisfies
\bea
\nonumber
\psi_{\tau +1}(\xi_{\tau +1})&=&
\sum_{j\in\Lambda^+ (\xi_{\tau +1})}
\frac{\nu_{\tau}(\theta_j^-\xi_{\tau +1} )
q_\tau (\theta_j^-\xi_{\tau +1}, \xi_{\tau +1})}
{\mu_{\tau+1}( \xi_{\tau +1})}\\
\nonumber
&=&
\sum_{j\in\Lambda^+ (\xi_{\tau +1})}\frac{\mu_{\tau}
(\theta_j^-\xi_{\tau +1} )q_\tau (\theta_j^-\xi_{\tau +1}, \xi_{\tau +1})}
{\mu_{\tau+1}( \xi_{\tau +1})}\psi_\tau
(\theta_j^-\xi_{\tau +1}).
\eea
By our choice of transition probabilities in \eqref{qRatestau},
\be
\label{numuRatio}
 \frac{\mu_{\tau}(\theta_j^-\xi_{\tau +1} )
q_\tau (\theta_j^-\xi_{\tau +1}, \xi_{\tau +1})}{\mu_{\tau+1}( \xi_{\tau +1})}=
\frac{Z_{\tau+1}}{Z_{\tau}}\lbr \sum_{i\in \Lambda^{-} (\theta_j^-\xi_{\tau +1})}
{\rm e}^{2\epsilon g_i}\rbr^{-1} .
\ee
Recalling  that $\left|\Lambda^+ (\xi_{\tau })\right| \equiv \left|\Lambda^+_{\tau}\right|
= K_0 +\tau$ does not depend
on the particular value of $\xi_{\tau }$,
\bea
\nonumber
\frac{Z_{\tau+1}}{Z_{\tau}}&=& \frac{1}{Z_\tau}\sum_{\xi\in\cS_M^{\tau+1 }}
w(\xi ) =
\frac{1}{Z_\tau}\sum_{\xi\in\cS_M^{\tau +1}}\frac{1}{\left| \Lambda^+ (\xi )\right|}
\sum_{j\in \Lambda^+ (\xi )} w(\theta_j^{-}\xi){\rm e}^{2\epsilon g_j}\\
\nonumber
&=&
\frac{1}{Z_\tau}\sum_{\xi\in\cS_M^\tau } w( \xi )\cdot
 \frac{1}{\left| \Lambda^+_{\tau+1}
\right|}
\sum_{j\in \Lambda^{-} (\xi )} {\rm e}^{2\epsilon g_j}=
\mu_\tau\lb \frac{1}{\left|\Lambda^+ (\xi_{\tau+1})\right|}
 \sum_{j\in \Lambda^{-} (\cdot )} {\rm e}^{2\epsilon g_j}\rb .
\eea
We conclude that the right hand side of \eqref{numuRatio} equals
\be\Eq(put.27)
\frac{1}{\left|\Lambda^+ (\xi_{\tau+1})\right|}\cdot
\frac{ \mu_\tau\lb
 \sum_{i\in \Lambda^{-} (\cdot )} {\rm e}^{2\epsilon g_i}\rb}
{\sum_{i\in \Lambda^{-}({\theta_j^{-}\xi_{\tau+1})} } {\rm e}^{2\epsilon g_i }}
=  \frac{1}{\left|\Lambda^+ (\xi_{\tau+1})\right|}\cdot
\frac{\AA_\tau }{\BB_{\tau}( \theta_j^{-}\xi_{\tau+1})} .
\ee
As a result,
\be\Eq(put.28)
 \psi_{\tau +1} (\xi_{\tau+1})=
\frac1{\left|\Lambda_+ (\xi_{\tau +1} )\right|}
\sum_{j\in \Lambda_+ (\xi_{\tau +1} )}
\frac{\AA_\tau}{\BB_\tau (\theta_j^-\xi_{\tau +1} )}
\psi_{\tau} (\theta_j^- \xi_{\tau +1} ) .
\ee
Iterating the above procedure we arrive to the following conclusion:
Consider the set,
$\DD (\xi_{\tau +1})$,
of all paths, $\underline{\xi} = (\xi_0 , \dots, \xi_{\tau} ,\xi_{\tau
  +1})$, of positive probability
from $\cS_M^0$ to $\cS_M^{\tau+1}$ to $\xi_{\tau+1}$.
The number, $D_{\tau +1}\equiv \left|\DD (\xi_{\tau +1 })\right|$,
of such paths does not depend on $\xi_{\tau+1}$.
Then, since $\psi_0 \equiv 1$,
\be\Eq(put.29)
 \psi_{\tau +1} (\xi_{\tau +1}) = \frac1{D_{\tau +1}}
\sum_{\underline{\xi}\in \DD (\xi_{\tau +1 })}
\prod_{s=0}^\tau\frac{\AA_s}{\BB_s (\xi_s )} .
\ee
We claim that
\be
\label{eq:RatioReduction}
\frac{\AA_s}{\BB_s (\xi_s )}  = \lb 1 + \frac{O(\epsilon )}{M}\rb
\frac{\AA_{s-1}}{\BB_{s-1} (\xi_{s -1})} ,
\ee
uniformly in all the quantities under consideration.
Once \eqref{eq:RatioReduction} is verified,
\be\Eq(put.30)
 \psi_{\tau} (\xi_{\tau} )\leq {\rm e}^{O(\epsilon )\tau^2/M}\max_{\xi_0\sim \xi_{\tau}}
\left[\frac{\AA_0}{\BB_0 (\xi_0 )}\right]^\tau ,
\ee
where for $\xi_0\in\cS_M^0$, the relation $\xi_0\sim \xi_{\tau}$ means
that there is a path of positive probability from $\xi_0$ to
$\xi_\tau$.  But all such $\xi_0$'s differ  at most in
$2\tau$ coordinates. It is then straightforward to see that if $\xi_0\sim
\xi_{\tau}$ and
$\xi_0^\prime \sim \xi_{\tau}$, then
\be\Eq(put.31)
 \frac{\BB_0 (\xi_0 )}{\BB_0 (\xi_0^\prime)} \leq  {\rm
   e}^{O(\epsilon )\tau/M},
\ee
and \eqref{eq:psiBound} follows.

It remains to prove \eqref{eq:RatioReduction}.  Let $\xi\in\cS_M^s$ and $\xi^\prime =
\theta_j^-\xi\in\cS_M^{s-1}$.  Notice, first of all, that
\be\Eq(put.32)
 \BB_{s-1}(\xi^\prime) - \BB_s (\xi ) = {\rm e}^{2\epsilon g_j} = 1 + O (\epsilon ) .
\ee
Similarly,
\bea
\nonumber
 \AA_{s-1} -\AA_s  &=&  \sum_{i=1}^M {\rm e}^{2\epsilon g_i}\lbr
\mu_{s-1}(i\in\Lambda^- ) - \mu_{s}(i\in\Lambda^- )\rbr\,
\\ \nonumber
&=&
1 + \sum_{i=1}^M\lb {\rm e}^{2\epsilon g_i}-1\rb \lbr
\mu_{s-1}(i\in\Lambda^- ) - \mu_{s}(i\in\Lambda^- )\rbr .
\eea
By usual local limit results for independent Bernoulli variables,
\be
\label{eq:mussplus1}
\mu_{s-1}(i\in\Lambda^- ) - \mu_{s}(i\in\Lambda^- )= O\lb\frac1{M}\rb ,
\ee
uniformly in $s=1, \dots,  r-1$ and $i=1, \dots ,M$.
Hence, $ \AA_{s-1} -\AA_s = 1 + O(\epsilon )$.

Finally,  both $\AA_{s-1}$ and $\BB_{s-1}({\xi^\prime})$ are (uniformly )  $O(M)$, whereas,
\be\Eq(put.33)
 \AA_{s-1} - \BB_{s-1}({\xi^\prime})= \sum_{i=1}^M\lb {\rm e}^{2\epsilon g_i} -1\rb
\lbr \mu_{s-1}(i\in\Lambda^- ) - \1_{\lbr i\in \Lambda^- (\xi^\prime )\rbr}\rbr =
O(\epsilon ) M .
\ee
Hence,
\be\Eq(put.34)
 \frac{\AA_s}{\BB_s (\xi )} = \frac{\AA_{s-1} - 1 + O(\epsilon )}
{\BB_{s-1} (\xi^\prime ) -1+ O(\epsilon )}
= \frac{\AA_{s-1}}{\BB_{s-1} (\xi^\prime )}\lb 1+ \frac{O(\epsilon)}{M}\rb ,
\ee
which is \eqref{eq:RatioReduction}.
\end{proof}

\paragraph{Back to the big microscopic chain.} Going back to \eqref{PsilProduct} we infer
that the corrector of the {\em big} chain $\Sigma$ satisfies the following upper bound: Let
 $\us = (\sigma_0 ,\sigma_1 , \dots )$ be a trajectory of $\Sigma$ (as sampled from $\P_{\ubx}$).
Then, for every $\ell = 0,1, \dots ,\ell_B -1$,
\be
\label{eq:PsilTerm}
 \Psi_\ell (\sigma_\ell )
\leq \exp\lbr c\epsilon \sum_{j=1}^n \frac{\tau_j [\ell ]^2}{M_j}\rbr
\prod_{j=1}^n \left[\frac{\AA_0^{(j)}}{\BB_0^{(j)} (\s_0^{(j)})}\right]^{\tau_j [\ell ]} ,
\ee
 where $M_j = \left|\Lambda_j\right| =\rho_j N$,
\be\Eq(put.35)
 \AA_0^{(j)} = \sum_{i\in \Lambda_j}{\rm e}^{2\tilde{h}_i}
\mu_{\b ,N}^{\bx_0 \! (j)}\!\lb i\in\Lambda^{-}_j\rb,
\quad\text{and}\quad
\BB_0^{(j)} (\s_0^{(j)}) =\sum_{i\in \Lambda_j}
{\rm e}^{2\tilde{h}_i} \1_{\lbr i\in\Lambda^{-}_j (\s_0^{(j)})\rbr} .
\ee
Of course, $\AA_0^{(j)}  = \mu_{\b ,N}^{\bx_0 \! (j)}\lb \BB_0^{(j)}\rb$. It is enough
to control the first order approximation,
\be
\label{eq:Yj}
\left[\frac{\AA_0^{(j)}}{\BB_0^{(j)} (\s_0^{(j)})}\right]^{\tau_j [\ell ]}
\approx
\exp\lbr -\tau_j [\ell ] \frac{ \BB_0^{(j)} (\s_0^{(j)}) - \AA_0^{(j)}}{\BB_0^{(j)} (\s_0^{(j)})}
\rbr \equiv
\exp\left({\tau_j [\ell ] Y_j}\right).
\ee
The variables $Y_1 ,\dots ,Y_n$ are independent once $\bx_0$ is fixed.
Thus, in view of our target, \eqref{CorrectorBound}, we need to derive
 an upper bound of order  $(1 + O(\epsilon))$ for
\bea
\nonumber
&&\E^{\ubx} \sum_{\ell = 0}^{\ell_B -1}
\exp\lbr c\epsilon \sum_{j=1}^n \frac{\tau_j [\ell ]^2}{M_j} + \sum_{j=1}^n
\tau_j [\ell ] Y_j \rbr
\phi_{\bA ,\bB}(\bx_{\ell }, \bx_{\ell +1})\\
\label{eq:RandomTerm}
&&\quad =
\sum_{\ell = 0}^{\ell_B -1}
\exp\lbr c\epsilon \sum_{j=1}^n \frac{\tau_j [\ell ]^2}{M_j}\rbr
\prod_1^n\mu_{\b ,N}^{\bx_0 (j)}\lb {\rm e}^{\tau_j [\ell ]Y_j}\rb
\phi_{\bA ,\bB}(\bx_{\ell }, \bx_{\ell +1}),
\eea
which holds with
$\P_N^{\frf_{\bA ,\bB}}$-probability of order $1 - O(\epsilon )$.

\subsection{Good mesoscopic trajectories}
\label{sub:Good}
A look at \eqref{eq:RandomTerm} reveals what is to be expected from
{\em good} mesoscopic trajectories. First of all,  we may assume
that it
passes through the
tube
$G_N^0$ (see \eqref{GNnot}) of $\bz^*$.  In particular, $\bx_0\in G_N^0$.
Next,
by our construction of the mesoscopic
chain $\P^{\frf_{\bA , \bB}}_N$, and in view of \eqref{fine.9} and
\eqref{fine.10}, the step frequencies,
 $\tau_j [\ell ]/\ell$, are, on  average, proportional to $\rho_j$. Therefore,
there exists a  constant, $C_1$, such that,
up to
exponentially  negligible $\P^{\frf_{\bA , \bB}}_N$-probabilities,
 \be
\label{Good3}
\max_j \frac{\tau_j [\ell_B]}{M_j }\leq C_1
\ee
holds.

\paragraph{A bound on microscopic moment-generating functions.}
  We will now use the  estimate \eqv(Good3) to obtain an
 upper bound on the product terms in
\eqref{eq:RandomTerm}.  Clearly, $\BB_0^{(j)} (\s_0^{(j)})= (1
 +O(\epsilon ))M_j$,  uniformly in
$j$ and $\s_0^{(j)}$. Thus, by \eqref{eq:Yj},
\be\Eq(put.36)
 Y_j (1+ O(\epsilon )) = \frac1{M_j}\sum_{i\in\Lambda_j} \lb1- {\rm e}^{2\wt{h}_i}\rb
\lb\1_{\lbr \sigma (i) = -1\rbr} - \mu_{\b ,N}^{\bx_0 (j)} (\sigma (i)
 = -1)\rb \equiv \wt{Y}_j.
\ee
Now, for any $t\geq 0$,
\be
\label{eq:logBound}
 \ln\mu_{\b ,N}^{\bx_0 (j)} \lb {\rm e}^{t\wt{Y}_j}\rb \leq \frac{t^2}{2M_j^2}
\max_{s\leq t}\Va_{\b ,N}^{\bx_0 (j) ,s}\lb \sum_{i\in\Lambda_j}
 \lb1- {\rm e}^{2\wt{h}_i}\rb \1_{\lbr \sigma (i) = -1\rbr}\rb ,
\ee
where $\Va_{\b ,N}^{\bx_0 (j) ,s}$ is the variance with respect to the
tilted conditional
measure, $ \mu_{\b ,N}^{\bx_0 (j) ,s}$, defined through
\be\Eq(put.38)
 \mu_{\b ,N}^{\bx_0 (j) ,s}(f)  \equiv
\frac{\mu_{\b ,N}^{\bx_0 (j) }\lb f{\rm e}^{s \wt{Y}_j} \rb }{
\mu_{\b ,N}^{\bx_0 (j)}\lb {\rm e}^{s \wt{Y}_j} \rb }.
\ee
However, $\mu_{\b ,N}^{\bx_0 (j) ,s}(\cdot )$ is again a conditional
product Bernoulli
measure on $\cS_N^{(j)}$, i.e.,
\be\Eq(put.39)
 \mu_{\b ,N}^{\bx_0 (j) ,s}(\cdot ) = \bigotimes_{i\in\Lambda_j}\B_{p_i (\epsilon ,s)}
\lb ~\cdot~\Big|
\sum_{i\in\Lambda_j} \sigma (i) = N\bx_0 (j)\rb ,
\ee
where
\be\Eq(put.40)
 p_i (\epsilon ,s ) = \frac{{\rm e}^{\wt{h}_i}}{{\rm e}^{\wt{h}_i} +
{\rm e}^{ - \wt{h}_i + \frac{s}{M_j}(1- {\rm e}^{2\wt{h}_i})}} .
\ee
By \eqref{Good3} we need to consider only the case $s/M_j\leq C_1$.   Evidently,
there exists $\delta_1 >0$, such that,
\be
\label{piBound}
\delta_1\leq \min_j \min_{s\leq C_1M_j}\min_{i\in\Lambda_j} p_i (\epsilon ,s)\leq
\max_j \max_{s\leq C_1M_j}\max_{i\in\Lambda_j} p_i (\epsilon ,s)\leq
1-\delta_1.
\ee
On the other hand, since $\bx_0\in G_N^0$, there exists $\delta_2 >0$,
such that
\be
\label{xnotBound}
\delta_2\leq \min_{j}\frac{N \bx_0 (j)}{M_j}\leq \max_{j}\frac{N \bx_0 (j)}{M_j}\leq
1-\delta_2.
\ee
We use  the following  general covariance bound for product of Bernoulli
measures, which can be derived from local limit results in a
straightforward,  albeit
painful manner.

\begin{lemma}\TH(put.50) Let $\delta_1>0$ and $\delta_2>0$ be fixed.
  Then, there exists a constant, $C  = C(\delta_1 ,\delta_2 )<\infty$,
   such that, for all  conditional Bernoulli product measures on $\cS_M$,
   $M\in\N$, of the   form
\be\Eq(put.41)
 \bigotimes_{i=1}^M\B_{p_i}\lb ~\cdot ~\Big|\sum_{k=1}^M \xi_k = 2M_0\rb ,
\ee
with $p_1 ,\dots, p_M\in (\delta_1, 1-\delta_1 )$ and $2M_0 \in
   (-M(1-\delta_2) ,M(1-\delta_2))$,  and for all  $1\leq k< l\leq M $,
it holds that
\be
\label{eq:covBound}
\left| \C{\rm ov} \lb\1_{\lbr \xi_k =-1\rbr} ; \1_{\lbr \xi_l=-1\rbr}\rb\right| \leq \frac{C}{M}.
\ee
\end{lemma}

Going back to \eqref{eq:logBound} we  infer from this that
\be
\label{eq:productFinal}
\prod_1^n\mu_{\b ,N}^{\bx_0 (j)}\lb {\rm e}^{\tau_j [\ell ]Y_j}\rb\leq
\exp\lbr O(\epsilon^2 ) \sum_{j=1}^{n} \frac{\tau_j [\ell ]^2}{M_j}\rbr ,
\ee
uniformly in $\ell =0, \dots ,\ell_B$.

\paragraph{Statistics of mesoscopic trajectories.}
\eqref{eq:RandomTerm} together with  the bound
\eqref{eq:productFinal} suggests the following notion of {\em goodness} of
mesoscopic trajectories $\ubx$:
\begin{definition}
\label{def:Good}
We say that a mesoscopic trajectory $\ubx =(\bx_{-\ell_A }, \dots ,\bx_{\ell_{B}})$ is good,
and write $\ubx\in\TT_{\bA ,\bB}$, if it
passes through $G_N^0$, satisfies \eqref{Good3} (and its analog for the reversed chain)
and, in addition, it satisfies
\be
\label{Good1}
\sum_{\ell = -\ell_A}^{\ell_B-1}\exp\lbr O(\epsilon )\sum_{j=1}^n
\frac{\tau_j [\ell ]^2}{M_j}\rbr
\phi_{\bA ,\bB} (\bx_{\ell} ,\bx_{\ell +1}) \leq  1 + O(\epsilon ) .
\ee
\end{definition}
By construction \eqref{CorrectorBound} automatically holds for any $\ubx\in \TT_{\bA ,\bB}$.
Therefore, our target lower bound  \eqref{capaLB} on microscopic
capacities will follow from
\begin{proposition}
 \label{Good}
Let $\frf_{\bA ,\bB}$ be the mesoscopic flow constructed in Subsections~\ref{sub:FlowDN}
and \ref{sub:Outside}, and let the set of mesoscopic trajectories $\TT_{\bA ,\bB}$ be
 as in Definition~\ref{def:Good}. Then \eqref{AllAreGood} holds.
\end{proposition}

\begin{proof}
By \eqref{pathdecayk} we may assume that there exists $C >0$ such  that, for all
$\ubx$ under consideration and   for all $\ell = -\ell_A , \dots ,\ell_B -1$,
\be
\label{Good2}
\phi_{\bA ,\bB} (\bx_{\ell} ,\bx_{\ell +1}) \leq {\rm e}^{-C
\ell^2 /N
}.
\ee
In view of
\eqref{onepath} it is enough to check
 that
\be
\label{Oepsilon}
\sum_{\ell = 0}^{\ell_B-1}\lb
\exp\lbr O(\epsilon )\sum_{j=1}^n
\frac{\tau_j [\ell ]^2}{M_j}\rbr  -1\rb
\phi_{\bA ,\bB} (\bx_{\ell} ,\bx_{\ell +1}) =  O(\epsilon ) ,
\ee
with $\P_N^{\frf_{\bA ,\bB}}$-probabilities of order $1-\so$.  Fix
$\delta >0$ small and
split the sum on the left hand side of  \eqref{Oepsilon} into two sums
corresponding  to the terms with
$\ell \leq N^{1/2 -\delta}$ and  $\ell > N^{1/2 -\delta}$
respectively. Clearly,
\be
\sum_{j=1}^n
\frac{\tau_j [\ell ]^2}{M_j} = \so,
\ee
uniformly in $0\leq \ell\leq N^{1/2 -\delta}$.  On the other hand,
from our construction
of the mesoscopic flow $\frf_{\bA ,\bB}$, namely from the choice
\eqref{qell} of transition
rates inside $G_N^0$, and from the property \eqref{eq:mincurve}
of the minimizing curve
$\hat\bx (\cdot )$,  it follows that there exists a universal
($\epsilon$-independent) constant,
$K <\infty$, such that
\be
 \P^{\frf_{\bA ,\bB}}_N\lb \max_j\max_{\ell >N^{1/2 - \delta}}\frac{\tau_j [\ell ]}{\ell\rho_j }
 >K\rb = \so .
\ee
Therefore,  up to $\P^{\frf_{\bA ,\bB}}_N$-probabilities of order $\so$, the
inequality
\be
 O(\epsilon )\sum_{j=1}^{n} \frac{\tau_j^2[\ell ]}{M_j} \leq  O(\epsilon )
K^2 \ell^2\sum_{j=1}^n \frac{\rho_j^2}{M_j} = K^2 O(\epsilon )\frac{\ell^2}{N} ,
\ee
holds
uniformly in  $\ell >N^{1/2 - \delta}$ . A comparison with \eqref{Good2} yields
 \eqref{Oepsilon}.
\end{proof}

The last proposition leads to the inequality \eqv(capaLB), which, together
the upper bound given in \eqv(up.70.1),  concludes the proof of Theorem
\thv(CAP-thm).

\section{Sharp estimates on the mean hitting times}

In this section we conclude the proof of Theorem \thv(theorem1).
 To do this we will use Equation \eqv(pot.12)  with
 $A= \cS[m_0^*]$  and $B=\cS[M]$, where $m_0^*$ is a
 local minimum of $F_{\b,N}$ and $M$ is the
  set of minima deeper than $m_0^*$.
 The denominator on the right-hand side of \eqv(pot.12),
the capacity, is
 controlled by Theorem \thv(CAP-thm). What we want to prove now is that
 the equilibrium potential, $h_{A,B}(\s)$, is close to one in the
 neighborhood of the starting set $A$, and so small elsewhere that the
contributions from the sum over $\s$ away from the valley containing
 the set $A$ can be neglected.
 Note that this is not generally true but depends on the choice of
sets $A$ and $B$:  the condition that all  minima $m$
of $F_{\b,N}$
 such that $F_{\b,N}(m)< F_{\b,N}(m_0^*)$
belong to the target set $B$ is crucial.

In earlier work (see \cite{Bo5}) the standard way to estimate the
equilibrium potential $h_{A,B}(\s)$ was to use the
 renewal inequality
$h_{A,B}(\s)\leq \frac{\capa(A,\s)}{\capa(B,\s)}$ and bounds on
 capacities. This bound cannot be used here,
 since the capacities of single points are too small.
We will therefore use another method to cope with this problem.

\subsection{Mean hitting time and equilibrium potential}

Let us start by considering a local minimum $m_0^*$ of the
one-dimensional function  $F_{\b,N}$, and denote by  $M$ the set
of minima $m$ such that  $F_{\b,N}(m)<F_{\b,N}(m_0^*)$.
We then consider the disjoint subsets  $A\equiv \cS[m_0^*]$
and $B\equiv\cS[M]$, and write Eq. \eqv(pot.12) as
\be\Eq(time0)
\sum_{\s\in A} \nu_{A,B} (\s)\E_\s \t_B
=\frac{1}{\capa(A,B)}\sum_{m\in [-1,1]}\sum_{\s\in\cS[m] }
\mu_{\b,N}(\s) h_{A,B}(\s) .
\ee

We want to estimate the right-hand side of \eqv(time0).
This is expected to be of order $\QQ_{\b,N}(m^*_0)$, thus we can
readily do away with all contributions where $\QQ_{\b,N}$ is much smaller.
More precisely, we choose $\d>0$ in such a way that, for all $N$ large enough,
there is no critical point $z$ of $F_{\b,N}$ with
$F_{\b,N}(z)\in\left[F_{\b,N}(m_0^*),F_{\b,N}(m_0^*)+\d\right]$, and
define
\be\Eq(udelta.1)
\cU_\d\equiv \{m: F_{\b,N}(m)\leq F_{\b,N}(m_0^*)+\d\}.
\ee
Denoting by $\cU_\d^c$ the complement of $\cU_\d$, we obviously have
\begin{lemma}\TH(time.1)
\be\Eq(strange.16)
 \sum_{m\in \cU_\d^c}\sum_{\s\in\cS[m] }
\mu_{\b,N}(\s) h_{A,B}(\s)\leq N e^{-\b N\d} \QQ_{\b,N}(m^*_0) .
\ee
\end{lemma}

The main problem is to control the equilibrium potential $h_{A,B}(\s)$
for configurations $\s\in \cS[\cU_\d]$.
To do that, first notice that
\be\Eq(udelta.2)
\cU_\d= \cU_\d(m_0^*)\bigcup_{m\in M} \cU_\d(m),
\ee
where $\cU_\d(m)$ is the connected component of $\cU_\d$ containing $m$
(see Fig. \ref{fig.landscape}).
Note that it can happen that $\cU_\d(m)=\cU_\d(m')$
for two different minima $m,m'\in M$.
\begin{figure}\label{fig.landscape}
\begin{center}
\psfrag{a}{$m_0^*$}
\psfrag{b}{$z$}
\psfrag{c}{$m_1$}
\psfrag{d}{$m_2$}
\psfrag{e}{$\mathcal U_\d(m_0^*)$}
\psfrag{f}{$\mathcal U_\d(m_1)$}
\psfrag{g}{$\mathcal U_\d(m_2)$}
\psfrag{h}{$F_{\b,N}(m)$}
\psfrag{i}{$F_{\b,N}(m_0^*)+\d$}
\psfrag{l}{$-1$}
\psfrag{m}{$1$}

\includegraphics[width=11cm]{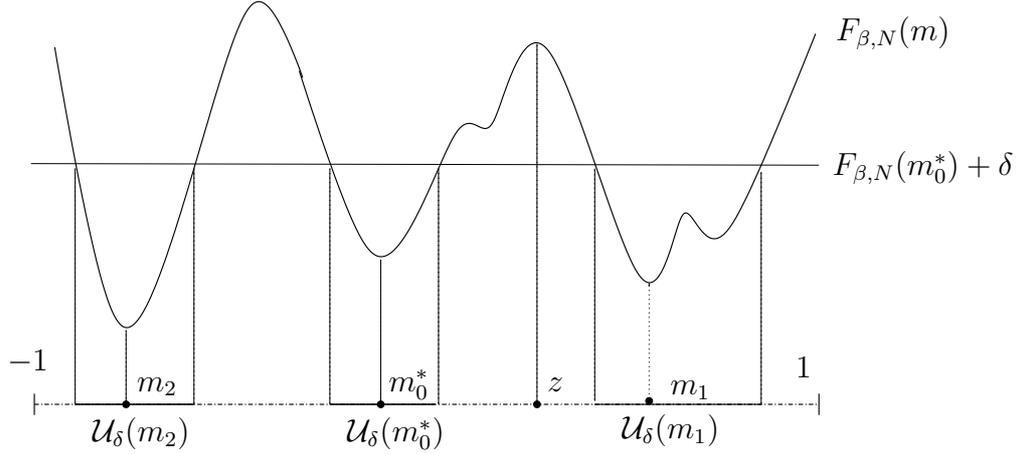}
\caption{Decomposition of the magnetization space
$[-1,1]$: $\mathcal U_\d^c$ is represented by dotted lines,
 while the continuous lines correspond 
$\mathcal U_\d\equiv\cU_\d(m_0^*)\bigcup_{m\in M} \cU_\d(m)$.}
\end{center}
\end{figure}

With this notation we have the following lemma.

\begin{lemma}\TH(time.1.2)
 There exists a constant,  $c>0$, such that,
\begin{itemize}
\item[(i)]  for every $m\in M$,
\be \Eq(time.2)
\sum_{\s\in\cS[\cU_\d(m)] }
\mu_{\b,N}(\s) h_{A,B}(\s) \leq e^{-\b N c} \QQ_{\b,N}(m^*_0), 
\ee
and
\item[(ii)] 
\be \Eq(time.3)
\sum_{\s\in\cS[\cU_\d(m_0^*)] }
\mu_{\b,N}(\s) \left[1-h_{A,B}(\s)\right] \leq e^{-\b N c}
\QQ_{\b,N}(m^*_0).
 \ee
\end{itemize}
\end{lemma}
 The treatment of points (i) and (ii) is completely similar, as
both rely on a rough estimate of the probabilities to leave the
starting well before visiting its minimum, and it will be discussed
in the next section.

Assuming Lemma \thv(time.1.2), we can readily
conclude the proof of Theorem \thv(theorem1).
Indeed, using \eqv(time.2) together with \eqv(strange.16),
we obtain the upper bound
 \bea\Eq(upper1)
\sum_{\s\in S_N} \mu_{\b,N}(\s) h_{A,B}(\s) &\leq&
\sum_{m\in \cU_\d(m_0^*)}\QQ_{\b,N}(m) +O\left(\QQ_{\b,N}(m_0^*)e^{-\b
Nc}\right)
 \nonumber\\
&=& \QQ_{\b,N}(m_0^*)\sqrt{\frac{\pi N}{2\b a(m_0^*)}}(1+o(1)),
 \eea
where $a(m^*_0)$ is given in \eqv(static.201).
 On the other hand, using \eqv(time.3), we get the corresponding
 lower bound
  \bea\Eq(lower1)
\sum_{\s\in S_N} \mu_{\b,N}(\s) h_{A,B}(\s)
&\geq& \sum_{m\in \cU_\d(m_0^*)} \sum_{\s\in\cS[m] }
\mu_{\b,N}(\s)\left[1-(1-h_{A,B}(\s))\right]
\nonumber\\
&\geq& \sum_{m\in \cU_\d(m_0^*)}\QQ_{\b,N}(m)
-O(\QQ_{\b,N}(m_0^*) e^{-\b Nc})
 \nonumber\\
&=& \QQ_{\b,N}(m_0^*)\sqrt{\frac{\pi N}{2\b a(m_0^*)}}(1+o(1)).
 \eea
From  Equation \eqv(static.16)
for $\QQ_{\b,N}(m_0^*)$
 and Equation \eqv(cap1.2) for  $\capa(A,B)$, we finally obtain
 \bea\Eq(upper2)
\E_{\nu_{A,B}}\t_B
&=&\sum_{\s\in S_N}\frac{\mu_{\b,N}(\s)h_{A,B}(\s)}{\capa(A,B)}
\nonumber\\
 &=& \exp\left(\b N\left( F_{\b,N}(z^*)-F_{\b,N}(m_0^*)\right)\right)
\nonumber\\
&\times&\frac {2\pi N}  {\b |\hat\g_1|} \sqrt{\frac{\b\E_h
\left(1-\tanh^2
\left(\b(z^*+h)\right)\right)-1} {1-\b\E_h
\left(1-\tanh^2\left(\b(m_0^*+h)\right)\right)}}(1+o(1)),
\eea
which proves  Theorem \thv(theorem1).

\subsection{Upper bounds on harmonic functions.}

We now prove Lemma \thv(time.1.2) giving
a detailed proof only for (i), the proof of (ii)
being completely analogous.
This requires, for the first time in this
paper, to get an estimate on the minimizer of the Dirichlet form,
the harmonic function $h_{A,B}(\s)$.

First note that, since $h_{A,B}(\s)\equiv \P_\s(\t_A<\t_B)$
for all $\s\notin A\cup B$,
the only non zero contributions to the sum in (i)
come from those  sets  $\cU_\d(m)$ (at most two)
whose corresponding  $m$ is such that
there are no minima of $M$ between $m_0^*$ and $m$.
By symmetry we can just analyze one of these two sets,
denoted by $\cU_\d(m^*)$, assuming  for
definiteness that $m_0^*<m^*$.

Note also that since  $h_{A,B}(\s)=0$  for all $\s$ such that $m^*\leq m(\s)$,
the problem  can be reduced  further on to the set
\be\Eq(definitionU)
\cU_\d^- \equiv \cU_\d(m^*)\cap \{m: m<m^*\}.
\ee
Define the mesoscopic counterpart of $\cU_\d^-$,
namely, for fixed $m^*\in M$ and $n\in\N$,
let $\bm^*\in\Gamma_N^n$
be the minimum of $F_{\b,N}(\bx)$ correspondent to $m^*$,
and define
\be\Eq(udelta.3)
\bU_\d\equiv \bU_\d(\bm^*)\equiv \{\bx\in\G_N^n: m(\bx)\in \cU_\d^-\}.
\ee
We write the boundary of $\bU_\d$ as
$\partial \bU_\d= \partial_A\bU_\d\sqcup \partial_B\bU_\d$,
where $\partial_B\bU_\d= \partial\bU_\d\cap \bB$, and observe that, for all
$\s\in\cS[\bU_\d]$
 \be\Eq(time.11)
h_{A,B}(\s)=\P_\s[\t_A<\t_B]
\leq
\P_\s[\t_{S[\partial_A \bU_\d]}<\t_{S[\partial_B \bU_\d]}].
\ee
Let $\max_{\ell} \rho_\ell\ll\th(\e)\ll 1$,
and for $\th\equiv\th(\e)$ define
\be\Eq(time.10.1)
\bG_\th\equiv \left\{\bm\in \bU_\d:
\sum_{\ell=1}^n\frac{(\bm_\ell-\bm_{\ell}^*)^2}{\rho_\ell}
\leq \frac{\e^2}{\th}\right\}.
\ee
As before, we denote  by $\partial\bG_\th$ the boundary of $\bG_\th$,
and write $\partial \bG_\th= \partial_A \bG_\th\sqcup \partial_B\bG_\th$,
where $\partial_B \bG_\th= \partial\bG_\th\cap \bB$
(see Fig. \ref{fig.potential}).

\begin{figure} \label{fig.potential}
\begin{center}
\psfrag{a}{$\bm_0^*$}
\psfrag{b}{$\bA=\{\bx: m(\bx)= m_0^*\}$}
\psfrag{c}{$\bU_\d(\bm_0^*)$}
\psfrag{d}{$\bU_\d\equiv\bU_\d(\bm^*)$}
\psfrag{e}{$\partial_A \bG_\th$}
\psfrag{f}{$\partial_B \bG_\th$}
\psfrag{g}{$\bm^*$}
\psfrag{h}{$\bG_\th$}
\psfrag{i}{$\bB=\{\bx: m(\bx)= m^*\}$}
\psfrag{l}{$\partial_A \bU_\d$}
\psfrag{m}{$\partial_B \bU_\d$}

\includegraphics[width=12cm]{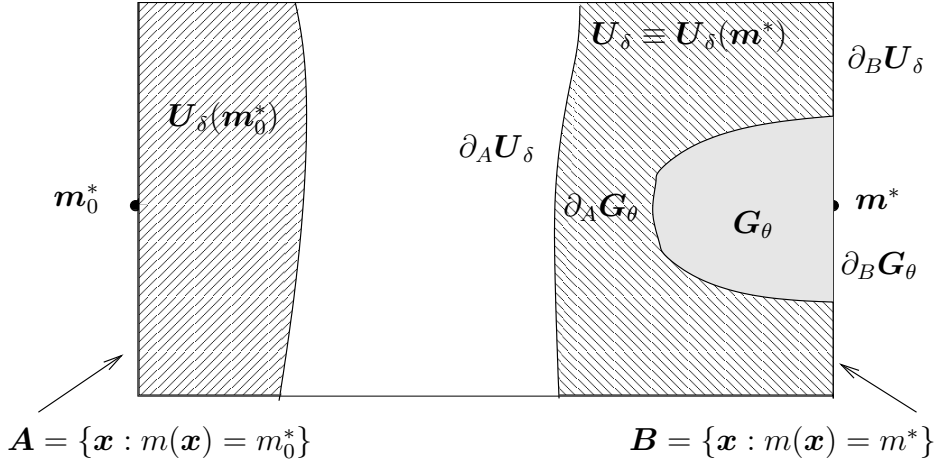}
\caption{Neighborhoods of $\bm_0^*$ and $\bm^*$ in
the space $\G_N^n$. Here we denoted by $\bU_\d(\bm_0^*)$
the mesoscopic counterpart of $\mathcal U(m_0^*)$}.
\end{center}
\end{figure}

The strategy to control the equilibrium potential, $\P_\s(\t_A<\t_B)$,
consists in estimating the probabilities
$\P_\s[\t_{A}<\t_{S[\partial_A \bG_\th]\cup B}]$,
for $\s\in\cS[\bU_\d\setminus\bG_\th]$,
and $\P_\s[\t_{S[\partial_A \bG_\th]}<\t_{B}]$, for $\s\in\bG_\th$,
in order to apply a renewal argument and to get from these estimates
a bound on the probability of the original event.

Proceeding on this line, we state the following:
 \begin{proposition}\TH(claim.01)
For any $\a\in(0,1)$, there exists $n_0\in\N$, such that the
inequality 
\be\Eq(claim.02)
\P_\s(\t_{A}<\t_{S[\partial_A \bG_\th]\cup B})
\leq  e^{-(1-\a)\b N \left[F_{\b,N}(m_0^*)+\d- F_{\b,N}(\bm(\s))\right]}
\ee
holds for all $\s\in \cS[\bU_\d\setminus\bG_\th]$, $n\geq n_0$,
and for all $N$ sufficiently large.
\end{proposition}

\paragraph{Proof of Proposition \thv(claim.01): Super-harmonic barrier functions.}
Throughout the next computations, $c$, $c'$ and $c''$ will denote positive
constants which are independent on $n$ but may depend on $\b$ and on
the distribution of $h$.
 The particular value of $c$ and $c'$ may change from line to line as
the discussion progresses.

We first observe
that, for all $\s\in \cS[\bU_\d\setminus\bG_\th]$,
 \be\Eq(time.11.1)
\P_\s[\t_A<\t_{\cS[\partial_A \bG_\th]\cup B}]
\leq
\P_\s[\t_{S[\partial_A\bU_\d]}<\t_{\cS[\partial_A \bG_\th]\cup B}].
\ee
The probability in the r.h.s. of \eqv(time.11.1) is
the main object of investigation here.
The idea which is beyond the proof of bound \eqv(claim.02)
is quite simple.
Suppose that  $\psi$ is a bounded
super-harmonic function defined on $\cS[\bU_\d\setminus\bG_\th]$,
i.e.
\be\label{supermart}
(L \psi)(\s)\leq 0\quad\quad \mbox{for all }\s\in \cS[\bU_\d\setminus\bG_\th].
\ee
Then  $\psi(\s_t)$ is a supermartingale, and
$T\equiv\t_{S[\partial_A \bU_\d]}\land \t_{\cS[\partial_A \bG_\th]\cup B}$ is an integrable
stopping time, so that,   by Doob's  optional stopping theorem,
  $\forall\, \s\in\cS[\bU_\d\setminus\bG_\th]$,
 \be
 \Eq(superm.1)
 \E_\s \psi(\s_T)\leq \psi(\s).
\ee
On the other hand,
 \be\Eq(superm.2)
\E_\s \psi(\s_T) \geq \min_{\s'\in \cS[\partial_A\bU_\d]} \psi(\s')
\P_\s (\t_{\cS[\partial_A\bU_\d]}<\t_{\cS[\partial_A \bG_\th]\cup B}),
 \ee
and hence
\be\Eq(superm.3)
\P_\s (\t_{\cS[\partial_A \bU_\d]}<\t_{\cS[\partial_A \bG_\th]\cup B}) \leq
\max_{\s'\in \cS[\partial_A\bU_\d]} \frac{\psi(\s)}{\psi(\s')}.
 \ee
The problem is to find a  super-harmonic function in order to get
a suitable bound in \eqv(superm.3).

\begin{proposition}\TH(time.20)
For any $\a\in(0,1)$, there exists $n_0\in\N$ such that
the function $\psi(\s)\equiv \phi(\bm(\s))$,
with $\phi:\R^n \mapsto\R$  defined as %
\be\Eq(superm.5)
 \phi(\bx)\equiv e^{(1-\a)\b NF_{\b,N}(\bx)},
\ee
is  super-harmonic in $\cS[\bU_\d\setminus\bG_\th]$
for all $n\geq n_0$ and $N$ sufficiently large.
\end{proposition}

The proof of Proposition
 \thv(time.20) will involve computations with  differences
 of the functions $F_{\b,N}$. We therefore first collect some
   elementary properties that we will use later. First we need some
   control on the second derivative of this function.
From
\eqv(fine.16) we infer that
\be\Eq(second.0)
\frac{\del^2 F_{\b,N}(\bx)}{\del \bx_\ell^2}= \frac{2}{N}
\left( -1+ \frac{1}{\b\rho_\ell} I''_{N,\ell}(\bx_\ell/\rho_\ell)
\right).
\ee
Thus all the potential problems come from the function $I_{N,\ell}$.

\begin{lemma}\TH(large-I)
For any $y\in (-1,1)$,
\be\Eq(boundarycase.4)
\tanh^{-1}(y)-\b\e
\leq I'_{N,\ell}(y)\leq
\tanh^{-1}(y)+\b\e,
\ee
In particular, as $y\rightarrow \pm 1$,  $I'_{N,\ell}(y)\rightarrow \pm\infty$.
\end{lemma}

\begin{proof}
Recall that  $I'_{N,\ell}(y)= U'^{-1}_{N,\ell}(y)$.
Set $I'_{N,\ell}(y)\equiv t$.
Then
\be\Eq(boundarycase.31)
y= \frac{1}{|\L_\ell|}
\sum_{i\in\L_\ell} \tanh(t +\b \tilde h_i)
\ee
and hence
\be\Eq(boundarycase.3)
\tanh( t -\b\e)\leq
y\leq \tanh( t +\b\e),
\ee
or, equivalently,
\eqv(boundarycase.4),
which proves the lemma.
\end{proof}

\begin{lemma}\Eq(second.1)
For any $y\in (-1,1)$ we have that
\be\Eq(second.2)
0\leq I''_{N,\ell}(y)\leq \frac 1{1-\left(|y|+\e \b(1-y^2)\right)^2}.
\ee
In particular,
for all $y\in[-1+\nu,1-\nu]$,
with  $\nu\in (0,1/2)$,
\be\Eq(second.5)
 0\leq I''_{N,\ell}(y) \leq \frac{1}{2\nu+\nu^2 +O(\e)}
 \leq c,
\ee
and, for all $y\in(-1,-1+\nu]\cup[1-\nu,1) $,
\be\Eq(second.4)
 0\leq I''_{N,\ell}(y)\leq \frac {1}{1-|y|}.
\ee
\end{lemma}

\begin{proof} We consider only the case $y\geq 0$, the case $y<0$ is
  completely analogous. Using  the relation
$I''_{N,\ell}(x)= \left(U''_{N,\ell}(I'_{N,\ell}(x))\right)^{-1}$ and
setting $t_\ell\equiv I'_{N,\ell}(y)
\mbox{arctanh}(y)$, and using Lemma \thv(large-I),
we  obtain
\bea\Eq(second.2.1)
 I''_{N,\ell}(y)&=&
 \frac{1}{\sfrac{1}{|\L_\ell(\bx)|}\sum_{i\in\L_\ell(\bx)}
 (1-\tanh^2(\b\tilde h_i +t_\ell))}
 \nonumber\\
 &\leq&
 \frac{1}{1-\tanh^2(\e\b +t_\ell)}
\nonumber\\
 &\leq&
 \frac{1}{1-\tanh^2(\tanh^{-1}(y)+2\e\b)}
\nonumber\\
&\leq&
 \frac{1}{1-\left(y+ 2\e\b\tanh'(\tanh^{-1}(y))\right)^2}
 \nonumber\\
 &=&
 \frac{1}{1-\left(y+ 2\e\b(1-y^2)\right)^2},
\eea
where we used that $\tanh$ is monotone increasing. The remainder of
the proof is elementary algebra.
\end{proof}

Let us define, for all $\bm$ such that $\bx_\ell/\rho_\ell\in [-1,1-2/N]$,
\be\Eq(def-g)
g_\ell(\bx) \equiv \sfrac{N}{2}\left(
 F_{N,\b}(\bx +\bfe_\ell)-F_{N,\b}(\bx)\right).
\ee
Lemma \thv(second.1) has the following corollary.

\begin{corollary}\TH(second.15)
\begin{itemize}
\item[(i)] If $\bx_\ell/\rho_\ell\in [-1+\nu,1-\nu]$, with $\nu>0$,
then
\be\Eq(grad.3)
g_\ell(\bx) =
-x -\bar h_\ell +\sfrac{1}{\b}I'_{N,\ell}(\bx_\ell/\rho_\ell) + O(1/N).
\ee
\item [(ii)] If  $\bx_\ell/\rho_\ell\in [-1,-1+\nu]\cup
  [1-\nu,1-2/N]$,
then
 \be\Eq(grad.3.1)
g_\ell(\bx) =
-x -\bar h_\ell +\sfrac{1}{\b}I'_{N,\ell}(\bx_\ell/\rho_\ell) + O(1),
\ee
where $O(1)$ is independent of $N,n$, and $\nu$.
\item[(iii)]
If $\bx_\ell/\rho_\ell\in [-1+\nu,1-\nu]$, with $\nu>0$,
then there exists $c<\infty$, independent of $N$, such that
\be\Eq(second.16)
|g_\ell (\bx)- g_\ell(\bx-\bfe_\ell)|\leq \frac{c}{N}.
\ee
\item [(iv)] If  $\bx_\ell/\rho_\ell\in [-1,-1+\nu]\cup
  [1-\nu,1-2/N]$,
then
\be\Eq(second.17)
|g_\ell (\bx)- g_\ell(\bx-\bfe_\ell)|\leq C,
\ee
where $C$ is a numerical constant independent of $N,n$, and $\nu$.
\end{itemize}
\end{corollary}

The proof of this corollary is elementary and will not be detailed.

The usefulness of (ii) results from the fact that $|I'_{N,\ell}|$ is
large on that domain. More precisely, we have the following lemma.

\begin{lemma}\TH(boundary-g.1)
There exists $\nu>0$, independent of $N$ and $n$, such that, if
$\bx_\ell/\rho_\ell>1-\nu$, then
$g_\ell(\bx)$ is strictly increasing in $\bx_\ell$ and tends to
$+\infty$ as $\bx_\ell/\rho_\ell\uparrow +1$; similarly
 if
$\bx_\ell/\rho_\ell<-1+\nu$, then
$g_\ell(\bx)$ is strictly decreasing in $\bx_\ell$ and tends to
$-\infty$ as $\bx_\ell/\rho_\ell\downarrow -1$.
\end{lemma}

\begin{proof} Combine (ii) of Corollary  \thv(second.15) with
  Lemma \thv(large-I) and note that $\bar h_\ell$ is bounded by hypothesis.
\end{proof}

The next step towards the proof of Proposition \thv(time.20) is the
following lemma.

\begin{lemma}\TH(grad)
Let $\bm\in\bU_\d\setminus\bG_\th$ and denote by $S(\bm)=\{\ell:
\bm_\ell/\rho_\ell\neq 1\}$.
Then there exists a constant $c\equiv c(\b,h)>0$, independent of $N$
and $n$,
such that the following holds. If
\be\Eq(gradshtein)
\sum_{\ell\not\in  S(\bm)} \rho_\ell \leq  \frac {\e^2}{8\th},
\ee
then
\be\Eq(grad.1)
\sum_{\ell\in S(\bm)} \rho_\ell \left(g_\ell(\bm)\right)^2
\geq c \frac{\e^2}{\th},
\ee
\end{lemma}

\begin{proof}
From the relation
 $I'_{N,\ell}(x)= U'^{-1}_{N,\ell}(x)$, we get that, for all $\ell\in S(\bm)$,
\be\Eq(grad.4)
\bm_\ell
= \frac{1}{N}\sum_{i\in\L_\ell} \tanh\left(\b
\left(g_\ell(\bm)(1+\po(1))+m +
 h_i\right)\right).
\ee
Here $\po(1) $ tends to zero as $N\rightarrow\infty$.

We are concerned about small $g_\ell(\bm)$.
Subtracting $\frac{1}{N}\sum_{i\in\L_\ell} \tanh\left(\b
\left(m + h_i\right)\right)$ on both sides of \eqv(grad.4) and
expanding the right-hand side to first order in $g_\ell(\bm)$, and
then summing over $\ell\in S(\bm)$ , we obtain
\bea\Eq(grad.5)\nonumber
&&\left|m-
\frac{1}{N}\sum_{i=1}^N \tanh\left(\b
\left(m + h_i\right)\right)-\sum_{\ell\not\in
  S(\bm)}\left(\bm_\ell-\frac 1N\sum_{i\in
  \L_\ell}\tanh\left(\b(m+h_i)\right)\right)\right|
 \\&&\quad\leq
c\sum_{\ell\in S(\bm)} \rho_\ell \left|g_\ell(\bm)\right|
\leq
c\left(\sum_{\ell\in S(\bm)} \rho_\ell g_\ell^2(\bm)\right)^{1/2}.
\eea
Notice that the function $m\mapsto m-\frac{1}{N}\sum_{i=1}^N \tanh\left(\b
\left(m + h_i\right)\right)$ has, by
\eqv(static.202), non-zero derivative at $m^*$.
Moreover, by construction, $m^*$ is the only zero
of this function in $\cU_\d^-(m^*)$.
From this observations, together with \eqv(grad.5),
we conclude that
\be\Eq(grad.6)
\left(\sum_{\ell=1}^n \rho_\ell g_\ell^2(\bm)\right)^{1/2}
\geq c |m-m^*|-2\sum_{\ell\not\in S(\bm)}\rho_\ell,
\ee
for some constant $c<\infty$. Here we used the triangle inequality and
the fact that
$\left|\bm_\ell-\frac 1N\sum_{i\in
  \L_\ell}\tanh\left(\b(m+h_i)\right)\right|\leq 2\rho_\ell$. Under
the hypothesis of the lemma, this gives the desired bound if
$|m-m^*|\geq c''\e/\sqrt \th$ for some constant $c''<\infty$.
On the other hand, we can write, for $\ell\in S(\bm)$,
\bea\Eq(grad.8)
\left|\bm_\ell - \bm_\ell^*\right|&\leq&
\frac 1N \sum_{i\in\L_\ell}
\left|\tanh\left(\b\left(g_\ell(\bm)(1+\po(1))+m + h_i\right)\right)
- \tanh\left(\b\left(m + h_i\right)\right)
\right|
\nonumber\\
 &+& \frac 1N \sum_{i\in\L_\ell}
\left|\tanh\left(\b\left(m + h_i\right)\right)
- \tanh\left(\b\left(m^* + h_i\right)\right)
\right|
\nonumber\\
&\leq& c \rho_\ell |m-m^*| +c' \rho_\ell|g_\ell(\bm)|.
\eea
Hence we get the bound
\bea\Eq(grad.9)
\left(\sum_{\ell\in S(\bm)} \rho_\ell g_\ell^2(\bm)\right)^{1/2}
&\geq &
c\left(\sum_{\ell\in S(\bm)}\sfrac{(\bm_\ell-\bm_{\ell}^*)^2}{\rho_\ell}
\right)^{1/2}- c' |m-m^*|
\nonumber\\
&=&
c\left(\sum_{\ell=1}^n\sfrac{(\bm_\ell-\bm_{\ell}^*)^2}{\rho_\ell}-
\sum_{\ell\not\in S(\bm)}\sfrac{(\bm_\ell-\bm_{\ell}^*)^2}{\rho_\ell}
\right)^{1/2}- c' |m-m^*|\nonumber\\
&\geq&
c\left( \e^2/\th-
4\sum_{\ell\not\in S(\bm)}\rho_\ell
\right)^{1/2}- c' |m-m^*|\nonumber\\
&\geq& c\e/\sqrt{2\th} - c' |m-m^*|
\eea
where in the last line we just used that
$\bm\not\in \bG_\th$.
The inequalities \eqv(grad.6) and \eqv(grad.9)
now yield \eqv(grad.1),  concluding
the proof of the lemma.
\end{proof}

\begin{proof}[Proof of Proposition \thv(time.20)]
Let  $\s\in \cS[\bU_\d\setminus\bG_\th]$ and
 set $\bx\equiv \bm(\s)$, so that, for $\psi$ as in
 Proposition \thv(time.20),  $L \psi(\s)= L\phi(\bx)$.
Let $\s^i$ be  the configuration obtained from $\s$
 after a spin-flip at $i$, and introduce the notation
\be\Eq(claim.0)
 L\phi(\bx)= \sum_{\ell=1}^n L_\ell\phi(\bx),
\ee
where
\be\Eq(claim.1)
L_\ell \phi(\bx)
=\sum_{i\in\Lambda_\ell^{-}(\bx)}p_N(\s,\s^i)
 [\phi(\bx+\bfe_\ell)-\phi(\bx)]+\sum_{i\in\Lambda_\ell^{+}(\bx)}
p_N(\s,\s^i) [\phi(\bx-\bfe_\ell)-\phi(\bx)].
\ee
Notice that when $\bx_\ell/\rho_\ell=\pm1$,
then $\L^\pm_\ell(\bx)=\emptyset$
and the summation over $\L^\pm_\ell(\bx)$ in \eqv(claim.1) disappears.

We define the probabilities
 \be\Eq(time.50)
 \P_{\pm,\ell}^\s\equiv \sum_{i\in\Lambda_\ell^{\mp}(\bx)}p_N(\s,\s^i),
\ee
and observe that they are uniformly close to the
mesoscopic rates defined in \eqv(rfcw.11), namely
\be\Eq(claim.2)
e^{-c\e}\leq \frac{\P_{\pm, \ell}^\s}
{r_N(\bx,\bx\pm \bfe_\ell)}\leq e^{c\e},
\ee
for some $c>0$ and $\e=1/n$. Notice also that
\be\Eq(claim.3)
c\rho_\ell \leq \P_{+, \ell}^\s +\P_{-, \ell}^\s\leq c'\rho_\ell.
\ee
With the above notation and using the convention $0/0=0$,  we get
\bea\Eq(claim.4)
L_\ell\phi(\bx)&=&  \phi(\bx)
\P_{+,\ell}^\s\left[\exp{\left(2\b(1-\a)g_\ell(\bx)\right)}-1\right]
 \nonumber\\
 && + \phi(\bx)\P_{-,\ell}^\s\left[\exp{\left(-2\b(1-\a)
 g_\ell(\bx-\bfe_\ell)\right)}-1\right]
 \nonumber\\
&=& \phi(\bx)\left(
\1_{\{\P_{+,\ell}^\s\geq\P_{-,\ell}^\s\}}
\P_{+,\ell}^\s G_\ell^+(\bx)+\1_{\{\P_{-,\ell}^\s>\P_{+,\ell}^\s\}}
\P_{-,\ell}^\s G_\ell^-(\bx)\right)
 \eea
where we introduced the functions
\be\Eq(G+)
G_\ell^+(\bx)=
\exp{\left(2\b(1-\a) g_\ell(\bx)\right)}
-1+ \sfrac{\P_{-,\ell}^\s}{\P_{+,\ell}^\s}
\left(\exp{\left(-2\b(1-\a) g_\ell(\bx-\bfe_\ell)\right)}-1\right)
  \ee
\be\Eq(G-)
G_\ell^-(\bx)=  \exp{\left(-2\b(1-\a) g_\ell(\bx-\bfe_\ell)\right)}-1
+ \sfrac{\P_{+,\ell}^\s}{\P_{-,\ell}^\s}
\left(\exp{\left(2\b(1-\a) g_\ell(\bx)
\right)}-1\right)
  \ee
If $\bx_\ell/\rho_\ell=\pm1$,
the local generator takes the simpler form
\be\Eq(gb)
L_\ell\phi(\bx)
=\left\{
\ba{ll}
 \phi(\bx)\P_{-,\ell}^\s\left[\exp{\left(-2\b(1-\a)
 g_\ell(\bx-\bfe_\ell)\right)}-1\right]& \mbox{if } \bx_\ell/\rho_\ell=1
\\
 \phi(\bx)\P_{+,\ell}^\s\left[\exp{\left(2\b(1-\a)
 g_\ell(\bx)\right)}-1\right]& \mbox{if } \bx_\ell/\rho_\ell=-1
\ea
\right.
\ee
From Lemma \thv(boundary-g.1) and inequalities \eqv(claim.3), 
it follows that, for all $\ell$ such that $\bx_\ell/\rho_\ell=\pm1$,
\be\Eq(gb.1)
L_\ell\phi(\bx)\leq -(1+\po(1))\rho_\ell \phi(\bx).
\ee

Let us now return to the case when $\bx$ is not a boundary point.
By the detailed balance conditions, it holds that
\be\Eq(time.41)
\ba{l}
r_N(\bx,\bx+\bfe_\ell)=
\exp{\left(-2\b g_\ell(\bx)\right)}
r_N(\bx+\bfe_\ell,\bx)\\
r_N(\bx,\bx-\bfe_\ell)=
\exp{\left(2\b g_\ell(\bx-\bfe_\ell)\right)}
r_N(\bx-\bfe_\ell,\bx),
\ea
\ee
which implies, together with  \eqv(claim.2),
\be\Eq(time.41.1)
\ba{l}
\exp{\left(-2\b g_\ell(\bx) -c\e \right)}\leq
\frac{\P_{+,\ell}^\s}{\P_{-,\ell}^\s} \leq
\exp{\left(-2\b g_\ell(\bx) +c\e \right)}\\
\exp{\left(2\b g_\ell(\bx- \bfe_\ell) -c\e \right)}\leq
\frac{\P_{-,\ell}^\s}{\P_{+,\ell}^\s} \leq
\exp{\left(2\b g_\ell(\bx-\bfe_\ell) +c\e \right)}
\ea
\ee
Inserting the last bounds in \eqv(G+) and \eqv(G-),  and with some computations,
we  obtain
\bea\Eq(G+.1)
G_\ell^+(\bx)&\leq& \left(\exp{\left(2\b(1-\a)
g_\ell (\bx)\right)}-1\right)\left(1- \exp{\left( 2\b\a
g_\ell(\bx-\bfe_\ell)\mp c\e\right)}\right)
\\
&& + \exp{\left( 2\b
g_\ell(\bx-\bfe_\ell)\mp c\e\right)}
\left( \exp{2\b(1-\a)\left(g_\ell (\bx)- g_\ell(\bx-\bfe_\ell)
 \right)} -1\right)\nonumber
\eea
\bea\Eq(G-.1)\
G_\ell^-(\bx)&\leq&
\left(\exp{\left(-2\b(1-\a)
g_\ell (\bx-\bfe_\ell)\right)}-1\right)\left(1- \exp{\left(- 2\b\a
g_\ell(\bx)\mp c\e\right)}\right)\quad\quad
\\
&& + \exp{\left(- 2\b
g_\ell(\bx)\mp c\e\right)}
\left(\exp{2\b(1-\a)\left(g_\ell (\bx)- g_\ell(\bx-\bfe_\ell)
 \right)} -1\right)
\nonumber
\eea
where $\mp\equiv -\sign \left(g_\ell(\bx)\right)=
-\sign \left(g_\ell(\bx-\bfe_\ell)\right)$.

For all $\ell$ such that $\bx_\ell/\rho_\ell\in[-1+\nu,1-\nu]$,
we can use \eqv(second.16) to get
\be\Eq(G+.2)
G_\ell^+(\bx)\leq \left(\exp{\left(2\b(1-\a)
g_\ell (\bx)\right)}-1\right)\left(1- \exp{\left( 2\a\b
g_\ell(\bx)\mp c\e\right)}\right) + c/N
\ee
\be\Eq(G-.2)\
G_\ell^-(\bx)\leq
\left(\exp{\left(-2\b(1-\a)
g_\ell (\bx)\right)}-1\right)\left(1- \exp{\left(- 2\a\b
g_\ell(\bx)\mp c\e\right)}\right)+c/N.
\ee
The  right hand sides of both
 \eqv(G+.2) and \eqv(G-.2)
are  negative if and only if $\left|g_\ell\right|>\frac{c\e}{2\a\b}$.
Let us define the index sets
\bea\Eq(indices.1)
S^<&\equiv&\{\ell: \bx_\ell/\rho_\ell\in[-1+\nu,1-\nu],
\left|g_\ell(\bx)\right|\leq\sfrac{c\e}{\a\b}\}\\
S^>&\equiv&\{\ell: \bx_\ell/\rho_\ell\in[-1+\nu,1-\nu],
\left|g_\ell(\bx)\right|>\sfrac{c\e}{\a\b}\}.
\eea
If $\ell\in S^<$, we get immediately that
\be\Eq(+contr.1)
\max\{G_\ell^+(\bx), G_\ell^-(\bx)\}\leq \sfrac{c}{\a}\e^2 ,
\ee
and thus, from \eqv(claim.4) and \eqv(claim.3),
\be\Eq(+contr.2)
L_\ell\phi(\bx)\leq \sfrac{c'}{\a}\e^2 \rho_\ell\phi(\bx) .
\ee
To control the r.h.s. of \eqv(G+.2) and \eqv(G-.2) when
$\ell\in S^>$,  set
\be\Eq(smarty)
y_\ell\equiv\min\left\{\b\left|g_\ell(\bx)\right|,\sfrac{1}{2}\right\}\leq \b
\left|g_\ell(\bx)\right|.
\ee
If  $g_\ell(\bx)> \sfrac{c\e}{\a\b}$,
then
\bea\Eq(smarty.1)
\exp{\left(2\b(1-\a)g_\ell(\bx)\right)}-1
&\geq&\exp{\left(2(1-\a)y_\ell\right)}-1
\geq 2(1-\a)y_\ell
\eea
and
\bea\Eq(smarty.2)
1- \exp{\left(2\b\a g_\ell(\bx)- c\e\right)}
&\leq& 1-\exp{(\a y_\ell)}
\leq -\a y_\ell,
\eea
 so that the product in the r.h.s. of \eqv(G+.2) is bounded from above by
$ -2(1-\a)\a y_\ell^2$.
On the other hand, if
 $g_\ell(\bx) < -\sfrac{c\e}{\a\b}$,
\bea\Eq(smarty.3)
\exp{\left(2\b(1-\a)g_\ell(\bx)\right)}-1
&\leq& \exp{\left(-2(1-\a)y_\ell\right)}-1
\leq - (1-\a)y_\ell
\eea
and
\bea\Eq(smarty.4)
1- \exp{\left(2\b\a g_\ell(\bx)+
c\e\right)}
&\geq& 1-\exp{(-\a y_\ell)}
\geq \sfrac{3}{4}\a y_\ell,
\eea
 and the product in the r.h.s. of \eqv(G+.2) is bounded from above by
$ -\sfrac{3}{4}(1-\a)\a y_\ell^2$.
Altogether, this proves that, for all  $\ell\in S^>$,
\be\Eq(Gbounds+)
G_\ell^+(\bx)\leq -\sfrac{3}{4}(1-\a)\a y_\ell^2,
\ee
and with a  similar computation, that
\be\Eq(Gbounds-)
G_\ell^-(\bx)\leq -\sfrac{3}{4}(1-\a)\a y_\ell^2.
\ee
If $\ell\in S^>$, then we have
\be\Eq(contr.3)
L_\ell\phi(\bx)\leq - c\a\rho_\ell y_\ell^2\phi(\bx).
\ee

It remains to control the case when
$\bx_\ell/\rho_\ell\in(-1,-1+\nu]\cup[1-\nu,1) $.
From  Lemma \thv(boundary-g.1) it follows that,
while  the positive contribution to $G_\ell^+(\bx)$ and $G_\ell^-(\bx)$
remains bounded by a constant, the negative contribution
becomes very large as soon as $\nu$ is small enough.
More explicitly, for all $\nu$ small enough, we have
\be\Eq(G^+.3)
\ba{l}
G_\ell^+(\bx)\leq - (\exp(\pm C')-1)^2 + \exp(\pm C')(\exp(2\b(1-\a)c)-1)\leq -(1+\po(1))\\
G_\ell^-(\bx)\leq - (1-\exp(\mp C'))^2 + \exp(\mp
C'')(\exp(2\b(1-\a)c)-1)\leq -(1+\po(1))
\ea
\ee
where $C'$ and $C''$  are positive constants tending to $+\infty $
as $\nu\downarrow 0$, and the sign $\pm$ is equal to the sign of
$\bx_\ell$.
Together with \eqv(claim.3) and \eqv(claim.4), we finally get
\be\Eq(boundarycase.5)
L_\ell\phi(\bx)\leq
- (1+\po(1))\rho_\ell \phi(\bx).
\ee

From \eqv(gb.1), \eqv(+contr.2), \eqv(contr.3) and \eqv(boundarycase.5),
it turns out that  the positive contribution
to the generator $L\phi(\bx)=\sum_{\ell=1}^n L_\ell\phi(\bx)$,
comes at most from the indexes $\ell\in S^<$, and can be estimated
by
\be\Eq(time.51)
\sfrac{c'}{\a}\e^2 \sum_{\ell\in S^<}\rho_\ell
\leq \sfrac{c'}{\a}\e^2.
\ee

Now we distinguish two cases according to whether the hypothesis of
Lemma \thv(grad) are satisfied or not.

\noindent\underline {\emph{Case 1:}}  $\sum_{\ell\not\in S(\bx)}
\rho_\ell>\frac {\e^2}{8\th}$. By \eqv(gb.1), we get 
\bea\Eq(not.1)
\sum_{\ell=1}^nL_\ell\phi(\bx)& \leq& \sum_{\ell\not\in  S(\bx)}
L_\ell\phi(\bx)
+\sum_{\ell \in S^<}
L_\ell\phi(\bx)\\\nonumber
&\leq& -\frac {\e^2}{8\th} (1+\po(1)) \phi(\bx)+
\sfrac{c'}{\a}\e^2,
\eea
which is negative as desired if $\th$ is small enough, that is, with
our choice, if $\e$ is small enough.

\noindent\underline {\emph{Case 2:}}  $\sum_{\ell\not\in S(\bx)}
\rho_\ell\leq  \frac {\e^2}{8\th}$. In this case, the assertion of Lemma
\thv(grad) holds.

By \eqv(gb.1), \eqv(contr.3), and \eqv(boundarycase.5), we have that,
for all $\ell\in S(\bx)\setminus L^{<} $,
\be\Eq(-contr.1)
L_\ell\phi(\bx)\leq - \rho_\ell\phi(\bx)
\min\{ c\a y_\ell^2,1\}\leq - c\a\rho_\ell y_\ell^2\phi(\bx),
\ee
where the last inequality holds for $\a<4/c$.
Let us write the generator as
\be\Eq(not.2)
L\phi(\bx)\leq  \sum_{\ell\in  S(\bx)\setminus S^<}
L_\ell\phi(\bx)
+\sum_{\ell \in S^<}
L_\ell\phi(\bx).
\ee

The first sum in \eqv(not.2)  is bounded from above by
\bea\Eq(time.51.1)
-c\a\phi(\bx)\sum_{\ell\in S(\bx)\setminus S^<}\rho_\ell y_\ell^2
&\leq& - c\a\phi(\bx)\sum_{\ell\in s(\bx)\setminus S^<}\rho_\ell
\min\left\{\b^2 g_\ell^2(\bx);\sfrac{1}{4} \right\}
\nonumber\\
&\leq& - c\a\phi(\bx)\min\left\{
\b^2\sum_{\ell\in S(\bx)\setminus S^<}
\rho_\ell g_\ell^2(\bx) ;\sfrac{1}{4} \right\}.\eea
But from Lemma \thv(grad), we know
that, for all $\bx\in\bU_\d\setminus\bG_\th$,
\be
\sum_{\ell\in S(\bx)\setminus S^<}\rho_\ell g_\ell^2 (\bx)\geq  c\frac{\e^2}{\th} -
\frac{c'}{\a^2}\e^2\geq
c''\frac{\e^2}{\th},
\ee
where $c''$ is a positive constant provided that
$\a\geq c \th$.
Taking $n$ large enough, it holds that
\be
\min\left\{
\b^2\sum_{\ell\in s(\bx)\setminus S^<}
\rho_\ell g_\ell^2(\bx) ;\sfrac{1}{4} \right\}\geq
\min\left\{c''\frac{\e^2}{\th}; \sfrac{1}{4}\right\}
=c''\frac{\e^2}{\th},
\ee
and then, from \eqv(time.51) and \eqv(time.51.1), we get
\be\Eq(time.52)
L \psi(\s)\leq- \e^2(1-\a)\phi(\bx)
(c''\a\th^{-1} -c'\a^{-1}).
\ee
By our choice of $\th$ and taking $n$ large enough,
 the condition $c''\a\th^{-1} -c'\a^{-1}>0\Leftrightarrow
\a>c\th$  is satisfied   for any $\a\in(0,1)$.
Hence,  for such $n$'s and for $N$ large enough, we get that
$L\psi(\s) =L \phi(\bx)\leq0$
concluding the proof of Proposition \thv(time.20).
\end{proof}

Substituting the expression of the super-harmonic
function \eqv(superm.5) in \eqv(superm.3), and together
with \eqv(time.11.1), we obtain that, for all
$\s\in \cS[\bU_\d\setminus\bG_\th]$,
\bea\Eq(last.1)
 \P_\s[\t_A<\t_{\cS[\partial_A \bG_\th]\cup B}]
&\leq&
\max_{\s'\in\cS[\partial_A\bU_\d]}e^{-(1-\a)\b N
\left[F_{\b,N}(\bm(\s'))- F_{\b,N}(\bm(\s))\right]}
\nonumber\\
&\leq&
e^{-(1-\a)\b N
\left[F_{\b,N}(m_0^*)+\d- F_{\b,N}(\bm(\s))\right]},
\eea
where the last inequality follows from the definition of $\bU_\d$
together with the bounds in \eqv(fine.507).
This concludes the proof of Proposition \thv(claim.01).

\paragraph{Renewal estimates on escape probabilities.}
Let us now come back to the proof of
Lemma \thv(time.1.2).
An easy consequence of Eq. \eqv(claim.02)
is that, for all $\s\in\cS[\partial_A\bG_\th]$,
\be\Eq(escape.2)
\P_\s(\t_A<\t_{\cS[\partial_A\bG_\th]\cup B})
\leq
e^{-(1-\a)\b N \left(F_{\b,N}(m_0^*)
+\d\right)} \max_{\bm\in\partial_A\bG_\th}e^{(1-\a)
\b NF_{\b,N}(\bm)} ,
\ee
while obviously $\P_\s(\t_A<\t_{\cS[\partial_A\bG_\th]\cup B})
\equiv 0$ for all $\s\in\cS[\bG_\th\setminus \partial_A\bG_\th]$.
 To control the r.h.s. of \eqv(escape.2), we need the following
 lemma:
 \begin{lemma}\label{lemma_remark}
There exists a constant $c<\infty$, independent of $n$, such that,
for all $\bm\in\bG_\th$,
\be\Eq(remark)
F_{\b,N}(\bm)\leq F_{\b,N}(\bm^*)+ c \e.
\ee
\end{lemma}

\begin{proof}
Fix $\bm\in\bG_\th$ and set $\bm-\bm^*\equiv \bv$.
Notice that, from the definition of $\bG_\th$,
\be\Eq(vnorm)
\|\bv\|_2^2\leq \max_{\ell} \rho_\ell \sum_{\ell=1}^n
\frac{(\bm_\ell-\bm_\ell^*)^2}{\rho_\ell} \leq \e^2.
\ee
 Using  Taylor's formula, we have
\be\Eq(remark.1)
F_{\b,N}(\bm)=F_{\b,N}(\bm^*)+\frac{1}{2}
\left(\bv, \A (\bm^*)\bv\right) + \frac{1}{6}D^3 F_{\b,N}(\bx) \bv^3 ,
\ee
where $\A (\bm^*)$ is the positive-definite matrix described in
Sect. \ref{sect:nearcritical} (see Eq. \eqv(fine.6)) and $\bx$
is a suitable element of the ball around $\bm^*$.
From the explicit representation of the eigenvalues of $\A(\bm^*)$,
we see that $\|\A(\bm^*)\|\leq c \e^{-1}$, and hence
\be\Eq(remark.1.1)
\left(\bv, \A (\bm^*)\bv\right)
\leq c \e^{-1}\|\bv\|_2^2\leq c \e.
\ee
The remainder is given in explicit form as
\bea\Eq(remark.1.2)
D^3 F_{\b,N}(\bx) \bv^3 &=&
\sum_{\ell=1}^n \frac{\del^3 F_{\b,N}}{\del \bx_\ell^3}(\bx)\bv_\ell^3
= \frac{1}{\b}\sum_{\ell=1}^n \frac{1}{\rho_\ell^2}
I_{N,\ell}'''(\bx_{\ell}/\rho_\ell)\bv_\ell^3
\\
&=& - \frac{1}{\b}\sum_{\ell=1}^n \frac{1}{\rho_\ell^2}
\frac{U_{N,\ell}'''(t_\ell)}
{\left(U_{N,\ell}''(t_\ell)\right)^3}\bv_\ell^3
\nonumber\\
&=& - \frac{1}{\b}\sum_{\ell=1}^n \frac{1}{\rho_\ell^2}
\frac{|\L_\ell|^{-1}\sum_{i\in\L_\ell}\tanh(t_\ell +\b \tilde h_i)
(1-\tanh^2(t_\ell +\b \tilde h_i))}
{\left(|\L_\ell|^{-1}\sum_{i\in\L_\ell}(1-\tanh^2
(t_\ell +\b \tilde h_i))\right)^3}\bv_\ell^3,
\nonumber
\eea
where $t_\ell=I_{N,\ell}'(\bx_{\ell}/\rho_\ell)$. Thus
\be\Eq(remark.1.3)
\left|D^3 F_{\b,N}(\bx) \bv^3\right| \leq c \sum_{\ell=1}^n
\frac{1}{\rho_\ell^2}\bv_\ell^3\leq c' \e^{-1} \|\bv\|_2^2\leq c'\e,
\ee
where we used that $|\bv_\ell/\rho_\ell |\leq 1$.
Hence, for some $c<\infty$, independent of $n$,
\bea\Eq(remark.3)
F_{\b,N}(\bm)&\leq& F_{\b,N}(\bm^*)+ c\e
\eea
which proves the lemma. \end{proof}

Inserting the result of Lemma \ref{lemma_remark} into \eqv(escape.2),
and recalling that $F_{\b,N}(\bm^*)=F_{\b,N}(m^*)$,
we get that for all $\s\in\cS[\partial_A \bG_\th]$
\be\Eq(escape.3)
\P_\s(\t_A<\t_{\cS[\partial_A \bG_\th]\cup B})
\leq e^{-(1-\a)\b N \left(F_{\b,N}(m_0^*)
+\d -F_{\b,N}(m^*)-c\e\right)}  .
\ee

 The last needed ingredient in order to get a suitable estimate
 on $\P_\s(\t_A<\t_B)$, is stated in the following lemma.

 \begin{lemma}\TH(lemma.claim.1)
For any $\d_2>0$, there exists $n_0\in\N$, such that,
for all $n\geq n_0$,
for all $\s\in \cS[\partial_A \bG_\th]$, and for all $N$ large enough,
\be\Eq(claim.04)
\P_\s(\t_B< \t_{\cS[\partial_A \bG_\th]})\geq e^{- N\b\d_2 }.
\ee
\end{lemma}

\begin{proof}
Fix $\s\in\cS[\partial_A \bG_\th]$ and set $\bm(0)\equiv\bm(\s)$.
As pointed out in the proof of Lemma \thv(lemma_remark),
every $\bm(0)\in\partial_A \bG_\th$ can be written
in the form $\bm(0)=\bm^* +\bv$,
with $\bv\in\G_N^n$ such that $\|\bv\|_2\leq \e$.
Then, let
$\underline{\bm}=(\bm(0),\bm(1),\ldots,\bm(\|\bv\|_1N)\equiv \bm^*)$
be a nearest neighbor path in $\G_N^n$ from $\bm(0)$ to $\bm^*$, of
 length $N\|\bv\|_1$, with
the following property: Denoting by $\ell_t$
the unique index in $\{1,\ldots, n\}$ such that
$\bm_{\ell_t}(t)\ne\bm_{\ell_t}(t-1)$, it holds that
\be
\Eq(propertypath)
\bm_{\ell_t}(t)=\bm_{\ell_t}(t-1) + \sfrac{2}{N}
s_t,\quad \forall t\geq1,
\ee
where we define 
\be\Eq(soft)
s_t\equiv \sign\left(\bm^*_{\ell_t}-\bm_{\ell_t}(t-1)\right).
\ee
Note that, by property \eqv(propertypath),
$\bm(t)\in\bG_\th$ for all $t\geq 0$.
Thus, all microscopic paths, $(\s(t))_{t\geq 0}$,
 such that $\s(0)=\s$ and $\bm(\s(t))=\bm(t)$,
 for all $t\geq 1$,
 are contained in the event 
 $\{\t_B< \t_{\cS[\partial_A \bG_\th]}\}$. Thus we get that 

\bea\Eq(minpath.1)
\P_\s(\t_B< \t_{\cS[\partial_A \bG_\th]})
&\geq&
\P_\s (\bm(\s(t))=\bm(t), \forall t=1,\ldots,\|\bv\|_1N)
\nonumber\\
&=&
\prod_{t=1}^{\|\bv\|_1N}\P_\s(\bm(\s(t))=\bm(t)\big{|}\bm(\s(t-1))=\bm(t-1))
\nonumber\\
&=&
\prod_{t=1}^{\|\bv\|_1N}\sum_{i\in\L_{\ell_t}^{s_t} }p_N(\s(t-1),\s^i(t-1)).
\eea

Note that  $\L^{s_t}_{\ell_t}$ is the  set of sites in which a spin-flip
corresponds to a step from $\bm(t-1)$ to $\bm(t)$.

The sum of the probabilities in the r.h.s. of \eqv(minpath.1)
corresponds to the quantity $\P^{\s(t-1)}_{s_t,\ell_t}$ defined in
\eqv(time.50). From the inequalities
\eqv(claim.2) 
and \eqv(bounded.2), it follows that, for some constant $c>0$
depending on $\b$ and on the distribution of the field,
\be\Eq(ratespath.1)
\P^{\s(t-1)}_{s_t,\ell_t}\geq
c |\L_{\ell_t}^{s_t}(\bm(t-1))|/N\geq
c |\L_{\ell_t}^{s_t}(\bm^*)|/N,
\ee
where the second inequality follows by our choice of the path
$\underline{\bm}$.
Now, since $|\L_{\ell}^\pm(\bm^*)|/N=
\sfrac{1}{2}\left(\rho_{\ell}\pm\bm^*_{\ell}\right)$,
using the expression \eqv(fine.9) for $\bm^*_{\ell_t}$
and continuing from \eqv(ratespath.1), we obtain
\be\Eq(ratespath.2)
\P^{\s(t-1)}_{s_t,\ell_t}\geq c' \rho_{\ell_t}.
\ee
Inserting the last inequality in \eqv(minpath.1), and using that,
by definition of the path $\underline{\bm}$,
the number of steps corresponding to a spin-flip in $\L_\ell$
is equal to $|\bv_\ell|N$, for all $\ell=\{1,\ldots,n\}$ ,
we get
\bea\Eq(minpath.2)
\P_\s(\t_B< \t_{\cS[\partial_A \bG_\th]})
&\geq&
\prod_{t=1}^{\|\bv\|_1N}c' \rho_{\ell_t}
\nonumber\\
&=&
e^{\|\bv\|_1N\ln(c')}
\prod_{\ell=1}^n \rho_\ell^{|\bv_\ell|N}
\nonumber\\
&\geq&
e^{N\sqrt\e \ln(c')}
e^{-N \sum_{\ell=1}^n \bv_\ell \ln\left(1/\rho_\ell\right)}
\nonumber\\
&\geq&
e^{ N\sqrt\e \ln(c')}
e^{-N \sum_{\ell=1}^n \bv_\ell/\sqrt\rho_\ell}
\nonumber\\
&\geq&
e^{N\e\ln(c')}
e^{-N \left(\sum_{\ell=1}^n \bv_\ell^2/\rho_\ell\right)^{1/2}\e^{-1/2}}
\nonumber\\
&\geq& e^{-N \left(\sqrt{\sfrac{\e}{\th}}- \sqrt\e \ln(c')\right)},
\eea
where in the third line we used  the inequality
$\|\bv\|_1\leq \e^{-1/2}\|\bv\|_2 \leq \sqrt\e$, and
in the last line we used that $\bm(0)= \bm^* + \bv \in\bG_\th$.
By our choice of $\th\gg\e$, there exists $n_0\in\N$
such that, for all $n\geq n_0$,
$\sqrt{\sfrac{\e}{\th}}- \sqrt\e \ln(c')\leq \b\d_2$.
For such $n$'s, inequality \eqv(minpath.2)
yields the bound \eqv(claim.04)
and concludes the proof of the Lemma.
\end{proof}
We finally state  the following proposition:
\begin{proposition}\TH(prop.claim.1)
For all $\s\in\cS[\bU_\d]$ it holds that
\be\Eq(prop.claim.2)
\P_\s(\t_A<\t_B) \leq
e^{-\b N \left((1-\a)\left(F_{\b,N}(m_0^*)
+\d -F_{\b,N}(m^*)-c\e\right) -\d_2\right)}(1+\po(1))
\ee
\end{proposition}
\begin{proof}
Let us first consider a configuration $\s\in\cS[\partial_A\bG_\th]$.
Then it holds
\bea\Eq(renewal.1)
\P_\s(\t_A<\t_B)&\leq&
\P_\s(\t_A<\t_{\cS[\partial_A\bG_\th]\cup B}) +
\sum_{\eta\in\cS[\partial_A\bG_\th]}
\P_\s(\t_A<\t_B,\,\t_{\eta}\leq\
\t_{\cS[\partial_A\bG_\th]\cup A\cup B})\quad\quad
\nonumber\\
&\leq&
\P_\s(\t_A<\t_{\cS[\partial_A\bG_\th]\cup B}) +
\max_{\eta\in\cS[\partial_A\bG_\th]}
\P_\eta(\t_A<\t_B) \P_\s(\t_{\cS[\partial_A\bG_\th]}<\t_{B})
\nonumber\\
&\leq& \P_\s(\t_A<\t_{\cS[\partial_A\bG_\th]\cup B}) +
\max_{\eta\in\cS[\partial_A\bG_\th]}
\P_\eta(\t_A<\t_B) \left(1-e^{-\b N\d_2 }\right),
\nonumber\\
\eea
where in the second line we applied the Markov property,
and in the last line
we insert the result  \eqv(prop.claim.1).
Taking the maximum over $\s\in\cS[\partial_A\bG_\th] $
on both sides of \eqv(renewal.1),
and rearranging the summation,  we get
\bea\Eq(renewal.2)
\max_{\s\in\cS[\partial_A\bG_\th]} \P_{\s}(\t_A<\t_B)
&\leq& \max_{\s\in\cS[\partial_A\bG_\th\cup B]}
\P_\s(\t_A<\t_{\cS[\partial_A\bG_\th]})e^{\b N\d_2}
\nonumber\\
&\leq&
e^{-\b N \left((1-\a)\left(F_{\b,N}(m_0^*)
+\d -F_{\b,N}(m^*)-c\e\right) -\d_2\right)},
\eea
where in the last line we used the bound \eqv(escape.3).
This concludes the proof of \eqv(prop.claim.2)
for $\s\in\cS[\partial_A\bG_\th]$.

Then, let us consider $\s\in\cS[\bU_\d\setminus\partial_A\bG_\th]$.
As before, it holds
\bea\Eq(renewal.3)
\P_\s(\t_A<\t_B)&\leq&
\P_\s(\t_A<\t_{\cS[\partial_A\bG_\th]\cup B}) +
\sum_{\eta\in\cS[\partial_A\bG_\th]}
\P_\s(\t_A<\t_B,\,\t_{\eta}\leq
\t_{\cS[\partial_A\bG_\th]\cup A\cup B})
\nonumber\\
&\leq&
\P_\s(\t_A<\t_{\cS[\partial_A\bG_\th]\cup B})
+ \max_{\eta\in\cS[\partial_A\bG_\th]}
\P_\eta(\t_A<\t_B) \P_\s(\t_{\cS[\partial_A\bG_\th]}<\t_{B})
\nonumber\\
&\leq& \P_\s(\t_A<\t_{\cS[\partial_A\bG_\th]\cup B}) +
\max_{\eta\in\cS[\partial_A\bG_\th]}
\P_\eta(\t_A<\t_B) ,
\eea
where $\P_\s(\t_A<\t_{\cS[\partial_a\bG_\th]\cup B})$ is $0$ for all
$\s\in\cS[\bG_\th\setminus\partial_A\bG_\th]$,
and exponentially small in $N$ for all $\s\in\cS[\bU_\d\setminus\bG_\th]$
(due to Proposition \thv(claim.01)).
Inserting the bound \eqv(renewal.2) in the last equation, provides Eq.
\eqv(prop.claim.2) for $\s\in\cS[\bU_\d\setminus\partial_A\bG_\th]$ and 
concludes the proof.
\end{proof}

The proof of formula \eqv(time.2) now follows straightforwardly.
From  \eqv(prop.claim.2), we get
\bea\Eq(upper3)
&&\hspace{-1cm}\sum_{\s\in\cS[\cU_\d(m^*)]} \mu_{\b,N}(\s) \P_\s(\t_A<\t_B)
\nonumber\\
&&\quad\quad\quad\leq e^{-\b N
\left[(1-\a)\left(F_{\b,N}(m_0^*) +\d- F_{\b,N}(m^*)-c\e\right) -
\d_2
\right]}\sum_{\bm\in \bU_\d}\QQ_{\b,N}(\bm)
 \nonumber\\
&&\quad\quad\quad=\QQ_{\b,N}(m_0^*)e^{\b N \left[\a F_{\b,N}(m_0^*)
-(1-\a)(\d-F_{\b,N}(m^*)-c\e)+\d_2\right]}
\sum_{\bm\in\bU_\d} e^{-\b NF_{\b,N}(\bm)}
\nonumber\\
&&\quad\quad\quad\leq \QQ_{\b,N}(m_0^*) N^n e^{\b N\left[\a\left(F_{\b,N}
(m_0^*)- F_{\b,N}(m^*)\right)-(1-\a)\left(\d-c \e
\right)+\d_2\right]},
\eea
where in the second inequality  we used the expression \eqv(static.11.1)
for $\QQ_{\b,N}(m_0^*)$, while in the last line we applied the bound
$F_{\b,N}(\bm)\leq F_{\b,N}(\bm^*)= F_{\b,N}(m^*)$,
and then bounded the cardinality of $\bU_\d$ by $N^n$.
Finally, choosing $\a$ small enough, namely
\be\Eq(me.1)\a < \frac{\d-c\e-\d_2}{F_{\b,N}(m_0^*)-
F_{\b,N}(m^*) +\d -c \e},
\ee
  we can easily
 ensure that \eqv(upper3) implies \eqv(time.2).

In exactly the same way one proves \eqv(time.3).
This concludes the proof of Lemma \thv(time.1.2) and thus
of  Theorem \thv(theorem1).


 \end{document}